\documentclass[10pt]{amsart}
\usepackage{amsmath,amsthm,amssymb,amsfonts, eucal, amscd, mathbbol, mathrsfs}
\usepackage[shortlabels]{enumitem}
\usepackage{cite}
\usepackage{setspace}
\usepackage[all,cmtip]{xy}
\usepackage{graphicx}
\usepackage{subfig}
\usepackage{fancyhdr}
\usepackage{latexsym}
\usepackage{fncylab}
\usepackage{epic}
\usepackage{ifthen}
\usepackage[bookmarks,bookmarksnumbered, plainpages=false, pdfpagelabels]{hyperref}
\usepackage{xcolor}
\hypersetup{
  colorlinks   = true, 
  urlcolor     = green, 
  linkcolor    = blue, 
  citecolor   = red 
}

\usepackage{rotating}
\usepackage{scalefnt}
\usepackage{enumitem}
\setlist{nolistsep}

\newtheorem{thm}{Theorem}[section]
\newtheorem*{thm*}{Main Theorem}
\newtheorem*{thm**}{Theorem}
\newtheorem{cor}[thm]{Corollary}
\newtheorem{lem}[thm]{Lemma}
\newtheorem{prop}[thm]{Proposition}
\theoremstyle{definition}
\newtheorem{defn}[thm]{Definition}
\theoremstyle{definition}

\theoremstyle{definition}
\newtheorem{remark}[thm]{Remark}
\theoremstyle{definition}

\theoremstyle{definition}

\theoremstyle{definition}

\theoremstyle{definition}

\numberwithin{thm}{subsection}

\newcommand{\R}{\ensuremath{\mathbb{R}}}
\newcommand{\N}{\ensuremath{\mathbb{N}}} 
\newcommand{\C}{\ensuremath{\mathcal{C}}} 
\newcommand{\Q}{\ensuremath{\mathcal{Q}}} 
\newcommand{\F}{\ensuremath{\mathbb{F}}}

\newcommand{\<}{\langle} 
\renewcommand{\>}{\rangle} 
\newcommand{\Z}{{\Bbb Z}} 

\def\p{\partial}
\def\i{\infty}

\def\det{\it{det}} 
\def\supp{\it{supp}}

\def\diam{\emph{diam}} 
\def\fr{\emph{fr}}
\def\vol{\emph{vol}}

\def\cal{\mathcal}

\def\a{\alpha}
\def\b{\beta} 
\def\g{\gamma} 
\def\G{\Gamma}
\def\d{\delta}
\def\D{\Delta} 
\def\e{\epsilon}
\def\k{\kappa} 
\def\l{\lambda} 
\def\L{\Lambda} 
\def\s{\sigma} 

\def\o{\omega} 
\def\O{\Omega} 
 
\def\B{\mathcal{B}}
\def\P{\mathcal{P}}
\def\A{\mathcal{A}}
\def\H{\mathcal{H}}
\def\T{\mathcal{T}}
\def\F{\mathcal{F}}
\def\M{\mathcal{M}}
\def\V{\mathcal{V}}

\def\hB{\hat{\mathcal{B}}}

\def\Lie{\mathcal{L}} 
\def\sp{\mathcal{S}}

\def\tM{\widetilde{M}}

\def\Span{\mathfrak{S}}

\def\rotateminus{\reflectbox{\rotatebox[origin=c]{155}{\hspace{.6pt}-}}}

\def\Xint#1{\mathchoice
{\XXint\displaystyle\textstyle{#1}}%
{\XXint\textstyle\scriptstyle{#1}}%
{\XXint\scriptstyle\scriptscriptstyle{#1}}
{\XXint\scriptscriptstyle\scriptscriptstyle{#1}}%
\!\int}
\def\XXint#1#2#3{{\setbox0=\hbox{$#1{#2#3}{\int}$}
\vcenter{\hbox{$#2#3$}}\kern-.5\wd0}}

\def\cint{\Xint\rotateminus }

\renewcommand{\labelenumi}{(\alph{enumi})}
\renewcommand{\theenumi}{(\alph{enumi})}
\renewcommand{\labelenumi}{\theenumi}
\makeatletter
\renewcommand{\p@enumii}{\theenumi}
\makeatother                                                

\begin{document} 
	\author[J. Harrison \& H. Pugh]{J. Harrison \\Department of Mathematics \\University of California, Berkeley  \\ H. Pugh\\ Mathematics Department \\ Stony Brook University} 
	\title[Plateau's Problem]{Existence and Soap Film Regularity of Solutions to Plateau's Problem}

\begin{abstract}  
	Plateau's problem is to find a surface with minimal area spanning a given boundary. Our paper presents a theorem for codimension one surfaces in \( \R^n \) in which the usual homological definition of span is replaced with a novel algebraic-topological notion. In particular, our new definition offers a significant improvement over existing homological definitions in the case that the boundary has multiple connected components.  Let \( M \) be a connected, oriented compact manifold of dimension \( n-2 \) and \( \Span \)  the collection of compact sets spanning \( M \). Using Hausdorff spherical measure as a notion of ``size,'' we prove:

 	\begin{thm**}
		There exists an \( X_0 \) in \( \Span \) with smallest size. Any such \( X_0 \) contains a ``core'' \( X_0^*\in \Span \) with the following properties: It is a subset of the convex hull of \( M \) and is a.e. (in the sense of \( (n-1) \)-dimensional Hausdorff measure) a real analytic \( (n-1) \)-dimensional minimal submanifold. If \( n=3 \), then \( X_0^* \) has the local structure of a soap film.  Furthermore, set theoretic solutions are elevated to current solutions in a space with a rich continuous operator algebra.
	\end{thm**}

\end{abstract}
 
\maketitle

\section*{Introduction}
	\label{sec:introduction}
	One of the classical problems in the Calculus of Variations is to prove the existence of a surface of least area spanning a given Jordan curve in Euclidian space. The problem was first posed by Lagrange \cite{lagrange} in 1760, and Lebesgue \cite{lebesgue} called it ``Plateau's Problem'' after Joseph Plateau \cite{plateauoriginal} who experimented with films of oil and wire frames.

	\subsection*{History}
		\label{par:history}
		In 1930, Douglas \cite{douglas} solved the problem for surfaces which arise as the image of a disk, via minimization of an energy functional\footnote{See \cite{rado} for an extensive report of work before 1930.}. His method was extended by Douglas and Courant \cite{courant} to surfaces of higher topological type bounding disjoint systems of Jordan curves. In 1960, Federer and Fleming   \cite{federerfleming} used integral currents with a given algebraic boundary to model films, and minimized current mass (see \S\ref{sub:mass_norm}) instead of energy or area of a surface. In dimension 7 and lower, their solutions turned out to be oriented submanifolds. 

		At the same time, Reifenberg \cite{reifenberg} approached the problem from an entirely different direction. For each \( M \) Reifenberg chose a subgroup \( L \) of the \v{C}ech homology of \( M \) with coefficients in a compact abelian group (e.g., \( \Z/2\Z \), \( \Z/3\Z \), or \( \R/\Z \).) A set \( X\supset M \) was said to be a surface with boundary \( \supseteq L \) if \( L \) was in the kernel of the homomorphism of homology induced by the inclusion \( M\hookrightarrow X \). He concluded there was a surface with boundary \( L \), depending on these two choices, that had smallest Hausdorff spherical measure.  
		
		Reifenberg's approach yielded a case-by-case analysis, and two cases have considerable historical interest. Using \( \Z/2\Z \) coefficients, he found a minimizer in a collection containing all orientable and non-orientable surfaces in \( \R^3 \) with a given boundary, excluding from consideration non-manifold surfaces with triple junctions or other such singularities. This theorem improved that of Douglas and Courant, because it considered surfaces of arbitrary genus simultaneously, and solved the non-orientable problem. His second special case used \( \R/\Z \) coefficients to deal with more complicated spanning sets, specifically those compact sets \( X\supset M \) with no retraction \( X\to M \). This collection contained surfaces with triple junctions and other singularities, but his theorem was limited in that the boundary was required to be a single closed curve (or in higher dimensions, a topological sphere.) For more general boundaries, he gave no single unifying result as he had for boundaries which were single topological spheres. For example, consider the disjoint union of a disk and a circle in \( \R^3 \). There is no retraction to the pair of circles, yet we would not want to consider this as an admissible spanning set. As another example, consider the surfaces \( X_i \), \( i=1,2,3 \), in Figure \ref{fig:catenoid}. Any one could be a surface with minimal area, depending on the distance between the circles, but a simple computation (see Proposition \ref{prop:noreifenberg}) shows there is no non-trivial collection of Reifenberg surfaces which contains all three simultaneously. Thus, one would have to find an appropriate group of coefficients and subgroup of homology which would produce the correct minimizer, and this task would change depending on the configuration of the circles in the ambient space. 
	
		Our goal was therefore to find a ``unifying solution'' similar to Reifenberg's \( \R/\Z \) solution for spherical boundaries, that simultaneously dealt with a large collection of reasonable surfaces, including those with soap-film singularities, spanning a boundary more general than a topological sphere.
 
		\begin{figure}[htbp]
			\centering
			\includegraphics[height=2in]{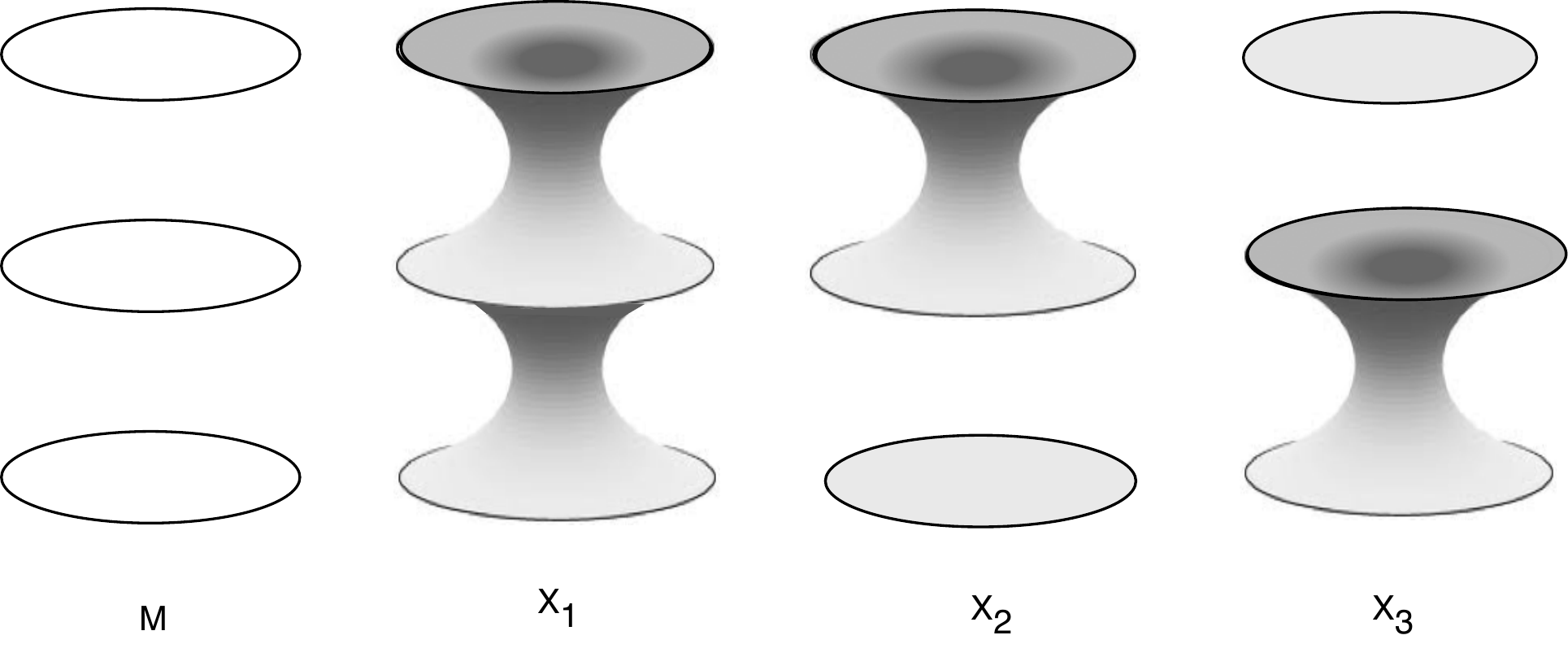}
			\caption{ \( M \) is the disjoint union of three circles. \( X_1 \) is homeomorphic to a cylinder, and \( X_2 \) and \( X_3 \) are both homeomorphic to a cylinder union a disc. All three sets span \( M \) and are thus in our collection. Each is a Reifenberg surface for some collection, but no non-trivial collection of Reifenberg surfaces spanning \( M \) contains \( X_1, X_2, \) and \( X_3 \) simultaneously (see Proposition \ref{prop:noreifenberg}.)
}
			\label{fig:catenoid}
		\end{figure}
						
		In \cite{almgrenmimeo} Almgren introduced varifolds as a new approach to solving Plateau's problem.  He announced a compactness theorem for stationary varifolds in \cite{varifolds} (see also \cite{almgrenmimeo} and \cite{allard}.)  This result produced a stationary varifold with minimal mass, and thus solved a version of Plateau's problem. However, solutions produced do not minimize mass over a given collection of surfaces, stationary and otherwise, only those that are a priori stationary. For example, if \( M \subset \R^3 \) is the unit circle in the \( xy \)-plane, the only competitors for area minimization are positive integer multiples of the unit disk. To greater effect, one could use his stationary varifold compactness theorem to minimize over mass- or size-minimizing solutions found using other techniques, since those solutions ought to give rise to stationary varifolds.  However, the limiting surface is only guaranteed to be a stationary varifold and may not necessarily span the boundary in any algebraic sense.  Related to this is the longstanding open problem to establish almost everywhere regularity of stationary varifolds of dimension at least two (see, e.g.,  \cite{bombieri}, \cite{delellis}.)

		 Almgren used \v{C}ech homology to define spanning sets in \cite{almgrenannals}, much as in \cite{reifenberg}, and made significant advances by minimizing size with respect to elliptic functionals and extending the coefficient groups to finitely generated abelian groups. His proof of both existence and a.e. regularity built upon a number of ideas in \cite{reifenberg}. For \( \Z \) coefficients, his a.e. regularity result for elliptic functionals \( F \) extended the celebrated result of De Giorgi \cite{degiorgi} where \( F = 1 \), while De Giorgi's results were influenced by methods developed by Caccioppoli \cite{caccio}.

		Differential chains (a type of current, see below) have a compactness theorem and a cut-and-paste procedure. A particular subtype, ``film chains,'' models all soap films and solves the boundary problem for triple junctions and M\"obius strips. These are not rectifiable currents, although size minimizing film chains have rectifiable supports. Furthermore, film chains avoid the problem of high multiplicity seen in mass minimizing problems as discussed in \cite{pauwsize}.

	\subsection*{A new definition of span using linking numbers}
		\label{par:_a_new_definitionstandard}  
		
		 Let \( M \) be an \( (n-2) \)-dimensional compact orientable submanifold of \( \R^n \), \( n\geq 2 \). We say that a circle \( S \) embedded in \( \R^n\setminus M \) is a \emph{\textbf{simple link of \( M \)}} if the absolute value of the linking number\footnote{Given some choice of orientation of \( M \) and \( S \), the definition being independent of this choice} \( L(S,M_i) \) of \( S \) with one of the connected components \( M_i \) of \( M \) is equal to one, and \( L(S,M_j)=0 \) for the other connected components \( M_j \) of \( M \), \( j\neq i \). We say that a compact subset \( X\subset \R^n \) \emph{\textbf{spans}} \( M \) if every simple link of \( M \) intersects \( X \).

		If \( M \) is a sphere, \( X\supset M \) and there is no retraction \( X\to M \), then \( X \) spans \( M \) (see Proposition \ref{prop:equivalences}.) If \( X \) is a manifold with boundary \( M \), then \( X \) spans \( M \). The sets \( X_1, X_2, X_3 \) in Figure \ref{fig:catenoid} span \( M \). See \S\ref{sec:spanning_sets} for more properties. One can drop the condition that \( M \) be a manifold by specifying particular \( (n-2) \)-cycles for which \( S \) to link. For example, \( M \) can be a frame such as the \( (n-2) \)-skeleton of an \( n \)-cube.

	\subsection*{Differential chains}
		\label{par:differential_chains}

		Our approach to Plateau's problem begins with a topological vector space of de Rham currents called ``differential chains\footnote{There is now a growing theory of differential chains with a number of applications unrelated to Plateau's problem in progress.}'' introduced in \cite{harrison1}, \cite{ravello}, \cite{OC2011}, and \cite{OC}, and developed by the authors in \cite{thesis} and \cite{topological}. The space of differential chains on a Riemannian manifold \( M \) is a differential graded topological vector space, defined as an inductive limit of a sequence of Banach spaces, beginning with the space of Whitney sharp chains, and such that the boundary operator takes one space to the next. These Banach spaces \( \B_k^r \) are completions of the space of finitely supported sections of the \( k \)-th exterior power of the tangent bundle of \( M \)\footnote{One may think of these as infinitesimal polyhedral chains.}, equipped with a decreasing sequence of geometrically defined norms \( \|\cdot\|_{B^r} \). The linking maps are inclusions, and the dual Banach spaces consist of increasingly differentiable bounded differential \( k \)-forms.
		
		If \( M \) is compact, the linking maps between these Banach spaces are compact, and this allows us to formulate compactness theorems which improve on the usual weak compactness theorems of currents. In other words, control of the \( B^r \)-norm gives relative compactness in \( \B^{r+1} \), and the topology on this space is strictly finer than the inductive limit topology on differential chains, which is again strictly finer than the weak topology on currents.
		
		A compact \( k \)-submanifold with boundary \( A \) is uniquely represented by an element \( \tilde{A} \) in the space \( \B_k^1 \) (the space of sharp \( k \)-chains,) and \( \p \tilde{A} \in \B_{k-1}^2 \) represents the manifold's boundary. However, \( \p \tilde{A} \) is again a sharp chain, so it is contained in the subspace \( B_{k-1}^1\subset B_{k-1}^2 \). The full space \( B_k^2 \) comes into play if we instead thicken \( \tilde{A} \) with the ``extrusion'' operator, dual to interior product, then apply boundary: the resulting current can be thought of as a ``dipole'' surface, and it is not a sharp chain. For example, consider a \( 2 \)-cell \( \s \) in \( \R^3 \). Its dipole version can be found as the limit \( \lim_{t \to 0} T_{tv}(\s)/t - \s/t \), where \( v \) is a unit vector normal to \( \s \), and \( T_{tv} \) is translation through the vector \( tv \). These dipole surfaces (see Figures \ref{fig:Figures_YProblem2} and \ref{fig:moebiusdipole}) which were first introduced in \cite{plateau10} play an important role in our solution to Plateau's problem.	
 	  
	\subsection*{Soap films and film chains}
		\label{par:soap_films_vs_film_chains}
		Away from triple junctions, actual soap films consist of three very thin layers: two layers consisting of soap on the outside and a layer of water on the inside \cite{bubbles}. Dipole surfaces \cite{plateau10}, generalized to ``film chains'' in this paper (see Definition \ref{def:filmchain},) nicely model the outer soapy surfaces. The algebraic boundary of a dipole surface  \( S \) spanning a closed loop is a ``dipole curve'' whose support is the loop, consistent with the way in which the two layers of soap might meet a wire. Application of a cone operator (see Definition \ref{def:cone}) to \( S \) produces a model of the layer of water inside. Application of the geometric Hodge star operator of \cite{OC} to a dipole surface \( S \) yields a current that can be visualized as oppositely oriented ``dipole normal vectors'' to \( S \). This models the hydrophobic/hydrophilic polarization of the two soap layers, as also seen in lipid bilayers of cells. Our dipole models bypass completely the problem of triple junctions contributing to the algebraic boundary (see Figure \ref{fig:Figures_YProblem2}.) Other branching structures similar to soap films include capillaries, lightning, highways, and fractures. Dipole surfaces and film chains can be defined in arbitrary dimension and codimension and are well suited to other minimization and maximization problems involving these structures. Note that while these currents are not rectifiable since they do not have finite mass, we do not use mass to measure area.

       \begin{figure}[htbp]
       	\centering
       		\includegraphics[height=3in]{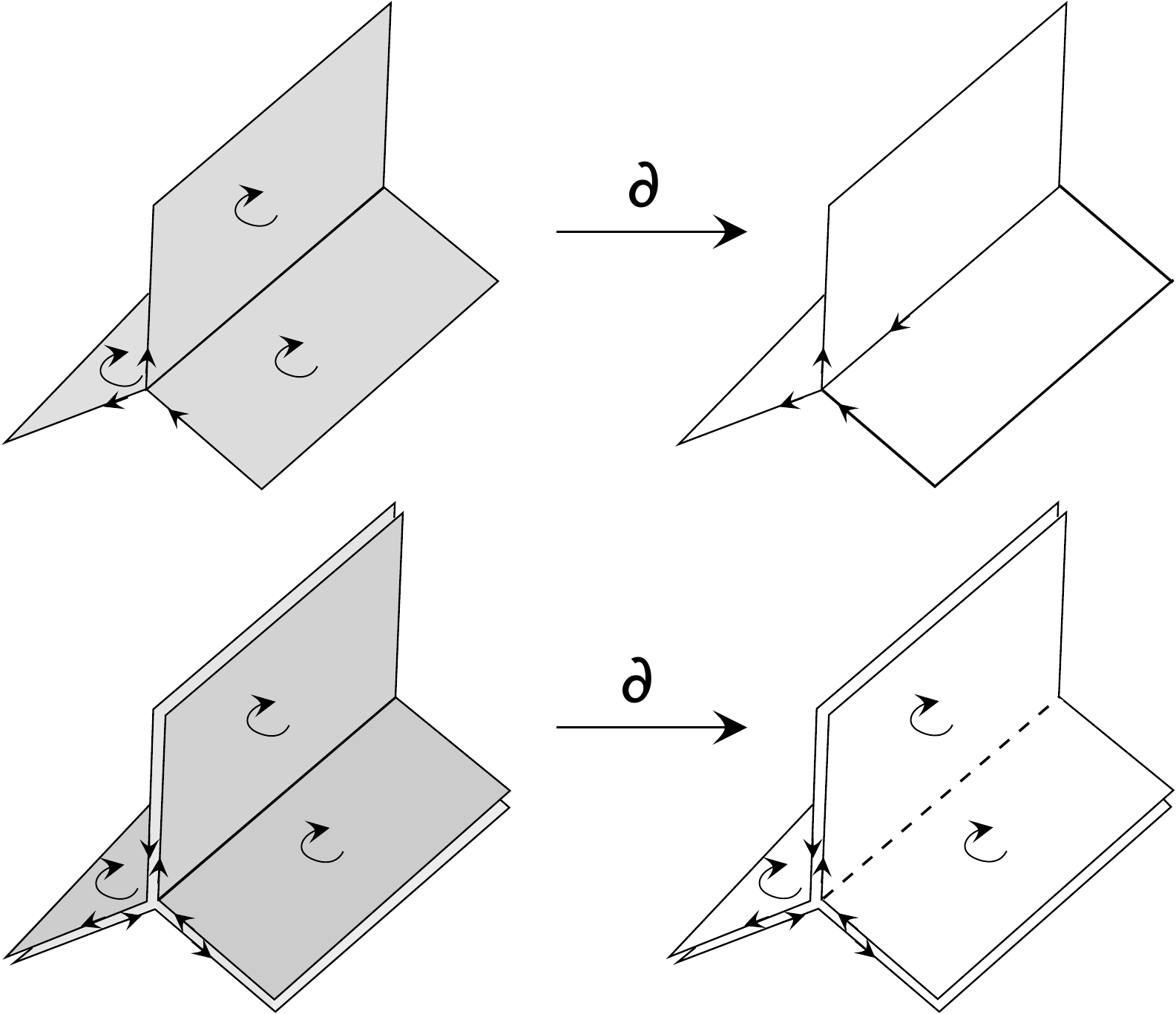}
       	\caption{A dipole surface with a triple junction and its dipole boundary. The triple junction is not part of the boundary. The two layers of a dipole surface have opposite orientation, and yet do not cancel. If one glues the two ``Y'' ends together with a twist, the algebraic boundary of the resulting dipole surface is a single dipole curve, and the support of this dipole surface is the triple m\"{o}bius band. The triple m\"{o}bius band, when scaled properly, minimizes Hausdorff measure among surfaces spanning its boundary curve. It is an example of a manifold boundary with a non-manifold minimizing spanning surface.} 
       	\label{fig:Figures_YProblem2}
       \end{figure}
       
	\begin{figure}[htbp]
		\centering
		\includegraphics[height=2in]{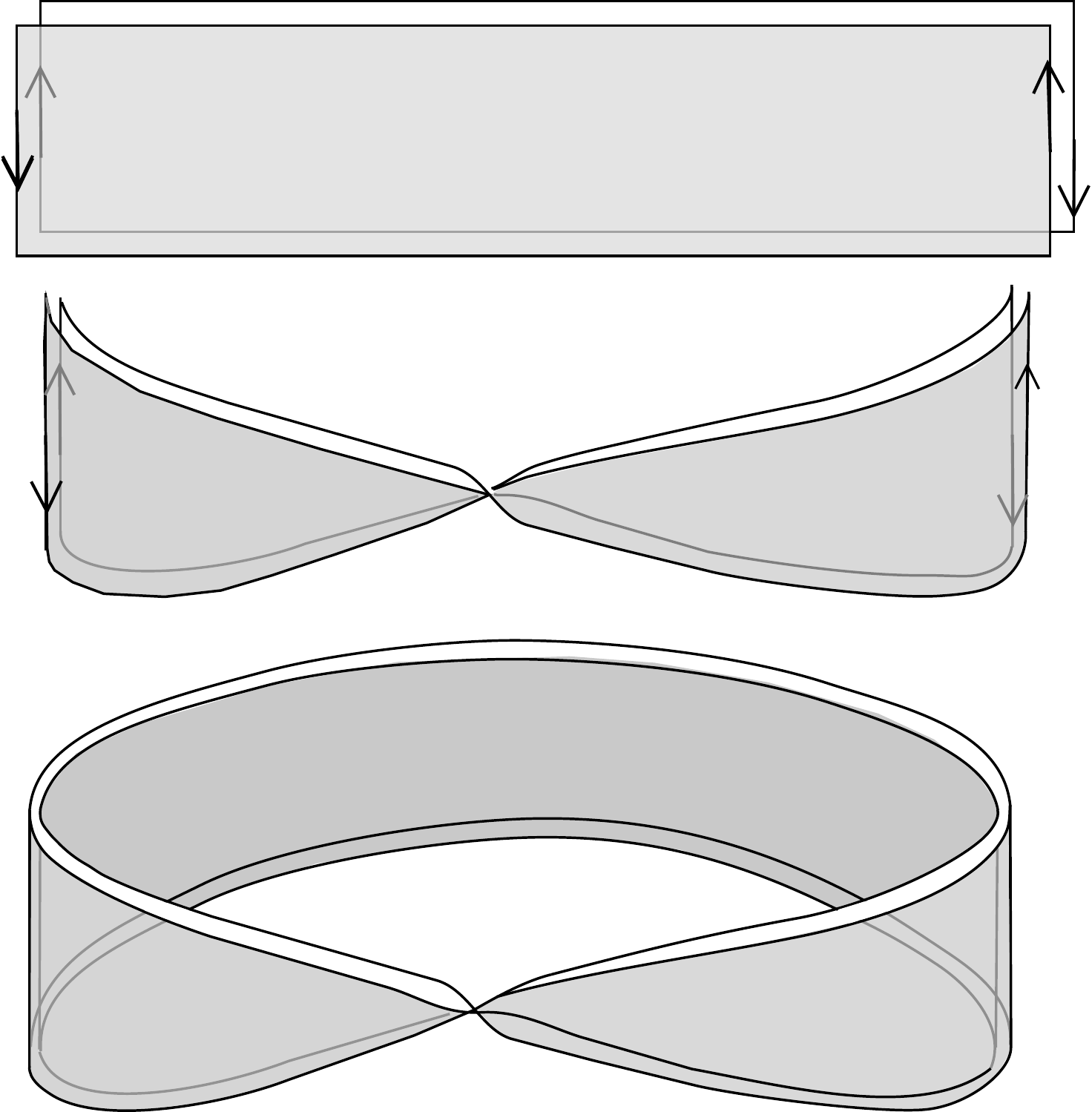}
		\caption{A dipole M\"obius strip. Here, the dipole surface looks like a double cover. This is not the case in general. Consider the triple junction in Figure \ref{fig:Figures_YProblem2}. There, the covering would be 3-to-1, and elsewhere 2-to-1.}
		\label{fig:moebiusdipole}
	\end{figure}

	\subsection*{Statement of our main result}
		\label{par:main_result}
		In our paper the word ``size'' means Hausdorff spherical measure\footnote{Hausdorff spherical measure is the same as Hausdorff measure except that coverings are required to consist of balls, rather than arbitrary sets. The two measures are proportional, and are equal on rectifiable sets. Hausdorff spherical measure works better for analysis however, since it has upper density one a.e. for sets with finite measure, whereas Hausdorff measure does not (see Lemmas \ref{lem:3} and \ref{lem:3H}.) Hausdorff spherical measure also satisfies the fundamental slicing inequality of Lemma \ref{lem:4}, while Hausdorff measure does not.} \( \sp^{n-1} \) of the support of a current. 

		Let \( \Span(M) \) denote the collection of all compact sets which span \( M \). In particular, \( \Span(M) \) includes orientable and non-orientable manifolds of all topological types, and manifolds with multiple junctions. In Section \ref{sec:film_chains} we define a collection \( \F(M) \) of film chains whose elements \( T \) satisfy \( \supp(T)=\supp(T)^*\in \Span(M) \) and \( \supp(\p T) = M \). We show that if \( X\in \Span(M) \) has finite Hausdorff measure, then its core\footnote{The \textbf{core} \( X^* \) of \( X \) is the support of the measure \( \H^{n-1}\lfloor_X \). I.e., \( X^* : = \{p \in X : \H^{n-1}(X \cap \O(p,r)) > 0 \) for all \( r > 0 \), where \( \O(p,r) \) is the open ball of radius \( r \) about \( p \) and \( \H^{n-1} \) is \( (n-1) \)-dimensional Hausdorff measure.} \( X^* \) supports a film chain \( T\in \F(M) \). Thus minimizing size in \( \F(M) \) solves the problem of minimizing \( (n-1) \)-dimensional Hausdorff spherical measure in \( \Span(M) \). 
		
 		Here is our statement and solution to Plateau's Problem for codimension one size minimizing surfaces:
		\begin{thm*}
			Suppose \( M \) is an \( (n-2) \)-dimensional compact orientable submanifold of \( \R^n \), \( n\geq 3 \). There exists an element \( S_0\in \F(M) \) with minimal size. If \( S \) is any such minimizer, its support \( X \) is contained in the convex hull of \( M \), has minimal \( (n-1) \)-dimensional Hausdorff spherical measure \( \sp^{n-1}(X) \) in \( \Span(M) \), and is almost everywhere a real analytic \( (n-1) \)-dimensional minimal submanifold. If \( n=3 \), \( X \) has the local structure of a soap film\footnote{In fact, \( X \) is ``restricted'' in the sense of Almgren (a.k.a. Almgren minimal) and this implies \( X \) is almost everywhere a real analytic \( (n-1) \)-dimensional minimal submanifold by \cite{almgrenannals} (1.7.) Soap film regularity for \( n = 3 \) follows from \cite{taylor}.}.
		\end{thm*}
		
		We can restrict the collection of spanning sets \( \Span(M) \) to smaller sub-collections as follows: Let \( \ell \ge 1 \). We say that a Jordan curve \( S \) is an \emph{\textbf{\( \ell \)-link of \( M \)}} if \( L(M_i, S)=\ell \) and \( L(M_j,S)=0 \), where \( M_i \) and \( M_j \) are as above. We say that \( X \) \emph{\textbf{\( \ell \)-spans}} \( M \) if every \( \ell \)-link of \( M \) intersects \( X \). Let \( \Span(M, \ell) \) be the collection of compact subsets of \( \R^n \) which \( \ell \)-span \( M \). Clearly, \( \Span(M, \ell) \subset \Span(M,k) \) if \( k \) is a factor of \( \ell \). By varying \( \ell \), we obtain specialized Plateau problems. For example, if \( M \) is the boundary circle of the M\"obius strip \( S \) in \( \R^3 \), then \( S\in \Span(M,1) \) but \( S\notin \Span(M,2) \). In our proof that follows, we assume \( \ell = 1 \), but the proof for \( \ell > 1 \) is essentially the same.
		
		The theorem also remains true if we allow a ``simple link'' to have multiple connected components. In this case, when \( M \) is a sphere, the spanning sets are exactly those which do not have a retract onto \( M \).
		
		In a sequel, we demonstrate that our results extend to arbitrary dimension and codimension, as well as to a class of continuous functionals including Almgren's elliptic functionals in \cite{almgrenannals}. We use the following generalized definition of span for higher codimension, loosely dual to Reifenberg's homological definition: if \( A\subset \R^n \), \( n\geq m\geq 1 \), \( G \) is a group of coefficients (in the sense of Eilenberg-Steenrod,) and \( L \) is a \emph{subset} of \( \check{H}^{m-1}(A;G) \), then \( X\supset A \) is a \emph\textbf{{surface with coboundary \( \subseteq L \)}} if the image of the homomorphism \( \check{H}^{m-1}(X;G)\to \check{H}^{m-1}(A;G) \) induced by the inclusion \( A\hookrightarrow X \) is disjoint from \( L \). If \( A \) is an \( (m-1) \)-dimensional compact oriented manifold, and \( G \) has a single generator, then there is a natural choice for \( L \), namely the collection of those cocycles which evaluate to \( 1 \) on the fundamental cycle of a particular connected component of \( M \), and \( 0 \) on the others. In the case that \( G=Z \) and \( m=n-1 \), this choice for \( L \) yields, via Alexander duality, our original linking number definition if one allows a ``simple link'' to have multiple components. 
 		      	
	\subsection*{Other new concepts and methods}
		\label{par:new_definitions_and_methods}
		Several other concepts are new and may be of use for other applications:
		\begin{enumerate}
 			\item \emph{\textbf{Strong compactness}} The compact inclusion of Theorem \ref{thm:compactinclusion} is an alternative to the usual weak compactness result for currents with bounded mass \cite{federerfleming}. The topology is fine enough to eliminate pathologies and achieve regularity, yet coarse enough to have good compactness properties.

			\item A new \emph{\textbf{hair cutting}} procedure Theorem \ref{thm:haircut} is used to remove tentacles as in Figure \ref{fig:tentacles} and uses most of the results in this paper. Our procedure replaces a minimizing sequence of film chains \( S_i\to S_0 \) with a modified sequence \( \tilde{S_i}\to S_0 \) whose supports converge in the Hausdorff metric to the support of \( S_0 \).
			
			\begin{figure}[htbp]
				\centering
				\includegraphics[height=2in]{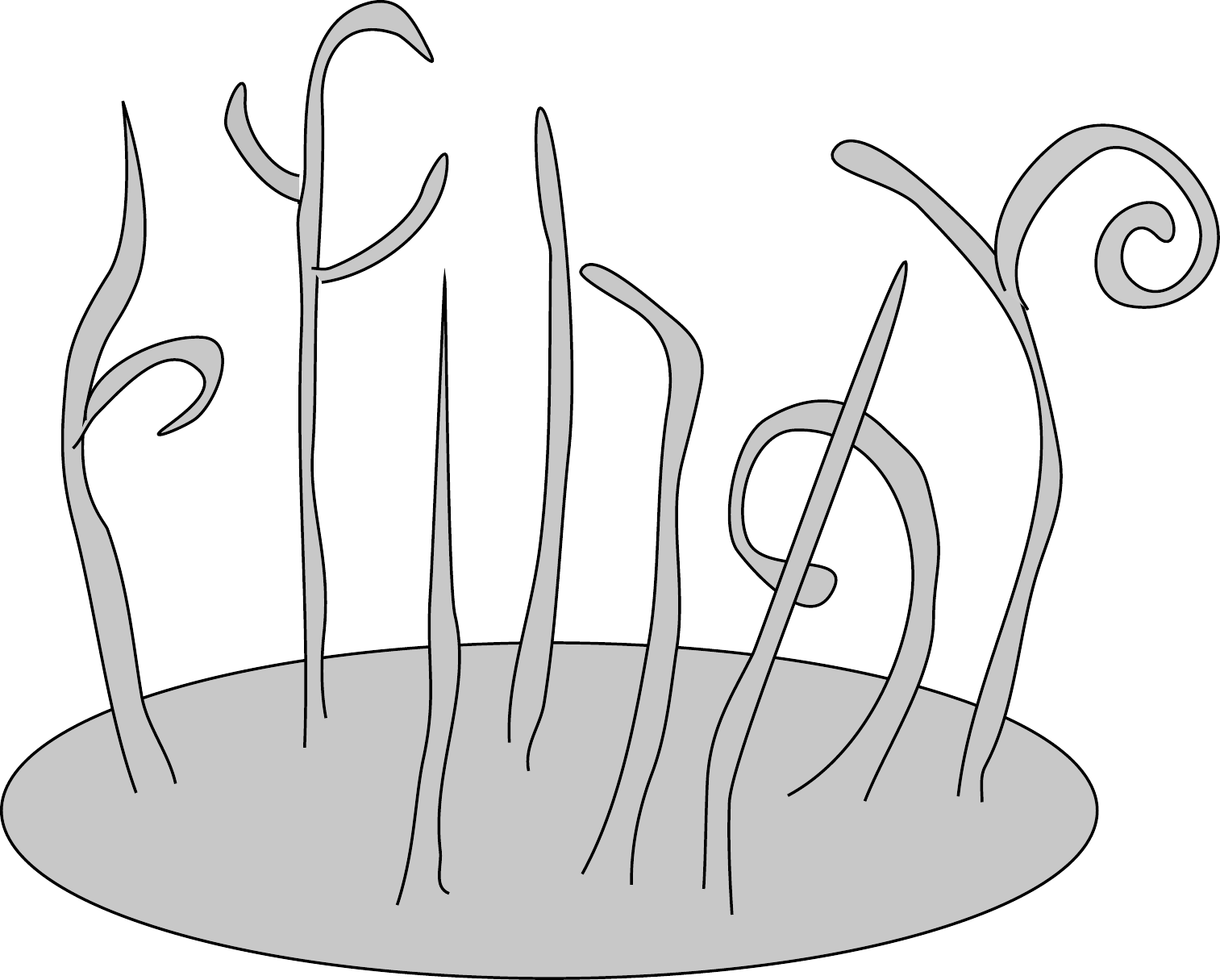}
				\caption{Long, thin tentacles with small Hausdorff \( 2 \)-measure can converge in the Hausdorff metric to a set of infinite Hausdorff \( 2 \)-measure.}
				\label{fig:tentacles}
			\end{figure}	 
		
			\item Our \emph{\textbf{deformation theorem}} \ref{thm:lipspan} is fundamental to this paper and used numerous times. This shows that if a spanning set is deformed, then it remains a spanning set.

			\item Reifenberg defines a type of sequence we call ``regular minimizing'' which satisfies lower semicontinuity of Hausdorff spherical measure due to Reifenberg and Besicovitch\footnote{In \cite{reifenbergbulletin} Reifenberg states, ``Lower semicontinuity in a suitable subsequence is proved by methods reminiscent of A. S. Besicovitch's work here.''} in \cite{reifenberg}. (See Theorem \ref{thm:lsc2} for a proper statement and proof for film chains.) This important result is rather buried in \cite{reifenberg}. It is closely related to the main result of \cite{davidAQM} which shows that Almgren's ``uniformly quasiminimal sequences'' satisfy lower semicontinuity of Hausdorff measure.
		\end{enumerate}
		
	\subsection*{Outline of our proof}
		\label{par:outline}
		
		The first section reviews the results about differential chains which we need in this paper. We refer readers to \cite{thesis}, \cite{OC} and \cite{topological} for figures and more details. Readers unfamiliar with differential chains  might wish to skip this section for the first reading and use dually defined operators to have a decent idea of the rest of our paper.  The second section is a brief application of certain operators of differential chains to set up some simple constructions used to define film chains. Section three establishes a correspondence between finite Borel measures and ``positive chains.'' Section four lists classical results of Hausdorff measures of Besicovitch \cite{besicovitch, besicovitchI}. Section five provides entirely new results about spanning sets, defined via linking numbers. Film chains are introduced for the first time in section six. Corollary \ref{cor:existenceagain} proves the existence of a differential chain \( S_0 \) minimizing our area functional \( \A \) (Definition \ref{def:area}) on film chains. Our methods proving Corollary \ref{cor:existenceagain} build upon, develop, and go beyond those found in \cite{plateau} and \cite{plateau10}, where linking numbers were first used to define spanning currents and a related compactness theorem was established. Theorem \ref{thm:measureequalsm} shows that any such \( S_0 \) is a film chain. This result is the second part of our existence theorem, and we deduce from it that \( S_0 \) is a size minimizing current. Its proof uses most of the results of this paper, and adapts some methods of \cite{reifenberg}.
		
		Part of the proof of Theorem \ref{thm:measureequalsm} ``removes tentacles'' from film chains approximating our solution, and this we do with Theorem \ref{thm:haircut}, a new method which improves on Federer and Fleming projections. In particular, we prove that any minimizer in \( \T \) of our area functional \( \A \) can be approximated by a sequence in \( \F \) converging simultaneously in the topology of differential chains, and in the Hausdorff metric of their supports. Finally, regularity follows from Theorem \ref{thm:lipspan} and a regularity theorem of Almgren.
 
   The authors are grateful to Ulrich Menne for his numerous exchanges regarding the history of varifolds and the work of Frederick Almgren.  
   
   \section*{Subsequent works}
After this paper was posted on the arxiv October 1, 2013 \cite{hpplateau}, several papers appeared which have built upon it: To avoid confusion of priority, we review the results here. 
\begin{itemize}
	\item The first author's paper \cite{plateau10} with the original linking number definition for a smoothly embedded closed curve in \( \R^3 \) was posted on the arxiv in 2011 \cite{plateau10arxiv}. 
	\item This was extended to codimension one smoothly embedded compact submanifolds with finitely many components in \( \R^n \) in the authors' 2013 arxiv post \cite{hpplateau}.  Instead of using Milnor invariants, which are generalizations of linking numbers to finitely many curves, the authors required test loops to link once a given component and not link the rest.  
	\item De Lellis, Maggi, and Ghiraldin \cite{delellisandmaggi} relied upon the definition of linking numbers in \cite{hpplateau} to define spanning sets and replaced methods of differential chains with methods of Radon measures and geometric measure theory to provide a different proof to the main result of our paper. They also used linking numbers to find solutions to sliding minimizers, important in continuum mechanics \cite{Podio-Guidugli} (see also \cite{david2014}), an extension which we do not consider in this paper. 
	\item The authors extended the definition of linking numbers to arbitrary dimension and codimension using cohomological spanning conditions in \cite{cohomology}. They pointed out that one could modify results of Reifenberg \cite{reifenberg} and Almgren \cite{almgrenannals} using cohomological spanning conditions, but no details were provided.   
	
	\item De Philippis, De Rosa, and Ghiraldin \cite{threeguys} extended results of \cite{delellisandmaggi} to arbitrary dimension and codimension using ``linking spheres.''
	
	\item The authors posted an extension of \cite{hpplateau} to Lipschitz integrands with cohomological spanning conditions \cite{lipschitz}.
\end{itemize}

\section*{Notation}
	\label{sec:notation}
	If \( X\subset \R^n \),
	\begin{itemize}
		\item \( \fr\,X \) is the frontier of \( X \);
		\item \( \bar{X} \) is the closure of \( X \);
		\item \( \mathring{X} \) is the interior of \( X \);
		\item \( X^c \) is the complement of \( X \) in \( \R^n \);
		\item \( h(X) \) is the convex hull of \( X \);
		\item \( \diam(X) \) is the diameter of \( X \);
		\item \( \O(X,\e) \) is the open epsilon neighborhood of \( X \);
		\item \( \bar{\O}(X,\e) \) is the closed epsilon neighborhood of \( X \);
		\item \( C_q X \) is the inward cone over \( X \subset \R^n \) with vertex \( q\in \R^n \).
		\item \( \H^m(X) \) is the \( m \)-dimensional (normalized) Hausdorff measure of \( X \);
		\item \( \sp^m(X) \) is the \( m \)-dimensional (normalized) Hausdorff spherical measure of \( X \);
		\item If the Hausdorff dimension of \( X \) is \( m \), then \( X^*:= \{p \in X \,|\, \H^m(X \cap \O(p,r)) > 0 \text{ for all } r > 0 \} \) is the core of \( X \);
		\item \( X(p,r) = X \cap \bar{\O}(p,r) \);
		\item \( x(p,r) = X \cap \fr\,\O(p,r) \);
		\item \( F(p,r) = \int_0^r \sp^{n-2}(x(p,t))dt \);
		\item \( F_k(p,r) = F(p,r) \) if \( X = X_k \).
		\item \( \supp(J) \) is the support of a differential chain \( J \) (Definition \ref{def:support});
		\item \( L(M,N) \) is the linking number of \( M \) and \( N \);
		\item \( \a_m \) is the Lebesgue measure of the unit \( m \)-ball in \( \R^m \);
		\item \( \Span(M) \) is the collection of all compact sets which span  \( M \);
		\item \( \Span^*(M,U) \) is the collection of reduced compact subsets of an open set \( U \) with finite \( \sp^{n-1} \)-measure and which span \( M \).
	\end{itemize}

\section{Differential chains}
	\label{sec:differential_chains}
	In this preliminary section we provide a quick survey of the space of differential chains, its topology, and the operators used throughout the paper\footnote{Readers primarily interested in set theoretic solutions, without currents, can use Whitney's sharp norm to prove a compactness theorem to replace Theorem \ref{thm:compactinclusion}.  The sharp norm can be found in \S\ref{diracchains}-\S\ref{sub:support_of_a_chain} by setting  \( r = 1 \) in \S\ref{sub:difference_chains}.  The remaining subsections \S\ref{sub:pushforward}-\S\ref{ssub:cone_operator} are used to elevate our set theoretic solutions to current solutions.}.
 
	\subsection{Dirac chains}
		\label{Mackey}
		\begin{defn}
			\label{diracchains}
			For \( U \) open in \( \R^n \), let \( \A_k(U) \) denote the vector space of finitely supported functions \( U \to \L^k(\R^n) \). We call \( \A_k(U) \) the space of \emph{Dirac \( k \)-chains in \( U \)}.
		\end{defn}

		We write an element \( A  \in \A_k(U) \) using formal sum notation \( A=\sum_{i=1}^N (p_i; \a_i) \) where \( p_i \in U \) and \( \a_i \in \L^k(\R^n) \). We call \( (p;\a) \) a \( k \)-\emph{element in \( U \)} if \( \a \in \L^k(\R^n) \) and \( p \in U \).

	\subsection{Mass norm}
		\label{sub:mass_norm}
		\begin{defn}
			An inner product \( \<\cdot,\cdot\> \) on \( \R^n \) determines the mass norm on \( \L^k(U) \) as follows: Let \( \<u_1 \wedge \cdots \wedge u_k, v_1 \wedge \cdots \wedge v_k \> := \det(\<u_i,v_j\>) \). The \emph{\textbf{mass}} of a simple \( k \)-vector \( \a = v_1 \wedge \cdots \wedge v_k \) is defined by \( \|\a\| := \sqrt{\<\a,\a\>} \). The \emph{mass} of a \( k \)-vector \( \a \) is \( \|\a\| := \inf\left\{\sum_{j=1}^{N} \|(\a_i)\| : \a_i \mbox{ are simple, } \a = \sum_{j=1}^N \a_i \right\}. \) Define the \emph{mass} of a \( k \)-element \( (p;\a) \) by \(\|(p;\a)\|_{B^0} := \|\a\| \).

			The \emph{mass} of a Dirac \( k \)-chain \( A = \sum_{i=1}^N (p_i; \a_i) \in \A_k(U) \) is given by \[ \|A\|_{B^0} := \sum_{i=1}^N \|(p_i; \a_i)\|_{B^0}. \]
		\end{defn}

	\subsection{Difference chains and the \( B^r \) norm}
		\label{sub:difference_chains}
		\begin{defn}
			\label{def:seminorms} 
			Let \( h(X) \) denote the convex hull of \( X \).
			Given \( u\in \R^n \) and a \( k \)-element \( (p;\a)\in \A_k(\R^n) \), let \( T_u(p;\a) := (p+u;\a) \) be \emph{\textbf{translation}} through \( u \), and \( \D_u(p;\a) := (T_u -I)(p; \a) \). Extend both operators linearly to \( \A_k(\R^n) \). If \( u_1,\dots,u_j\in \R^n \), define the \emph{\textbf{\( j \)-difference \( k \)-chain}} \( \D_{\{u_1,\dots,u_j\}}(p; \a) := \D_{u_j}\circ\dots\circ\D_{u_1}(p; \a) \). The operators \( \D_{u_i} \) and \( \D_{u_k} \) commute, so \( \D_{\{u_1,\dots,u_j\}} \) depends only on the set \( \{u_1,\dots , u_j\} \) and not on the ordering. Say \( \D_{\{u_1,\dots,u_j\}} (p;\a) \) is \emph{inside} \( U \) if \( h(\supp(\D_{\{u_1,\dots,u_j\}} (p;\a)))\subset U \).
		\end{defn}
						
		To simplify our next definition, let \( \D_{\emptyset} (p;\a) := (p;\a) \), let \( s^j:=\{u_1,\dots,u_j\} \) denote a set of \( j \) vectors in \( \R^n \), and let \( s^0=\emptyset \). Let \( \|s^j\|=\|u_1\|\cdots\|u_j\| \), let \( \|s^0\|=1 \), and let \( |\D_{s^j}(p;\a)|_{B^j}=\|s^j\| \|\a\| \).

		\begin{defn}
			\label{def:norms}
			For \( A \in \A_k(U) \) and \( r\ge 0 \), define the norm \[ \|A\|_{B^{r,U}} := \inf \left\{ \sum_{i=0}^w |\D_{s_i^{j_i}}(p_i;\a_i)|_{B^{j_i}} : A=\sum_{i=0}^w \D_{s_i^{j_i}}(p_i;\a_i),\,\, 0\le j_i\le r \mbox { and } \D_{s_i^{j_i}}(p_i,\a_i) \mbox { is inside } U \right\}. \]
		\end{defn}

		That is, the infimum is taken over all ways of writing \( A \) as a finite sum of \( j \)-difference \( k \)-chains, for \( j \) between \( 0 \) and \( r \). It is shown in (\cite{OC}, Theorem 3.2.1) that \( \|\cdot\|_{B^{r,U}} \) is indeed a norm on \( \A_k(U) \). For simplicity, we often write \( \|\cdot\|_{B^r}:=\|\cdot\|_{B^{r,U}} \) if \( U \) is understood from the context.

		\begin{defn}
			\label{differential}
			Let \( \hB_k^r = \hB_k^r(U) \) be the Banach space obtained by completing the normed space \( (\A_k(U), \|\cdot\|_{B^{r,U}}) \). Elements of \( \hB_k^r(U), 0 \le r < \i \), are called \emph{\textbf{differential \( k \)-chains of class \( B^r \) in \( U \)}}.
		\end{defn}

		The maps \( \hB_k^r(U_1) \to \hB_k^r(U_2) \) induced by inclusions \( U_1\hookrightarrow U_2 \) are continuous. Moreover, if \( r\le s \), the identity map \( (\A_k(U), \|\cdot\|_r)\to (\A_k(U),\|\cdot\|_s) \) is continuous and thus extends to a continuous map \( u_k^{r,s}: \hB_k^r(U) \to \hB_k^s(U) \) called the \emph{\textbf{linking map}}.

		\begin{prop}[\cite{OC}, Lemma 5.0.2, Corollary 5.0.5]
			\label{prop:inclusion}
			The linking maps \[ u_k^{r,s}: \hB_k^r(U) \to \hB_k^s(U) \] satisfy
			\begin{enumerate}
				\item \( u_k^{r,r} = Id \);
				\item \( u_k^{s,t} \circ u_k^{r,s} = u_k^{r,t} \) for all \( r \le s \le t \);
				\item The image \( u_k^{r,s}(\hB_k^r(U) \) is dense in \( \hB_k^s(U) \);
				\item \( u_k^{r,s} \) is injective for each \( r \le s \).
			\end{enumerate}
		\end{prop}

		In \cite{OC} and \cite{topological} we study the inductive limit space \( \hB_k^\i(U) := \varinjlim \hB_k^r(U) \).

		\begin{thm}
			\label{thm:compactinclusion}
			If \( U\subset \R^n \) is bounded, the linking map \( u_n^{0,1}:\hB_n^0(U) \to \hB_n^1(U) \) is compact.
		\end{thm}
		
		\begin{proof} 
			Since \( \hB_n^1(U) \) is a Banach space, and since \( \A_n(U) \) is dense in \( \hB_n^0(U) \), it suffices to show that the image of \( \Theta := \{ J \in \A_n(U): \|J\|_{B^0} \le 1\} \) is totally bounded.

			For \( k \in \N \) let \( \Xi(k) \) be a covering of \( U \) by finitely many balls \( \O(x_j, 2^{-k}) \) centered at points \( x_j\in U \), \( 1\leq j \leq N_k \). Let \( \Q(k)=\{ y 2^{-k} N_k^{-1} : y\in\Z,\, 0 \le y \le 2^k N_k \} \) and let \[{\cal Z}(k) = \left\{ \sum_{j=1}^{N_k} (x_j; \b_j) \in u_n^{0,1}(\Theta) : \|\b_j\|  \in \Q(k) \text{ for all } 1\leq j\leq N_k\right\}. \]

			Let \( A = \sum_{s=1}^r (p_s; \a_s)\in u_n^{0,1}(\Theta) \). We approximate \( A \) with an element of \( {\cal Z}(k) \) as follows:

		 	Fix \( 1\leq s\leq r \). Since \( \Xi(k) \) covers \( U \), we have \( p_s\in \O(x_{j_s}, 2^{-k}) \) for some \( 1\leq j_s \leq N_k \). Set \( p_s'=x_{j_s} \). Put \( A'=\sum_{s=1}^r (p_s'; \a_s) \). By summing the \( n \)-vectors \( \a_s \) at the same point \( x_j \) and inserting zeros as necessary, we can write \( A'=\sum_{j=1}^{N_k}(x_j; \mu_j) \), where \( \sum_{j=1}^{N_k} \|\mu_j\|\leq 1 \).

			Now let \( \mu_j' \) be the \( n \)-vector whose mass \( \|\mu_j'\|\in \Q(k) \) satisfies \( 0\le \|\mu_j\|- \|\mu_j'\| < 2^{-k}\cdot N_k^{-1}. \) Set \( A'' = \sum_{j=1}^{N_k} (x_j; \mu_j') \in {\cal Z}(k)  \). It follows that
			\begin{align*}
				\|A - A''\|_{B^1} &\leq \sum_{s=1}^r \|(p_s;\a_s)-(p_s';\a_s)\|_{B^1} + \sum_{j=1}^{N_k}\|(x_j; \mu_j-\mu_j')\|_{B^1}\\
				&\leq \sum_{s=1}^r \|p_s-p_s'\|\cdot \|\a_s\| + \sum_{j=1}^{N_k} \|\mu_j-\mu_j'\|\\
				&\leq 2^{1-k}.
			\end{align*}
		\end{proof}

		\begin{defn}
			\label{def:forms}
			Let \( \B_k^0(U) \) be the Banach space of bounded \( k \)-forms on \( U \) equipped with the comass norm \( \|\o\|_{B^0} = \sup_{\|\a\|=1}\o(\a) \). For each \( r \ge 1 \), let \( \B_k^r(U) \) be the Banach space of \( (r-1) \)-times differentiable \( k \)-forms with comass bounds on the \( s \)-th order directional derivatives for \( 0 \le s \le r-1 \) with the \( (r-1) \)-st derivatives satisfying a bounded Lipschitz condition\footnote{Here, \( U \) is equipped with the intrinsic metric induced by the standard metric on \( \R^n \). That is, \( d_U(p,q)=\inf L(\g) \), where the infimum is taken over all paths \( \g \) in \( U \) from \( p \) to \( q \). We say that a form \( \o \) is Lipschitz if \( |\o|_{Lip}:=\sup \frac{|\o(p;\a)-\o(q;\a)|}{\|\a\|d_U(p,q)}< \infty \). C.f. Whitney's \emph{Lipschitz comass constant}, \cite{whitney}, V, \S 10.}, with norm given by \( \|\o\|_{B^r} = \sup_{|\ell| \le r-1}\{\| D^{\ell}\o\|_{B^0}, |D^{r-1} \o|_{Lip} \} \). Elements of \( \B_k^r(U) \) are called \emph{\textbf{differential \( k \)-forms of class \( B^r \) in \( U \)}}.
		\end{defn}

		We always denote differential forms by lower case Greek letters such as \( \o, \eta \) and differential chains by upper case Roman letters such as \( J, K \), so there is no confusion when we write \( \|\o\|_{B^r} \) or \( \|J\|_{B^r} \).

		\begin{prop}[Isomorphism Theorem, \cite{OC} Theorem 3.5.2]
			\label{thm:isomorphism}
	 		For \( r\ge 0 \), the continuous dual space \( \hB_k^r(U)' \) equipped with the dual norm is isometric to \( \B_k^r(U) \) via the restriction of covectors in \( \hB_k^r(U)' \) to \( k \)-elements.
		\end{prop}

		Thus there is a jointly continuous bilinear pairing \( \hB_k^r\times \B_k^r \to \R \) given by \( (J,\o) \to \cint_J \o := \o(J) \) satisfying \( |\o(J)|\le \|o\|_{B^r} \|J\|_{B^r} \). By a slight abuse of notation, we write \( \o(J) \) to mean the evaluation on \( J \) of the covector corresponding to \( \o \) given by the isomorphism of Proposition \ref{thm:isomorphism}.

	\subsection{Support of a chain}
 		\label{sub:support_of_a_chain}
		\begin{defn}
			\label{def:support}
			If \( J \in \hB_k^r(U) \), define the \emph{\textbf{support}} of \( J \) to be \[ \supp(J) := \left\{ p \in \R^n \,:\, \forall \e > 0, \exists\, \eta \in \B_k^r(U) \mbox{ with }  \eta \mbox{ supported in } \O(p,\e) \mbox{ s.t. } \cint_J \eta \ne 0 \right\}. \] The support of a nonzero differential chain is a nonempty closed subset of \( \overline{U} \) (\cite{OC}, Theorems 6.3.4). This definition also agrees with the usual notion of support on the subspace \( \A_k(U) \) of \( \hB_k^r(U) \).
		\end{defn}

		\begin{lem}
			\label{lem:converge}
			If \( J_i \to J \) in \( \hB_k^r(U) \) and \( p \in \supp(J) \), then there exists \( p_i \in \supp(J_i) \) such that  \( p = \lim_{i \to \i} p_i \).
		\end{lem}

		\begin{proof}
			Each \( \O(p,r) \) must intersect all \( \supp(J_i) \) for sufficiently large \( i \), for if not, there exists a differential form \( \eta \) supported in \( \O(p,r) \) with \( \cint_J \eta \ne 0 \), and a subsequence \( i_j\to\i \) such that \( \cint_{J_{i_j}} \eta = 0 \) for each \( i_j \), contradicting continuity of the integral \( \cint \).
		\end{proof}

	\subsection{Pushforward}
		\label{sub:pushforward}
		\begin{defn}
			\label{def:push}
			Suppose \( U_1 \subset \R^n\) and \( U_2 \subset \R^w \) are open and \( F:U_1 \to U_2 \) is a differentiable map. For \( p \in U_1 \) define \( F_*(p; \a) := (F(p), F_{p*}\a) \) for all \( k \)-elements \( (p;\a) \) and extend to a linear map \( F_*:\A_k(U_1) \to \A_k(U_2) \) called \emph{\textbf{pushforward}}.
		\end{defn}

		\begin{defn}
			\label{def:Mr}
			For \( r \ge 1 \) let \( \M^r(U, \R^w) \) be the vector space of differentiable maps \( F:U \to \R^w \) whose coordinate functions \( F_i \) satisfy \( \p F_i/\p e_j \in B^{r-1}(U) \) for all \( j \). Define the seminorm \[ \rho_r(F) := \max_{i,j} \{\| \p F_i/\p e_j\|_{B^{r-1, U}}\}. \] Let \( \M^r(U_1, U_2) := \{F \in \M^r(U_1, \R^w) : F(U_1) \subset U_2 \subset \R^w\} \).
		\end{defn}

		A map \( F \in \M^1(U, \R^w) \) may not be bounded, but its directional derivatives must be. An important example is the identity map \( x \mapsto x \) which is an element of \( \M^1(\R^n,\R^n) \).

		\begin{prop}[\cite{OC}, Theorem 7.0.6, Corollary 7.0.18]
			\label{prop:A}
	 		If \( F \in \M^r(U_1, U_2) \), then \( F_* \) satisfies \[ \|F_*(A)\|_{B^{r,U_2}} \le n2^r \max\{1, \rho_r(F)\} \|A\|_{B^{r,U_1}} \] for all \( A \in \A_0(U_1) \). It follows that \( F_* \) extends to a well defined continuous linear map \( F_*: \hB_k^r(U_1) \to \hB_k^r(U_2) \), whose dual map \( F^*:\B_k^r(U_2) \to \B_k^r(U_1) \) is the usual pullback of forms. Thus \( \cint_{F_* J} \o = \cint_J F^* \o \) for all \( J \in \hB_k^r(U_1) \) and \( \o \in \B_k^r(U_2) \).
		\end{prop}

		\begin{defn}
			A \emph{\textbf{\( k \)-cell \( \s \) in \( U\subseteq \R^k \)}} is a bounded finite intersection of closed half-spaces, such that \( \s \subset U \). An \emph{\textbf{oriented affine \( k \)-cell in \( U\subseteq \R^n \)}} is a \( k \)-cell in \( U \) which is also contained in some affine \( k \)-subspace \( K \) of \( \R^n \), and which is equipped with a choice of orientation of \( K \).
		\end{defn}

		\begin{defn}
			We say that \( J \in \hB_k^r(U)  \) \emph{\textbf{represents}} an oriented \( k \)-dimensional submanifold (possibly with boundary) \( M \) of \( U \) if \( \cint_J \o = \int_M \o \) for all \( \o \in \B_k^r(U) \).
		\end{defn}

		\begin{prop}[Representatives of \( k \)-cells, \cite{OC}, Theorem 4.2.2]
			\label{thm:B}
			If \( \s \) is an oriented affine \( k \)-cell \( \s \) in \( U \), then there exists a unique differential \( k \)-chain \( \widetilde{\s} \in \hB_k^1(U) \) which represents \( \s \).
		\end{prop}

		\begin{defn}
			\label{defn:poly}
			Let \[ \P_k(U):= \left\{ J\in \hB_k^1(U) : J=\sum_{i=1}^N a_i \widetilde{\sigma_i}, \,\,\sigma_i \mbox{ oriented affine k-cell }, a_i\in \R \right\}. \]  Elements of \( \P_k(U) \) are called \emph{\textbf{polyhedral \( k \)-chains in \( U \)}}.
		\end{defn}

		The subspace \( \P_k(U) \subset \hB_k^1(U) \) is dense in \( \hat{B}_k^r(U) \), for \( r\ge 1 \) (\cite{OC}, Theorem 4.2.5).

		\begin{remark}
			\label{sharpnorm}
			Polyhedral chains and Dirac chains are both dense subspaces of Whitney's sharp chains and \( \hB_k^r \), \( r \ge 1 \) (\cite{whitney}, VII, \S 8, Theorem 8A). The sharp norm of Whitney is comparable to the \( B^1 \) norm for \( k \)-chains, and they are identical for \( k = 0 \).
		\end{remark}

		\begin{defn}
			\label{algebraic}
			If \( \s \) is an oriented affine \( k \)-cell in \( U \) and \( F \in \M^1(U,W) \), then \( F_* \widetilde{\s} \in \hB_k^1(W) \), and is called an \emph{\textbf{algebraic \( k \)-cell}}.
		\end{defn}

		An algebraic \( k \)-cell \( F_* \widetilde{\s} \) is not the same as a singular \( k \)-chain \( F\s \). For example, if \( F(x) = x^2 \) and \( \s = (-1, 1) \), then the algebraic \( 1 \)-cell \( F_* \widetilde{\s} = 0 \), but the singular \( 1 \)-cell \( F \s \ne 0 \). In particular, singular chains have no relations, whereas algebraic \( k \)-cells inherit relations from the topology on \( \hB_k^r \).

		\begin{lem}[\cite{OC} Theorem 4.2.2, Corollary 7.0.18, Proposition 7.0.19]
			\label{rep}
			If \( \s \) is an oriented \( k \)-cell in \( U\subseteq \R^k \) and \( F \in \M^1(U,W) \) is a smooth embedding, then \( \cint_{F_*\widetilde{\s}} \o = \int_{F(\s)} \o \) for all \( k \)-forms \( \o\in \B_k^1(W) \). Moreover, \( \supp({F_*\widetilde{\s}}) = F(\s) \).
		\end{lem}

		The next result is a consequence of Lemma \ref{rep}.

		\begin{prop}
			\label{prop:submanifold}
			If \( M \) is a compact oriented \( k \)-submanifold with boundary, smoothly embedded in \( U\subseteq \R^n \), then there exists \( \widetilde{M} \in \hB_k^1(U) \) with \[ \cint_{\widetilde{M}} \o = \cint_M \o \] for all \( \o \in \B_k^1(U) \). Furthermore \( \supp(\widetilde{M}) = M. \) If \( U \subset \R^n \) is a bounded open set equipped with an orientation, then there exists a unique differential \( n \)-chain \( \widetilde{U} \in \hB_n^1(U) \) which represents \( U \).
		\end{prop}

	\subsection{Vector fields}
		\label{sub:vector_fields}
		Let \( \V^r(U) \) be the Banach space of vector fields \( X \) on \( U \) such that \( X^\flat\in \B_1^r(U) \), equipped with the norm \( \|X\|_{B^r}:=\|X^\flat\|_{B^r} \).

	\subsection{Extrusion}
		\label{sub:extrusion}
		The interior product \( i_X \) of differential forms \( \B_k^r(U) \) with respect to a vector field \( X\in \V^r(U) \) is dual to an operator \( E_X \) on differential chains \( \hB_k^r(U) \).

		\begin{defn}
			\label{def:extrusion}
			Let \( X \in \V^r(U) \). Define the graded operator \emph{\textbf{extrusion}} \( E_X: \A_k(U) \to \A_{k+1}(U) \) by \( E_X (p;\a) := (p; X(p) \wedge \a) \) for all \( p \in U \) and \( \a \in \L^k(\R^n) \).
		\end{defn}

		\begin{thm}[\cite{OC}, Theorems 8.2.2, 8.2.3]
			\label{thm:C}
			If \( X \in \V^r(U) \) and \( A \in \A_k(U) \), then \[ \|E_X(A)\|_{B^r} \le n^22^r\|X\|_{B^r}\|A\|_{B^r}. \]  Furthermore,  \( E_X: \hB_k^r(U) \to \hB_{k+1}^r(U) \) and \( i_X: \B_{k+1}^r(U) \to \B_k^r(U) \) are continuous graded operators satisfying
			\begin{equation}
				\label{eq:EX}
				\cint_{E_X J} \o = \cint_J i_X \wedge \o
			\end{equation}
			for all \( J \in \hB_k^r(U) \) and \( \o \in \B_{k+1}^r(U) \).
		\end{thm}

		Therefore, \( E_X \) extends to a continuous linear map \( E_X: \hB_k^r(U) \to \hB_{k+1}^r(U) \), and the dual operator \( i_X: \B_{k+1}^r(U)\to \B_k^r(U) \) is also continuous.

	\subsection{Retraction}
		\label{sub:retraction}
		\begin{defn}
			\label{def:retraction}
			For \( \a = v_1 \wedge \cdots \wedge v_k \in \L^k(\R^n) \), let \( \hat{\a}_i := v_1 \wedge \cdots \wedge\hat{v_i}\wedge \cdots \wedge v_k   \in \L^{k-1}(\R^n) \). For \( X\in \V^r(U) \) define the graded operator \emph{\textbf{retraction}} \( E_X^\dagger: \A_k(U) \to \A_{k-1}(U) \) by \( (p; \a) \mapsto \sum_{i=1}^k (-1)^{i+1} \<X(p),v_i\> (p; \hat{\a}_i), \) for \( p \in U \) and \( \a \) simple, extending linearly to all of \( \A_k(U) \).
		\end{defn}

		A straightforward calculation in \cite{OC} shows this to be well-defined. The dual operator on forms is wedge product with the \( 1 \)-form \( X^\flat \). That is to say, it is the operator \( X^\flat \wedge \cdot: \A_{k-1}(U)^* \to \A_k(U)^* \).

		\begin{thm}[\cite{OC}, Theorem 8.3.3, 8.3.4]
			\label{thm:Cdagger}
			If \( r\ge 1 \), \( X\in \V^r(U) \) and \( J \in \hB_k^r(U) \), then \[ \|E_X^\dagger(J)\|_{B^r} \le k {n \choose k} \|X\|_{B^r}\|J\|_{B^r}. \]  Furthermore,  \( E_X^\dagger: \hB_k^r(U) \to \hB_{k-1}^r(U) \) and \( X^\flat \wedge \cdot: \B_{k-1}^r(U) \to \B_k^r(U) \) are continuous graded operators satisfying
			\begin{equation}
				\label{eq:EXD}
				\cint_{E_X^\dagger J} \o = \cint_J X^\flat \wedge \o
			\end{equation}
			for all \( J \in \hB_k^r(U) \) and \( \o \in \B_{k-1}^r(U) \).
		\end{thm}

		The commutation relation
		\begin{equation}
			\label{eq:cr}
			E_V E_W^\dagger + E_W^\dagger E_V = \<V,W\> Id
		\end{equation}
		can be found in  \cite{OC} Proposition 8.3.5(d).

	\subsection{Boundary}
		\label{sub:boundary}
		There are several equivalent ways to define the boundary operator \( \p:\hB_k^r(U) \to \hB_{k-1}^{r+1}(U) \) for \( r \ge 1 \). We have found it very useful to define boundary on Dirac chains directly. For \( v \in \R^n \), and a \( k \)-element \( (p; \a) \) with \( p \in U \), let \( P_v (p; \a) := \lim_{t \to 0} (p+tv;\a/t) - (p;\a/t) \). It is shown in (\cite{OC}, Lemma 8.4.1 that this limit exists as a well-defined element of \( \hB_k^2(U) \). We may then linearly extend this to a map \( P_v: \A_k(U) \to \hB_k^2(U) \) called \emph{\textbf{prederivative}}. Moreover, the inequality \( \|P_v (A)\|_{B^{r+1}} \le \|v\|\|A\|_{B^r} \) holds for all \( A \in \A_k(U) \) (\cite{OC}, Lemma 8.4.3(a)). For an orthonormal basis \( \{e_i\} \) of \( \R^n \), set
		\begin{equation}
			\label{eq:bd}
			\p := \sum P_{e_i} E_{e_i}^\dagger
		\end{equation}

		Since \( P_{e_i} \) and \( E_{e_i}^\dagger \) are continuous, \emph{\textbf{boundary}} \( \p \) is a well-defined continuous operator \( \p:\hB_k^r(U) \to \hB_{k-1}^{r+1}(U) \) that restricts to the classical boundary operator on  polyhedral \( k \)-chains \( \P_k(U) \), and is independent of choice of \( \{e_i\} \). Furthermore, \( \p \circ F_* = F_* \circ \p \) (\cite{OC}, Theorems 8.5.1 and Proposition 8.5.6).

		\begin{thm}[Stokes' Theorem, \cite{OC} Theorems 8.5.1, 8.5.2,  8.5.4]
			\label{thm:D}
			The bigraded operator boundary \( \p:\hB_k^r(U)\to\hB_{k-1}^{r+1}(U) \) is continuous with \( \p \circ \p = 0 \), and \( \|\p J\|_{B^{r+1}} \le k n\|J\|_{B^r} \) for all \( J \in \hB_k^r \) and \( r \ge 1 \). Furthermore, if \( \o \in \B_{k-1}^{r+1}(U) \) and \( J \in \hB_k^r(U) \), then \[ \cint_{\p J} \o = \cint_J d \o. \]
		\end{thm}

		\begin{lem}
			\label{lem:supportboundary}
			If \( J \in \hB_k^r(U) \), then \( \supp(\p J) \subset \supp(J) \).
		\end{lem}

		\begin{proof}
			Suppose \( p \in \supp(\p J) \) and \( p \notin \supp(J) \).  For sufficiently small \( r > 0 \),  \( \O(p,r) \cap \supp(J) = \emptyset \). Let \( \eta \) be a smooth \( k \)-form supported in \( \O(p,r) \) and with \( \cint_{\p J} \o \ne 0 \). Then \( \cint_J d\o \ne 0 \). But \( d\o \) is also supported in \( \O(p,r) \), yielding a contradiction.
		\end{proof}

		\begin{lem}
			\label{lem:linesandboundary}
			Suppose\footnote{This can be made more general. Namely, it is not necessary for some open sets \( U \) to assume that \( p\in U \).} \( J \in \hB_n^r(U) \), \( p \in \supp(J)\cap U \), and \( q \notin \supp(J) \). If \( \a:[0,1]\to \R^n \) is any smoothly embedded path connecting \( p \) and \( q \), then \( \a([0,1]) \cap \supp(\p J) \ne \emptyset \).
		\end{lem}

		\begin{proof}
			Suppose there exists such a path \( \a \), but \( \a([0,1]) \cap \supp(\p J) = \emptyset \). Let \( T \) be a tubular neighborhood of \( \a([0,1]) \) disjoint from \( \supp(\p J) \). Let \( \e>0 \) such that \( \O(p,\e) \subset T\cap U \), \( \O(q,\e)\subset T \) and \( \O(q,\e)\cap \supp(J)=\emptyset \). Since \( p \in \supp(J) \), there exists \( \eta\in \B_n^r(U) \) supported in \( \O(p,\e) \) with \( \cint_J \eta \neq 0 \). Since \( \O(p,\e)\subset U \), we may extend \( \eta \) by \( 0 \) to all of \( \R^n \). By Theorem 5.0.3 in \cite{OC}, we can assume that \( \eta \) is smooth. Let \( \a \) be a smooth \( n \)-form supported in \( \O(q,\e) \) such that \( \int_T \a=\int_T \eta \).

			Then \( \int_T \eta -\a=0 \), and so by compactly supported de Rham theory, \( \eta-\a=d\o \) for some smooth \( (n-1) \)-form \( \o \) supported in \( T \). But \( \cint_J \a = 0 \), and so \( \cint_J d\o\neq 0 \), yielding a contradiction with the assumption that \( T \) is disjoint from \( \supp(\p J) \).
		\end{proof}

		\begin{cor}
			\label{cor:threechains}
			Suppose \( J \in \hB_n^r(U) \) satisfies \( \supp(J)\subset U \) and \( \supp(J) \) has empty interior. Then \( \supp(\p J) = \supp(J) \).
		\end{cor}

		\begin{proof}
			We know that \( \supp(\p J) \subseteq \supp(J) \) by Lemma \ref{lem:supportboundary}. Let \( p \in \supp(J) \). By hypothesis, every \( \O(p,\e)\subset U \) contains a point \( q\notin \supp(J) \), and hence a point \( q'\in \supp(\p J) \) by \ref{lem:linesandboundary}.
		\end{proof}

		\begin{cor}
			If \( \widetilde{U}\in \hB_n^1(U) \) is as in Proposition \ref{prop:submanifold} (or more generally, if we consider \( \widetilde{U} \) as the representative of \( U \) in \( \hB_n^1(W) \) for any \( U\subseteq W \),) then \( \supp(\p \widetilde{U})=\fr\,U \).
		\end{cor}
		
	\subsection{Prederivative}
		\label{sub:prederivative}
		The topological dual to Lie derivative \( \Lie_X \) of differential forms \( \B_k^r(U) \) restricts to a continuous operator \( P_X \) on differential chains \( \hB_k^r(U) \).

		\begin{defn}
			\label{def:preder}
			Suppose \( X \in \V^r(U), r \ge 1 \). Define the continuous graded linear operator \emph{prederivative} \( P_X: \hB_k^r(U) \to \hB_k^{r+1}(U) \) by \[ P_X := E_X \p + \p E_X. \]
		\end{defn}

		This agrees with the previous definition of prederivative \( P_v \) for \( v \in \R^n \) in the first paragraph of \( \S\ref{sub:boundary} \) since \[ E_v \p + \p E_v = \sum_i P_{e_i}( E_v E_{e_i}^\dagger - E_{e_i}^\dagger E_v) = \sum_i P_{e_i} \<v,e_i\> I = P_v. \] Its dual operator is Lie derivative \( \Lie_X \) by Cartan's formula, and \( \Lie_X: \B_k^r(U) \to \B_k^{r-1}(U) \) is continuous. Furthermore, \( P_X \p = \p P_X \).

		\begin{prop}[\cite{OC},Theorem 8.6.2]
			\label{thm:E}
			If \( 1 \le r < \i \) and \( X \in \V^r(U) \), then \[ \|P_X(J)\|_{B^{r+1}} \le 2k n^3 2^r \|X\|_{B^{r}} \|J\|_{B^r} \] for all \( J \in \hB_k^r(U) \).
		\end{prop}

		\begin{prop}[\cite{OC},Theorem 8.6.5]
			\label{thm:F}	
			If \( r\ge 1 \), \( X \in \V^{r+1}(U) \) and \( J \in \hB_k^r(U), r \ge 0 \) has compact support, then \[ P_X J = \lim_{t \to 0}( \phi_{t*} J/t - J/t) \] where \( \phi_t \) is the time-\( t \) map of the flow of \( X \).
		\end{prop}
		
		The next lemma is proved similarly to Lemma \ref{lem:supportboundary}:
		
		\begin{lem}
			\label{lem:supportPE}
			The supports of \( E_X J \) and \( P_X J\) satisfy \( \supp(E_X J)\subset \supp(J) \) and \( \supp(P_X J)\subset \supp(J) \).
		\end{lem}

	\subsection{Cone operator}
		\label{ssub:cone_operator}

		Let \( U\subset \R^n \) and \( W\subset \R^w \) be open sets, and for \( q\in W \) let \( \iota_1^q : U\to U\times W  \) the map \( p\mapsto (p,q) \).  Similarly define \( \iota_2^p: W\to U\times W \) for \( p\in U \). If \( P=\sum_i (p_i;\a_i)\in \A_k(U) \) and \( Q=\sum_j (q_j;\b_j)\in \A_l(W) \) define the \emph{\textbf{Cartesian wedge product}} \( \hat{\times}: \A_k(U)\times \A_l(W)\to \A_{k+l}(U\times W) \) by \( P\hat{\times}Q:= \sum_{i,j}((p_i,q_j);\iota_{1*}^{q_j}\a_i \wedge \iota_{2*}^{p_i}\b_j).\) A calculation shows \( \|P\hat{\times}Q\|_{B^{r+s}}\leq \|P\|_{B^r}\|Q\|_{B^s} \), and so \( \hat{\times} \) extends to a jointly continuous bilinear map \( \hat{\times}:\hB_k^r(U)\times \hB_l^s(W)\to \hB_{k+l}^{r+s}(U\times W) \). We can improve on the order \( r+s \) in some cases. In particular, if \( \widetilde{(a,b)}\in \hB_1^1((a,b)) \) represents the interval \( (a,b)\subset \R\), and \( J\in \hB_k^r(U) \), then \( \widetilde{(a,b)}\hat{\times}J\in \hB_{k+1}^r((a,b)\times U) \), and as Proposition 11.1.9 in \cite{OC} shows,
		
		\begin{align}
			\label{prop:GG}
			\|\widetilde{(a,b)} \hat{\times} J \|_{B^r} \le (b-a) \|J\|_{B^r}.
		\end{align}
		
		\begin{defn}
			\label{def:cone}
			Suppose \( F:[0,1]\times U\to U \) and \( q\in U \) satisfy \( F(0,p) = p \) and \( F(1,p) = q \) for all \( p \in U \). Suppose further that \( F\lfloor_{(0,1)\times U}\in \M^r((0,1) \times U,U) \). Then define the \emph{\textbf{cone operator}} \( \k_F: \hB_k^r(U) \to \hB_{k+1}^r(U) \) by \( \k_F(J) := F_*(\widetilde{(0,1)} \hat{\times} J) \).  
		\end{defn}
		
		According to Proposition \ref{prop:A} and \eqref{prop:GG}, \( \k_F \) is continuous, with \[ \|\k_F(J) \|_{B^r} \le (n+1)2^r\max\{1, \rho_r(F) \} \|J\|_{B^r}. \] Moreover, by Theorem 5.1.1 of \cite{plateau10} we have
		\begin{equation}
			\label{eq:cone}
			\k_F \p + \p \k_F = Id.
		\end{equation}

\section{Constructions with embedded submanifolds}
	\label{sub:constructions_with_curves}
	\setcounter{thm}{0}

	Let \( M \subset \R^n \) be a compact oriented, embedded \( (n-2) \)-dimensional submanifold, and let \( U\supset M \) be convex, bounded and open. Let \( Y \) be a smooth vector field on \( U \) supported in a neighborhood of \( M \) such that \( Y(p) \) is unit and normal to \( M \) for all \( p\in M \). Such a vector field exists since \( M \) necessarily has a trivial normal bundle (see, e.g. \cite{massey}.) Let \( F:[0,1]\times U\to U \) be a smooth null-homotopy to \( q\in U \) satisfying \( F_p'(0)=Y(p) \) for all \( p\in M \). Let \( \k_Y:=\k_F \) (See Definition \ref{def:cone}.)  
 
 	\begin{lem}
		\label{lem:gamm}
		The supports of \( P_Y\tM \in \hB_{n-2}^2(U)\) and \( E_Y\tM\in \hB_{n-1}^1(U) \) satisfy \( \supp(P_Y \tM) = \supp(E_Y \tM) = \supp(\tM) = M \).
	\end{lem}

	\begin{proof}
		Since \( \tM \) represents \( M \),	it follows from Proposition \ref{prop:submanifold} that \( M = \supp(\tM) \). Furthermore, \( \supp(P_Y \tM) \subset \supp(\tM) \) by Lemma \ref{lem:supportPE}. Conversely, let \( p \in M \), and let \( r>0 \). Since \( p\in \supp(\tM) \), there exists \( \o\in \B_{n-2}^1(U) \) with \( \supp(\o)\subset \O(p,r) \) and \( \cint_{\tM} \o\neq 0 \). A computation in coordinates \( (x_1, \dots, x_n) \) on \( \O(p,r) \) such that \( M\cap \O(p,r) \) can be written as the set \( x_{n-1} = x_n=0 \) (if necessary, shrink \( r \),) and such that \( Y(x_1, \dots, x_{n-2}, 0, 0) = \frac{\p}{\p x_{n-1}} \) shows that there exists \( \eta\in \B_{n-2}^2(U) \) supported in \( \O(p,r) \) such that \( \iota_Y d\eta (p', v)= \o(p',v)\) for any \( p'\in M \) and \( v\in \L^{n-2}T_{p'}M \). In particular, this shows by Stokes' Theorem that \( 0\neq\cint_{\tM} \o = \cint_{\tM} \iota_Y d\eta= \cint_{\tM} \Lie_Y \eta = \cint_{P_Y \tM}\eta \), and so \( p\in \supp(P_Y \tM) \). A similar computation shows \( \supp(E_Y \tM) =\supp(\tM) \).
	\end{proof}

	\begin{lem}
		\label{lem:fc}
		If \( \p S = P_Y \tM \), then \( S = \p\k_Y S + E_Y \tM \).
	\end{lem}

	\begin{proof}
		It follows from the continuity of \( \k_Y \) and \( E_Y \) and since \( E_Y^2 = 0 \) that \( \k_Y E_Y \tM = 0 \) (write \( \tM \) as a limit of Dirac chains \( A_i \) supported on \( M \) and compute directly \( \k_Y E_Y A_i=0 \).) The result then follows from the relations \( P_Y = E_Y \p + \p E_Y \) (Definition \eqref{def:preder}), and \( \k_Y \p + \p \k_Y = Id \) (Equation \eqref{eq:cone}).
	\end{proof}

	\begin{cor}
		\label{cor:SS}
		If \( \p S = P_Y \tM \), then \( M \subseteq \supp(S) \) and \( \supp(\p\k_Y S) \subset \supp(S) \).
	\end{cor}

	\begin{proof}
		The first inclusion follows from Lemmas \ref{lem:supportboundary} and \ref{lem:gamm}. The second inclusion follows from Lemma \ref{lem:fc} and the inclusion \( \supp(A+B)\subseteq \supp(A)\cup\supp(B) \).
	\end{proof}

	\begin{cor}
		\label{cor:key}
		Suppose \( S \) satisfies \( \p S = P_Y \tM \), \( \supp(\k_Y S)\subset U \) has empty interior, and \( M \subset \overline{\supp(\k_Y S)\setminus M} \). Then \( \supp(S) = \supp(\k_Y S) \).
	\end{cor}

	\begin{proof}
		We know \( \supp(\p\k_Y S) = \supp(\k_Y S) \) by Corollary \ref{cor:threechains} and \( \supp(\p\k_Y S) \subseteq \supp(S) \) by Corollary \ref{cor:SS}.

		Suppose \( p \in \supp(S)\setminus M \) and \( p \notin \supp(S - E_Y\tM) \). Then for each \( \e>0 \) such that \( \O(p,\e) \cap \supp(S - E_Y\tM) = \emptyset \), there exists an \( (n-1) \)-form \( \eta \) supported in \( \O(p,\e) \) such that \( \cint_S \eta \ne 0 \). But \( \cint_{S - E_Y\tM} \eta = 0 \) implies \( \cint_{E_Y\tM} \eta \ne 0 \). This shows that \( p \in \supp(E_Y\tM) = M \) by Lemma \ref{lem:gamm}, a contradiction.

		Now suppose \( p \in M \). By assumption, there exist \( p_i \in \supp(S) \setminus M \) with \( p_i \to p \). By the preceding argument, it follows that \( p_i \in \supp(S - E_Y\tM) \). Since this set is closed, we have \( p \in \supp(S - E_Y\tM) \).
	\end{proof}

\section{Borel measures and positive chains}
	\label{sub:positive_borel_measures}
	
	\subsection{The convex cone of positive chains}
		\label{sub:positive_chains}
		
		\begin{defn}
			\label{mass}
			Define the \emph{\textbf{mass}} of \( J \in \hB_k^1(U) \) to be the (possibly infinite) quantity \[ M(J) := \inf\{\liminf \|A_i\|_{B^0}: A_i \to J \mbox{ in } B^1, A_i \in \A_k(U) \}. \]
		\end{defn}

		\begin{prop}
			\label{prop:mass}
			If \( J \in \hB_k^1(U) \), then \( M(J) = \sup\left\{ \cint_J \o: \o \in \B_k^1(U), \|\o\|_{B^0} \le 1 \right\} \).
		\end{prop}

		\begin{proof}
			See \cite{whitney}, V, \S 16, Theorem 16A.
		\end{proof}

		\begin{cor}
			\label{cor:massdirac}
			If \( J \in \hB_k^0(U) \), then \( M(J) = \|J\|_{B^0} \).
		\end{cor}

		\begin{defn}
			A Dirac \( n \)-chain \( A = \sum (p_i;\a_i) \) is \emph{\textbf{positive}} if each \( n \)-vector \( \a_i \) is positively oriented with respect to the standard orientation of \( \R^n \).  We say \( J \in \hB_n^r(U) \) is \emph{\textbf{positive}} if \( J \) is a limit of positive Dirac \( n \)-chains \( A_i \to J \).
		\end{defn}

		\begin{lem}
			\label{lem:fdV}
			If \( r\ge 0 \) and \( J \in \hB_n^r(U) \) is positive, then \( \|J\|_{B^r} = \cint_J dV \).  If \( r = 1 \), then \( M(J) = \cint_J dV \).
		\end{lem}

		\begin{proof}
			Let \( A \) be a positive Dirac \( n \)-chain. By Proposition \ref{thm:isomorphism} and Corollary \ref{cor:massdirac}, \[ \|A\|_{B^r} = \sup_{\|\o\|_{B^r}=1}\, \cint_A \o \le \cint_A dV \le \|A\|_{B^r}. \] Since \( J \) is positive there exists a sequence positive Dirac \( n \)-chains \( \{A_i\} \) with \( J = \lim_{i \to \i} A_i \). Thus,
			\begin{align}
				\label{jdv}
				\cint_J dV = \lim_{i \to \i} \cint_{A_i} dV = \lim_{i \to \i} \|A_i\|_{B^r} = \|J\|_{B^r}.
			\end{align}

			For the second part let \( r = 1 \). By Definition \ref{mass}, Proposition \ref{prop:mass}, and Corollary \ref{cor:massdirac} we have \[ M(J) \le \liminf \|A_i\|_{B^0} = \liminf \int_{A_i} dV = \cint_J dV \le M(J). \]
		\end{proof}

		In particular, if \( J \in \hB_n^1(U) \) is positive, then \( M(J) < \i \).
		
		\begin{cor}
			\label{cor:positivecone}
			If \( U \) is bounded and \( J\in \hB_n^r(U) \) is positive, then \( J=u_n^{1,r}(J') \) where \( J'\in \hB_n^1(U) \) is positive.
		\end{cor}

		\begin{proof}
			Let \( A_i\to J \) be a sequence of positive Dirac chains converging to \( J \) in the \( B^r \) norm. By Lemma \ref{lem:fdV}, \( \|A_i\|_{B^r}=\|A_i\|_{B^0}=\cint_J dV\to \|J\|_{B^r} \), so \( \{A_i\}\subset \hB_n^0(U) \) is a bounded sequence. By Theorem \ref{thm:compactinclusion}, there exists a subsequence \( A_{i_j}\to J' \) in \( \hB_n^1(U) \). But \( u_n^{1,r}(A)=A \) for any Dirac chain \( A \), hence \( u_n^{1,r}(J')=J \) by continuity of \( u_n^{1,r} \).
		\end{proof}
		
		This lemma shows that all the positive chains lie in \( \hB_n^1(U) \). A similar argument shows:
		
 		\begin{cor}
			\label{cor:positivesequence}
 			If \( U \) is bounded and if \( J_i\to J \) in \( \hB_n^r(U) \) where \( J_i \) are positive, then \( J \) is positive, and writing \( J_i=u_n^{1,r}J_i' \), \( J=u_n^{1,r}J' \) as in Corollary \ref{cor:positivecone}, there is a subsequence \( J_{i_j}'\to J' \) in \( \hB_n^1(U) \).
 		\end{cor}

	\subsection{An isomorphism of convex cones}
		\label{sub:an_isomorphism_of_convex_cones}
		The next result is a special case of Theorem 11A of (\cite{whitney}, XI) which produces an isomorphism between sharp \( k \)-chains with finite mass and \( k \)-vector valued countably additive set functions (see \cite{whitney}, XI, \S 2). We specialize to positive chains and Borel measures in the case \( k = n \).

		Let \( \C(U) \) be the convex cone consisting of positive elements \( J \) of \( \hB_n^1(U) \) such that \( \supp(J)\subset U \). Let \( \M(U) \) be the convex cone consisting of all finite Borel measures \( \mu \) on \( U \) such that \( \supp(\mu)\subset U\subset \R^n \) is closed as a subset of \( \R^n \).

 		\begin{prop}
			\label{prop:whitney}
			There exists a bijection \( \Psi: \C(U) \to \M(U) \) with \( J \mapsto \mu_J \) satisfying
			\begin{enumerate}
				\item \label{samenorms} \( \mu_J(U) = \|J\|_{B^1} = \cint_J dV = M(J) \) for all \( J\in \C(U) \);
				\item \label{positivelinear} \( \mu_{aJ+bK}=a\mu_J+b\mu_K \) for all non-negative \( a,b\in \R \) and \( J,K\in \C(U) \);
				\item \label{mumass} \( \mu_A(X) = M(A\lfloor_X) \) for all \( A \in \A_n(U) \) and Borel sets \( X \subset U \);
				\item \label{sameintegrals} \( \int_U \, f\, d \mu_J =  \cint_J \, f\, dV  \) for all \( f \in \B_0^1(U) \) and \( J\in \C(U) \);  
				\item \label{samesupports} \( \supp(J) = \supp(\mu_J) \) for all \( J\in \C(U) \).
			\end{enumerate}
 		\end{prop}
		\begin{prop}
			\label{prop:maggi}
		  	Let \( J_i \to J \in \C(U).\) Then
			\begin{enumerate}
				\item \( \mu_{J_i} \) converges weakly to \( \mu_J \);
				\item If \( F \) is closed and \( G \) is open, then \( \mu_J(F) \ge \limsup \mu_{J_i}(F) \) and \( \mu_J(G) \le \liminf \mu_{J_i}(G) \);  
				\item If \( E \) is a Borel set with \( \mu_J(\fr\,E) = 0 \), then \( \mu_J(E) = \lim \mu_{J_i}(E) \).
			\end{enumerate} 
		\end{prop} 

		\begin{proof}
		 It follows from Proposition \ref{prop:whitney} \ref{sameintegrals} that \( \int_U \, f\, d \mu_{J_i} \to \int_U \, f\, d \mu_J \) for all \( f\in \B_0^1(U) \). Thus, the Portmanteau theorem\footnote{See, for example, \cite{maggi}, Proposition 4.26.} gives (a)-(c).
		\end{proof}
		
		\begin{defn}
			\label{def:compatible}
			Let \( J \in \C(U) \). A Borel set \( B\subset U \) is \emph{\textbf{\( J \)-compatible}} if \( B \) is a continuity set for \( \mu_J \), i.e. if \( \mu_J(\fr\, B) = 0 \).
		\end{defn}

		Since \( \mu_J \) is finite, it follows that for every \( p\in U \), all but countably many cubes and balls centered at \( p \) are \( J \)-compatible.

		If \( f \) is a bounded positive Borel function and \( J \in \C(U) \) let \( f \cdot J := \Psi^{-1} (f \mu_J) \in \C(U). \) Then \( \|f \cdot J\|_{B^1} = \int_U f d\mu_J \) by Proposition \ref{prop:whitney} \ref{samenorms}. If \( f \in \B_0^1(U) \), then \( f \cdot J = m_f J \) where \( m_f \) is the continuous operator ``multiplication by a function'' (\cite{whitney}, VII, \S 1 (4), XI \S 12 (3), \cite{OC} Theorem 6.1.3.)
  
\section{Basic results of Hausdorff and Hausdorff spherical measures}
	\label{sec:standard_results}
	\setcounter{thm}{0}
	Most of the results in the section are due to Besicovitch \cite{besicovitch,besicovitchI} and are reproduced in the first set of lemmas in \cite{reifenberg}, pp. 11-12. Let \( \diam(X) \) denote the \emph{\textbf{diameter}} of \( X \subset \R^n \) and let \( \a_m \) denote the Lebesgue measure of the unit \( m \)-ball in \( \R^m \). For any \( X \subset \R^n \), \( 0 \le m \le n \) and \( \d > 0 \), define \[ \H_\d^m(X) := \inf \left\{\sum_{i=1}^\i \a_m\left(\frac{\diam(Y_i)}{2}\right)^m: X \subset \bigcup_{i=1}^\i Y_i, Y_i \subset \R^n, \diam(Y_i) < \d \right\}. \] Then \( \H^m(X) := \lim_{\d \to 0}\H_\d^m(X) \) is the \textbf{\emph{(normalized) Hausdorff \( m \)-measure}} of \( X \). Closely related to Hausdorff measure is \emph{\textbf{(normalized) Hausdorff spherical measure}} \( \sp(X) \). This is defined in the same way, except that the sets \( Y_i \) are required to be \( n \)-balls. In particular, \( \H^m(X) \le \sp^m(X) \le 2^m \H^m(X) \). Both are metric outer measures, and so the Borel sets are measurable. 
 	
	For \( m \)-rectifiable sets (see \cite{mattila} Definition 15.3), \( \H^m \) and \( \sp^m \) are the same (Theorem 3.2.26 of \cite{federer}.) Most results which hold for one of these measures also holds for the other. There are two notable exceptions, giving Hausdorff spherical measure the edge for the calculus of variations in situations where it might be undesirable to assume that sets being considered are rectifiable a priori (See Lemmas \ref{lem:3} and \ref{lem:3H}, as well as Lemma \ref{lem:4}.)
 
	\begin{lem}[Vitali Covering Theorem for Hausdorff Measure]
		\label{lem:1}
		Suppose \( \H^m(X)<\i \) and \( \frak{B} \) is a collection of closed \( n \)-balls such that for each \( x\in X \) and \( \e>0 \) there exists a ball \( \bar{\O}(x,\d)\in \frak{B} \) with \( \d<\e \). Then there are disjoint balls \( B_i\in \frak{B} \) such that \( \H^m(X\setminus \cup_i B_i) = 0 \).
	\end{lem}

	\begin{proof}
		Apply Theorem 2.8 of \cite{mattila} to the measure \( \mu=\H^m\lfloor_X \), which is a finite Borel measure, hence a Radon measure by Theorem 7.8 of \cite{folland}.
	\end{proof}

	Since \( \H^m \) and \( \sp^m \) are comparable, one can replace \( \H^m \) with \( \sp^m \) in Lemma \ref{lem:1}.
	
	\begin{lem}
		\label{lem:2}
		Suppose \( X\subset \O(X,\e)\subset Y\subset \R^n \) for some \( \e>0 \), and \( f:Y \to Y \) is Lipschitz with Lipschitz constant \( \l \). Then \( \H^m(f(X)) \le \l^m \H^m(X) \) and \( \sp^m(f(X)) \le \l^m \sp^m(X) \).
	\end{lem}

	\begin{lem}
		\label{lem:3}
		Suppose \( \sp^m(X) < \i \). For \( \sp^m \) almost every \( p \in X \), \[ \limsup_{r \downarrow 0}\frac{\sp^m(X \cap \bar{\O}(p,r))}{\a_mr^m} = 1. \]
	\end{lem}
	See \cite{mattila} Theorem 6.6.

	Contrast this with the corresponding result for Hausdorff measure:

	\begin{lem}
		\label{lem:3H}
		Suppose \( \H^m(X) < \i \). For \( \H^m \) almost every \( p \in X \), \[ 2^{-m} \le  \limsup_{r \downarrow 0}\frac{\H^m(X \cap \bar{\O}(p,r))}{\a_mr^m} \le 1. \]
	\end{lem}

	See \cite{mattila} Theorem 6.2.

	\begin{lem}
		\label{lem:4}
		Let \( X_h \) be the set of points of \( X \) at a distance \( h \) from a fixed affine subspace of dimension \( m \) in \( \R^n \), \( 0\leq m<n \). Then
		\[ \int_0^\i \sp^{m-1}(X_h) dh \le \sp^m(X). \]
	\end{lem}

	This is Lemma 4 in \cite{reifenberg}, p. 12. The inequality does not hold if we replace Hausdorff spherical measure with Hausdorff measure. A counterexample and discussion can be found in \cite{ward}.

	\begin{lem}
		\label{lem:5}
		If \( X \subset \bar{\O}(p,r) \), then the cone \( C_p(X) \) satisfies \[ \sp^m(C_p(X)) \le r 2^m \frac{\a_m}{\a_{m-1}}\sp^{m-1}(X). \] If \( X \) is an \( (m-1) \)-dimensional polyhedron, then \[ \H^m(C_p(X)) \le \frac{r}{m}\H^{m-1}(X). \]  
	\end{lem}

	See Lemmas 5 and 6 in \cite{reifenberg}. More generally:
	
	\begin{lem}
		\label{lem:6}
		Suppose \( \Pi \) is a \( k \)-dimensional affine subspace of \( \R^n \), \( 0\leq k <n \), and \( X \subset \R^n \) a compact set. Let \( C = \cup_{x \in X} I_x \) where \( I_x \) is the line connecting \( x \) to its orthogonal projection on \( \Pi \). If \( X \subset \O(\Pi,r) \), then \[ \sp^m(C) \le 2^m \frac{\a_m}{\a_{m-1}} r\sp^{m-1}(X). \]
	\end{lem}

	See Lemma 6 of \cite{reifenberg}. 

\section{Spanning Sets}
	\label{sec:spanning_sets} 
	\setcounter{thm}{0}
	Please refer to the introduction for our definition of ``span.'' Here we give some examples and prove some results needed to work with this definition. In this section, \( M \) is a compact oriented, embedded \( m-1 \)-dimensional submanifold. Unless otherwise noted, we assume \( 1\leq m\leq n \).
 
	\begin{prop}
		\label{prop:noreifenberg}
		For any group \( G \) of coefficients (compact or otherwise,) the kernels of the homomorphisms \( {\iota_j}_* : \check{H}_1(M;G)\to \check{H}_1(X_j;G) \) induced by the injection maps \( \iota_j: M\hookrightarrow X_j \), \( j=1,2,3 \) in Figure \ref{fig:catenoid} have a trivial intersection. Thus, there is no non-trivial collection of Reifenberg spanning sets which contains all three simultaneously. 
	\end{prop}

	\begin{proof}
		The relevant groups are: \( \check{H}_1(M;G)\simeq G\oplus G\oplus G \); \( \check{H}_1(X_j;G)\simeq G \), \( j=1,2,3 \). The homomorphism \( {\iota_1}_* \) is given by \( {\iota_1}_*(a,b,c)=a+b+c \), likewise \( {\iota_2}_*(a,b,c)=a+b \), and \( {\iota_3}_*(a,b,c)=b+c \). Thus, if \( (a,b,c) \) is in the kernel of all three homomorphisms, then \( a=b=c=0 \).
	\end{proof}
	 
	\begin{lem}
		\label{prop:cone2}
		The cone \( C_q(M) \) spans \( M \) for all \( q \in \R^n \).
	\end{lem}

	\begin{proof}
		If not, there is a simple link \( N \) of \( M \) such that \( N\cap C_q(M)=\emptyset \). Let \( M_i \) be the connected component of \( M \) with \( L(M_i,N)\neq 0 \), and recall \( L(M_i,N) \) is the degree of the map
		\begin{align*} 
			f: M_i\times N &\to S^{n-1}\\
			(s,r) &\mapsto \frac{s-r}{\|s-r\|}.
		\end{align*}
		Since \( C_q(M) \) is disjoint from \( N \), the map \( f \) factors as 
		\[
			f=g\circ h: M_i\times N \mathop{\longrightarrow}^h C_q(M)\times N \mathop{\longrightarrow}^g S^{n-1}
		\]
		where \( h(s,r)=Id \) and \( g(p,r)=(p-r)/\|p-r\| \). So, it suffices to show that \( h_* \) is trivial on \( H_{n-1}(M_i\times N; \Z)\simeq H_{m-1}M_i \otimes H_{n-m} N \). Write \( h=\iota\times Id \), where \( \iota: M_i\to C_q(M) \) is given by \( \iota(s)=s \). Since \( \iota_* :H_{m-1} (M_i;\Z)\to H_{m-1}C_q(M);\Z) \) is itself trivial, it follows from naturality of the K\"{u}nneth formula that \( h_* =0\).
	\end{proof}

	In the above lemma, the only property of \( C_q(M) \) needed is that the map on homology \( i_*: H_{m-1}(M;\Z) \to H_{m-1} (C_q(M);\Z) \) induced by the inclusion \( \iota:M\hookrightarrow C_q(M) \) is trivial. So, any other compact set \( X \) containing \( M \) with that condition on homology will also span \( M \). Note that \(M\subset X \) if \( X \) spans \( M \). We can say slightly more, that \( M\subset \overline{X\setminus M} \). 
	
	Now suppose \( M \) is an \( n-2) \)-sphere in \( \R^n \). A compact set \( X \) spans \( M \) in the sense of Reifenberg's Theorem 2 of \cite{reifenberg} if \( X \) contains \( M \) and there is no retraction of \( X \) onto \( M \). The next proposition shows that spanning in the sense of Reifenberg's Theorem 2 implies spanning in our sense.
	
	\begin{prop}                   
		\label{prop:equivalences} 
		Suppose \(  M\subset \R^n \) is homeomorphic to an \( (n-2) \)-sphere and \( X\supset M \) is compact, of dimension \( ≤ n-1 \) (or sufficiently, of finite Hausdorff \( (n-1) \)-measure.) If there is no retraction \( X \to M \), then \( X \) spans \( M \).
 	\end{prop}
	
	\begin{proof} 
	If  \( X \) does not span \( M \), there exists a simple link \( \eta \) of \( M \) which does not intersect \( X \). The curve \( \eta \) defines a homology class in \( H_1(\R^n\setminus X) \) and therefore by Alexander duality a \v{C}ech cohomology class in \( H^{n-2}(X) \) which extends the fundamental cocycle of \( M \). Therefore there exists a retraction \( X\to M \) by the Hopf extension theorem.
	\end{proof}

	\begin{prop}
		\label{lem:unrectifiable}
		Let \( m=n-1 \). If \( X = Z \sqcup L \) is compact and spans \( M \) where \( Z \) is closed and \( L \) is purely \( m \)-unrectifiable, then \( Z \) also spans \( M \).
	\end{prop}

	\begin{proof}
		Suppose \( N \) is a simple link of \( M \) and \( N \cap Z = \emptyset \). Choose a tubular neighborhood \( T \) of \( N \) with \( Z \cap T = \emptyset \). Parameterize \( N \) by \( \Theta: [0,2\pi]\to N\simeq S^1 \), where \( \Theta(x) \) gives the \( \theta \)-coordinate of \( x\in N\simeq S^1\subset \R^2 \) in polar coordinates, where \( S^1\ \) is the unit circle in \( R^2 \). The pullback bundle of the tubular neighborhood via \( \Theta \) is isomorphic to a cylinder \( C \). Let \( \tilde{\Theta}:C\to N \) denote the resulting map, which is a diffeomorphism except at the endpoints, where the map is 2-1. Let \( H=\tilde{\theta}^{-1}(L\cap T) \). Since \( \tilde{theta} \) is Lipschitz, it follows that \( H \) is also purely \( m \)-unrectifiable, and so by the Besicovitch-Federer projection theorem (\cite{mattila} 18.1,) almost every orthogonal projection \( \pi:H \to K \) where \( K \) is an \( m \)-plane sends \( H \) to a set with Lebesgue \( m \)-measure zero. Thus, there exists an arc \( \rho\subset T \) which is disjoint from \( L \cap T \), whose endpoints \( p, q \) lie in the same fiber of \( T \) and which is given as a smooth section of the normal bundle to \( N \), except at the endpoints, where the section is discontinuous.

		Let \( \s \subset D \) be the line segment joining \( p  \) and \( q \). Since \( X \) is closed and \( \rho \cap X = \emptyset \), there exists \( \e>0 \) such that \( \O(p,\e) \cup \O(q,\e)\subset T \) is disjoint from \( L \). Let \( T'\subset T \) be the \( \e \)-neighborhood of \( \s \), and apply the projection theorem again to find a line segment \( \s' \) disjoint from \( L \cap T \) with one endpoint in \( \O(p,\e) \) and the other in \( \O(q,\e) \). Finally connect the appropriate endpoints of \( \rho \) and \( \s' \) inside \( \O(p,\e) \) and \( \O(q,\e) \) to create a simple link of \( M \) which is disjoint from \( X \).
	\end{proof} 
	
	\begin{defn}
		\label{def:competitor}
		Suppose \( X \) is compact and spans \( M \). A \emph{\textbf{competitor of \( X \) with respect to \( M \)}} is a set \( \phi(X) \) where \( \phi: \R^n \to \R^n \) is a Lipschitz map that is the identity on \( B^c \) and \( \phi(B) \subset B \), where \( B \subset \R^n \) is some closed ball disjoint from \( M \). The map \( \phi \) is called a \emph{\textbf{deformation of \( X \)}}.
	\end{defn}

	\begin{thm}
		\label{thm:lipspan}  
		If \( m=n-1 \) and \( X \subset \R^n \) is compact and spans \( M \), and \( \phi(X) \) is a competitor of \( X \), then \( \phi(X) \) is compact and spans \( M \).
	\end{thm}

	\begin{proof}
		If not, then there is a simple link \( \eta \) of \( M \) such that \( \eta(S^1) \cap \phi(X)=\emptyset \). Let \( B=B_r \) be a closed ball of radius \( r \) satisfying \( M\cap B=\emptyset \), \( \phi\equiv Id \) on \( B^c \) and \( \phi(B)\subseteq B \). Expanding \( B \) slightly, we may also assume \( \phi\equiv Id \) on an \( \e \)-neighborhood of \( \fr\,B \) and that \( \phi(B_{r-\e})\subseteq B_{r-\e} \). We can assume without loss of generality that \( \phi \) is smooth: Approximate \( \phi \) uniformly to within \( \min\{\e/2, d(\eta(S^1),\phi(X))/2 \} \) by a smooth function \( \tilde{\phi} \). Let \( \{f,g\} \) be a partition of unity subordinate to \( \{\mathring{B}_{r-\e/2}, B_{r-\e}^c\} \), and define \( \hat{\phi}=f\cdot\tilde{\phi}+g\cdot Id \). Then \( \hat{\phi} \) is smooth, \( \hat{\phi}(X)\cap \eta(S^1)=\emptyset \), \( \hat{\phi}\equiv Id \) on \( B^c \) and \( \hat{\phi}(B)\subseteq B \). Finally, expand \( B \) slightly again and we may assume \( \hat{\phi}\equiv Id \) on an \( \e' \)-neighborhood of \( \fr\,B \) and that \( \hat{\phi}(B_{r-\e'})\subseteq B_{r-\e'} \).

		By compactness, there exists a tubular neighborhood \( T \) of \( \eta(S^1) \) so that \( \bar{T}\cap \hat{\phi}(X)=\emptyset \). Since we may perturb \( \eta \) within \( T \), let us assume without loss of generality that the intersection \( \eta(S^1)\cap \fr\,B \) is transverse. By considering a possibly smaller \( \e' \) we may assume slightly more, that within the \( \e' \)-neighborhood of \( \fr\,B \), the curve \( \eta(S^1) \) consists of radial line segments. So, the set \( \eta(S^1)\cap B \) consists of a finite collection of arcs \( \{\eta_1, \dots, \eta_N\} \), each of which intersects \( \fr\,B \) radially. Note that \( T\cap \mathring{B} \) consists of pairwise disjoint neighborhoods \( T_i \) of the arcs \( \eta_i \) (minus their endpoints.) For each \( i \), let \( \pi_i: T_i \to D \) be the projection onto the normal \( m \)-disc \( D \) determined by the tubular neighborhood \( T \) and some trivialization of the normal bundle of \( \eta(S^1) \). Consider the smooth maps \( \psi_i:=\pi_i\circ\hat{\phi} : W_i\to D \) where \( W_i:=\hat{\phi}^{-1}(T_i) \). For each \( i \) fix a regular value \( x_i \) of \( \psi_i \) and consider the compact sets \( Y_i:=\pi_i^{-1}(x_i)\cap B_{r-\e'/2} \). Since \( \hat{\phi} \) is proper, each preimage \( Z_i:=\hat{\phi}^{-1}(Y_i) \) is also compact as a subset of \( \R^n \).

		For each \( i \), there exists by compactness of \( Z_i \) a smooth compact \( n \)-manifold with boundary \( P_i\subset W_i \), such that \( Z_i \subset P_i \), and such that both line segments constituting \( \pi_i^{-1}(x_i) \setminus B_{r-\e'} \) meet \( \p P_i \) transversally and at one point each. Now \( x_i \) is still a regular value of the restricted map \( \psi_i |_{P_i}:P_i\to D\). But also \( x_i \) is a regular value of \( \psi_i|_{\p P_i} \), since \( \psi_i|_{\p P_i}^{-1}(x_i)= (\pi_i^{-1}(x_i) \setminus B_{r-\e'})\cap \p P_i \), and this intersection is transverse and contained in a region on which \( \hat{\phi}\equiv Id \). So, by the inverse function theorem for manifolds with boundary (\cite{hirsch}, Theorem 4.1) the set \( \psi_i|_{P_i}^{-1}(x_i) \) is a collection of circles and arcs that meet the boundary \( \p P_i \) ``neatly.'' Since by construction the set \( \psi_i|_{\p P_i}^{-1}(x_i) \) consists of exactly two points, there is exactly one arc \( \beta_i \) in \( \psi_i^{-1}(x_i) \), and it joins the two points. Moreover, the arcs \( \beta_i \) are pairwise disjoint, since any point in an intersection of two must get mapped by \( \hat{\phi} \) into disjoint tubular neighborhoods.

		Furthermore, the arcs \( \beta_i \) terminate inside the tubular neighborhood \( T \) of \( \eta \). So, we may smoothly extend each \( \beta_i \) within \( T \), linking the arcs together to form an embedding \( \beta: S^1\to M^c \). The curve \( \beta \) is also a simple link of \( M \), since \( \beta \) is by construction regularly homotopic to \( \eta \). Thus, \( \b(S^1)\cap X \neq \emptyset \), and so \( \hat{\phi}(\b(S^1))\cap \hat{\phi}(X)\neq \emptyset \). But \( \hat{\phi}(\b(S^1))\subset T \), yielding a contradiction.
	\end{proof}

	The \emph{\textbf{core}} of \( X\subset \R^n \) is defined by \[ X^*:= \{p \in X \,|\, \H^m(X \cap \O(p,r)) > 0 \text{ for all } r > 0 \}. \] The core \( X^* \) of a closed set \( X \) is closed, and \( \H^m(X\setminus X^*) = 0 \). The definition of core and its properties are unaltered if \( \H^{n-1} \) is replaced by \( \sp^{n-1} \). We say \( X \) is \emph{\textbf{reduced}} if \( X = X^* \).

	\begin{lem}
		\label{lem:corespan}
		Let \( m=n-1 \). If  \( X \subset \R^n \) is compact and spans \( M \), then the core \( X^* \) is compact and spans \( M \).
	\end{lem}

	\begin{proof}
		If not, there is a simple link \( N \) of \( M \) and a tubular \( \e \)-neighborhood \( T \) of \( N \) whose closure is disjoint from \( X^* \). So, \( \H^m(X \cap T ) = 0 \). Let \( \pi: T \to D \) be the projection onto the normal \( m \)-disk \( D \) determined by the tubular neighborhood \( T \) and some trivialization of the normal bundle of \( N \). Then \( \H^m(\pi(X\cap T))=0 \), and so there exists a point \( x\in D\setminus (\pi(X\cap T)) \). Then \( \pi^{-1}(x) \) is a simple link of \( M \) missing \( X \), giving a contradiction.
	\end{proof}
	
	The next lemma is an easy consequence of the definition of the Hausdorff metric:

	\begin{lem}
		\label{lem:hausdorffspan}
		If \( X_i\subset \R^n \) is compact and spans \( M \) for each \( i \) and \( X_i\to X \) in the Hausdorff metric, then \( X \) spans \( M \).
	\end{lem} 

	\begin{lem}
		\label{lem:cutandcone}
		Let \( m=n-1 \). Suppose \( X\subset \R^n \) is compact and spans \( M \).  Let \( p\in \R^n \), \( r>0 \) and suppose \( \O(p,r)\cap M=\emptyset \).  If \( p'\in \bar{\O}(p,r) \), then \( X' := (X \cap \O(p,r)^c) \cup C_{p'}(X \cap \fr\,\O(p,r)) \) is compact and spans \( M \).
	\end{lem}

	\begin{proof}
		It is clear that \( X' \) is compact. If \( X' \) does not span \( M \), then let \( \eta \) be a simple link of \( M \) such that \( \eta(S^1)\cap X'=\emptyset \). Then \( \eta(S^1)\cap X\subset \O(p,r) \). First, we observe that \( X\cap \fr\, \O(p,r)\neq \emptyset \). If not, there is by compactness an \( \e \)-neighborhood \( T \) of \( \fr\,\O(p,r) \) disjoint from \( X \), so we may construct a diffeomorphism of \( \O(p,r) \) which fixes \( \fr\,\O(p,r) \) and sends \( \eta \) to a simple link \( \eta' \) of \( M \) such that \( \eta'(S^1)\cap \O(p,r)\subset T \). Thus, \( \eta'(S^1)\cap X =\emptyset \), yielding a contradiction.
		
		This implies that \( C_{p'}(X\cap \fr\,\O(p,r))\neq \emptyset \), and in particular, \( p'\in X' \). Thus, \( p'\notin \eta(S^1) \). Let \( \d>0 \) such that \( \O(p',2\d)\cap \eta(S^1)=\O(p',2\d)\cap M=\emptyset \). Let \( \rho: \bar{\O}(p',\d)^c\to \R^n \) be the identity on \( \O(p,r)^c \) and elsewhere the radial projection away from \( p' \) and onto the frontier of \( \O(p,r) \). By definition of \( C_{p'} \), it is enough to show \( \rho(\eta(S^1)) \) intersects \( X \).
		
		Suppose not. Since \( \rho(\eta(S^1)) \) is compact, there is an \( \e' \)-neighborhood \( T' \) of \( \rho(\eta(S^1)) \) such that \( T'\cap X=\emptyset \). Let \( V \) be a smooth radial vector field on \( \R^n \), with center point \( p' \), supported in \( \O(p,r) \), and normalized so that if \( \phi \) is the time-\( 1 \) flow of \( V \), then \( \phi(\eta(S^1))\subset T' \). Then \( \phi\circ\eta \) is a simple link of \( M \) disjoint from \( X \), a contradiction.
	\end{proof}
	
\section{Film chains}
	\label{sec:film_chains}  
	\setcounter{thm}{0}
	Throughout the rest of this paper, fix \( Y, U, M \) be as in \S\ref{sub:constructions_with_curves}.  Let \( \Span^*(M,U) \) denote the collection of reduced compact subsets of \( U \) with finite \( \sp^{n-1} \)-measure and which span \( M \). 
	
	\begin{defn}
		\label{def:filmchain}
		A \emph{\textbf{film chain}} is an element \( S \in \hB_{n-1}^2(U) \) satisfying
		\begin{enumerate}   
			\item\label{def:filmchain:item:1} \( \p S = P_Y \tM \);
			\item\label{def:filmchain:item:2} \( \k_Y S \in \C(U) \);
			\item\label{def:filmchain:item:3} \( \mu_{\k_Y S}=\sp^{n-1}\lfloor_X \) for some \( X \in \Span^*(M,U) \).
		\end{enumerate}
		 Let \( \F(M, Y, U) \) denote the collection of all film chains\footnote{The choice of \( Y \), as a smooth line field, is often naturally determined, as in \ref{fig:Figures_YProblem2} and \ref{fig:moebiusdipole}.}. 
	\end{defn}
		   
	   Dipole surfaces (see \cite{plateau10}) are examples of film chains. 
	\begin{thm}
		\label{thm:setchaincorrespondance}
		If \( X \in \Span^*(M,U) \), then there exists a unique film chain \( S_X \in \F(M, Y, U) \) with \( \supp(S_X) = X \). Moreover, \( S_X \) satisfies \( \supp(\k_Y S_X) = \supp(S_X) \) and \(\supp(\p S_X) = M \).
	\end{thm}
 
	\begin{proof}
		Let \( J \in \C(U) \) correspond to the measure \( \sp^{n-1}\lfloor_X \), and set \( S_X=\p J+E_Y\tM \). Then \( \p S_X=P_Y\tM \) since \( P_Y=\p E_Y \). As in the proof of Lemma \ref{lem:fc}, we have \( \k_Y S_X=\k_Y(\p J+E_Y \tM)=\k_Y \p J = J \), since \( \k_Y \p + \p \k_Y = Id \), and \( J \) is top dimensional, showing \( \k_Y J = 0 \). This establishes existence. For uniqueness, suppose \( S_X' \) also satisfies \( \p S_X' = P_Y\tM \) and \( \mu_{\k_Y S_X'}=\sp^{n-1}\lfloor_X \). Then by Proposition \ref{prop:whitney}, \( \k_Y S_X'=\k_Y S_X \). Thus, \( S_X=\p \k_Y S_X + \k_Y \p S_X = \p \k_Y S_X' + \k_Y P_Y\tM = S_X' \).

		To see that \(\supp(S_X) = \supp(\k_Y S_X) = X \), we apply Corollary \ref{cor:key} and Proposition \ref{prop:whitney} \ref{samesupports}, noting that \( \supp(\k_Y S_X)=\supp(\sp^{n-1}\lfloor_X)=X \) since \( X \) is closed and reduced. The last equality is Lemma \ref{lem:gamm}.
 	\end{proof}

	By Theorem \ref{thm:setchaincorrespondance}, there is a 1-1 correspondence between \( \F(M, Y, U) \) and \( \Span^*(M,U) \). Also by Theorem \ref{thm:setchaincorrespondance}, film chains \( S\in \F(M, Y, U) \) automatically satisfy \( \supp(S)\in \Span^*(M,U) \), \( \supp(S) = \supp(\k_Y S) \) and \( \supp(\p S) = M \).

	It is with this construction in mind that we define a continuous area functional on \( \hB_{n-1}^2(U) \):
	\begin{defn}
		\label{def:area}
	 	For \( S \in \hB_{n-1}^2(U) \), define \[ \A^{n-1}(S) = \A_Y^{n-1}(S) := \cint_{\k_Y S} dV. \]
	\end{defn}
	
	According to Corollary \ref{cor:key} and Proposition \ref{prop:whitney} \ref{samenorms},
	\begin{equation} 
		\label{eq:sp}
		\sp^{n-1}(\supp(S)) =  \A^{n-1}(S)
	\end{equation}
	for all \( S\in \F(M, Y, U) \).
	
	\begin{defn}
		Let \[ \frak{m} := \inf\{\sp^{n-1}(X): X \in \Span^*(M,U)\}. \] We say that a sequence \( \{X_k\} \subset \Span^*(M,U) \) is \emph{\textbf{minimizing}} if \( \lim \sp^{n-1}(X_k) = \frak{m} \).  
	\end{defn}

	Note that by Lemma \ref{lem:2}, \( \frak{m} \) is independent of our choice of \( U \). Also, \eqref{def:filmchain:item:3} implies that
	\begin{equation}
		\label{eq:frakm}
		\frak{m}  = \inf\{\A^{n-1}(S): S\in \F(M, Y, U)\}.
	\end{equation}
	
	A sequence of film chains \( \{S_k\} \subset \F(M,Y, U) \) is \emph{\textbf{\( \A^{n-1} \)-minimizing}} if \( \A^{n-1}(S_k) \to \frak{m} \).

	\begin{prop}
		\label{prop:minimalarea}
		There exists a constant \( a_0 > 0 \) such that \( \A^{n-1}(S) > a_0 \) for all \( S \in \F(M, Y, U) \).
	\end{prop}
		
	\begin{proof}  
		Let \( N=\g(S^1) \) be a simple link of \( M \) and let \( T \) be an \( \e \)-tubular neighborhood of \( N \) whose closure is disjoint from \( M \). Let \( D \) be the \( (n-1) \)-disk of radius \( \e \) and \( \rho: T \to D \) the canonical projection given by \( T \) and some trivialization of the normal bundle of \( N \). It follows from Lemma \ref{lem:2} that there exists a constant \( C > 0 \) such that if \( B\subset T \), then \( \sp^{n-1}(\rho(B))\le C\sp^{n-1}(B) \).

		Suppose \( X \) is compact and spans \( M \). Then \( \rho(X\cap T)=D \), otherwise the preimage \( \rho^{-1}(x) \) of a point \( x\in D\setminus \rho(X\cap T) \) would be a simple link of \( M \), and missing \( X \). Putting this together, we get \( \sp^{n-1}(X)\ge \sp^{n-1}(X\cap T)\ge \frac{1}{C}\sp^{n-1}(\rho(X\cap T))=\frac{1}{C}\a_{n-1} \e^{n-1} \). The result follows from \eqref{eq:sp}, setting \( a_0=\frac{1}{C}\a_{n-1} \e^{n-1} \).
 	\end{proof}   	

	Let \( c_0 = \diam(U) \vol(M) \), where \( \vol(M) \) is the volume of \( M \) and \( \diam(U) \) is the diameter of \( U \).

	\begin{prop}
		\label{lem:conespanner}
		There exists \( S \in \F(M, Y, U) \) such that \( \A^{n-1}(S) \le c_0 \).
	\end{prop}

	\begin{proof}
		We know \( C_q(M) \) spans \( M \) by Lemma \ref{prop:cone2}. Since \( \sp^{n-1}C_q(M) \le R \cdot \vol(M) < \i \), we may apply Theorem \ref{thm:setchaincorrespondance} to the set \( C_q(M)^*= C_q(M)\) to obtain a film chain \( S = S_{C_q(M)} \). Furthermore, \( \A^{n-1}(S) = \sp^{n-1}(C_q(M)) \le c_0 \).
	\end{proof}

	\begin{defn}
		\label{def:t2def}
		Define \[ \T(M, Y, U) := \{S \in \hB_{n-1}^2(U): S = \lim_{i \to \i} S_i,\,\, S_i \in \F(M, Y, U),\, \A^{n-1}(S_i) \le c_0 \text{ for all } i \} \] and let \( \C(U,c_0) := \{ T \in \C(U) : M(T) \le c_0 \} \).
	\end{defn}

	\begin{prop}
		\label{prop:t2complete}
		If \( S \in \T(M, Y, U) \), then

 	 	\begin{enumerate}
			\renewcommand{\theenumi}{(\alph{enumi})}
			\renewcommand{\labelenumi}{\theenumi}
			\makeatletter
			\renewcommand{\p@enumii}{\theenumi}
			\makeatother
 			\item \label{item:bddarea} \( \A^{n-1}(S) \le c_0 \);
			\item \label{item:bdry} \( \p S = P_Y \tM \);
			\item \label{item:positive} \( \k_Y S\in \C(U, c_0) \);
			\item \label{item:support} \( \supp(\k_Y S) \) spans \( M \).
		\end{enumerate}
	\end{prop}

	\begin{proof}
		Choose \( S_i \to S \) in \( \T(M, Y, U) \) with \( S_i \in \F(M, Y, U) \) and \( \A^{n-1}(S_i) \le c_0 \). Parts \ref{item:bddarea} and \ref{item:bdry} follow from continuity of \( \k \), \( \p \) and \( \cint \). A diagonal sequence of positive Dirac chains \( A_{i,j} \) approximating \( \k_Y S_i \) shows \( \k_Y(S) \) is positive. Since \( U \) is bounded, Corollary \ref{cor:positivecone} shows\footnote{Here we think of \( \hB_n^1(U) \) as a subspace of \( \hB_n^2 \) via the canonical injection map \( u_n^{1,2} \).} that \( \k_Y S\in \hB_n^1(U) \) and hence by Lemma \ref{lem:fdV} that \( M(\k_Y S)=\int_{\k_Y S}dV=\lim_{i\to\i}M(\k_Y S_i)\le c_0 \).

		Proof of \ref{item:support}: If not, there is a simple link \( N=\eta(S^1) \) of \( M \) and an \( \e \) tubular neighborhood \( T=\O(N,\e) \) of \( N \) such that \( \supp(\k_Y S) \cap T = \emptyset \). Let \( f: U \to \R \) be a smooth function satisfying \( f(x) = 1 \) for \( x\in \O(N,\e/2) \), \( \supp(f)\subset T \) and \( f \ge 0 \). Then by Proposition \ref{prop:whitney} \ref{sameintegrals},
		\begin{align*}
			\cint_{\k_Y S} f dV = \lim_{i \to \i} \cint_{\k_Y S_i} f dV &\ge \limsup_{i \to \i} \mu_{\k_Y S_i} (\O(N,\e/2))\\
			&= \limsup_{i \to \i} \sp^{n-1}(\supp(S_i) \cap \O(N,\e/2))\\
			&\ge \a_{n-1}(\e/2)^{n-1},
		\end{align*}
		where the last inequality follows since \( \supp(S_i) \) spans \( M \). This yields a contradiction, since \( f dV \) is supported away from \( \supp(\k_Y S) \).
	\end{proof}
 
	\begin{thm}
		\label{thm:compactC}
		\( \C(U, c_0) \subset \hB_n^1(U) \) is compact.
	\end{thm}

	\begin{proof}
		The set \( \C(U,c_0) \) is closed since the limit \( J \) in the \( B^1 \) norm of a sequence \( J_i \) of positive chains is again positive, and \( M(J)=\lim_{i\to\infty} M(J_i)\le c_0 \) by Lemma \ref{lem:fdV}. By Theorem \ref{thm:compactinclusion} the \( u_n^{0,1} \) image \( I \) of the ball of radius \( c_0 \) in \( \hB_n^0 \) is totally bounded. Since \( \C(U,c_0) \) is a subset of the closure of \( I \), it is therefore compact.
	\end{proof}

	\begin{cor}
		\label{cor:com}
		\( \k_Y(\T(M, Y, U))\subset \hB_n^1(U) \) is compact.
	\end{cor}

	\begin{proof} 
		By Proposition \ref{prop:t2complete} \ref{item:positive} and Theorem \ref{thm:compactC} it suffices to show \( \k_Y(\T(M, Y, U)) \) is closed. Suppose \( \{T_i\} \subset \k_Y(\T(M, Y, U)) \) and \( T_i \to T \). Write \( T_i = \k_Y(S_i) \) where \( S_i \in \T(M, Y, U) \). Then \( \p T_i= S_i-E_Y\tM \) by Proposition \ref{prop:t2complete} \ref{item:bdry} and Lemma \ref{lem:fc}, so \( S_i \to \p T+E_Y\tM =: S \in \T(M, Y, U)\) since \( \T(M, Y, U) \) is closed. It follows that \( T=\k_Y S\in \k_Y(\T(M, Y, U)) \).
	\end{proof}

	\begin{cor}
		\label{cor:existenceagain}
		There exists \( S_0 \in \T(M, Y, U) \) with \( \A^{n-1}(S_0) = \frak{m} \). 
	\end{cor}

	\begin{thm}
		\label{thm:convexhull}
		If \( S_0 \in \T(M, Y, U) \) is  \( \A^{n-1} \)-minimizing, then\footnote{Recall \( h(M) \) is the convex hull of \( M \).} \( \supp(\k_Y S_0)\subset h(M) \).
	\end{thm}

	\begin{proof}
		Let \( S_i\to S_0 \), where \( \{S_i\}\subset \F(M, Y, U) \) and set \( X_i:=\supp(S_i) \). If \( x\in \supp(\k_Y S_0) \) and \( x\notin h(M) \), let \( \e>0 \) be small enough so that \( \O(x,3\e)\cap h(M)=\emptyset \). Let \( f\in \B_0^1(U) \) be supported in \( \O(x,\e) \) such that \( \cint_{\k_Y S_0} f\,dV >0 \). Taking positive and negative parts\footnote{Recall \( \B_0^1 \) consists of bounded Lipschitz functions.} of \( f \) and rescaling, we can assume without loss of generality that \( 0\le f \le 1 \). By continuity, there exists \( \d>0 \) and \( N \) such that \( \cint_{\k_Y S_i}f\,dV>\d \) for all \( i>N \). Thus by Proposition \ref{prop:whitney} \ref{sameintegrals}, \[ 0<\d< \cint_{\k_Y S_i} f\,dV=\int_U f d\mu_{\k_Y S_i}\le \mu_{\k_Y S_i}(\O(x,\e)) = \sp^{n-1}(X_i\cap \O(x,\e)) \] for all \( i>N \).
		
		Let \( A=h(M) \), \( B=\O(h(M),\e)\setminus A \) and \( C=U\setminus (A\cup B) \). Fix \( p\in A \) and let \( G:\R^n\to \R^n \) be a diffeomorphism given by the \( (t=1) \)-flow of a vector field \( V \) which is everywhere pointing towards \( p \) (except where it is zero), and such that
		\begin{itemize}
			\item \( V\lfloor_A\equiv 0 \);
			\item There is some \( 0<s<1 \) such that if \( y\in C \), then \( G(y)=sy+(1-s)p \);
			\item The magnitude of \( V \) is non-increasing along flow-lines.
		\end{itemize}
		Then the Lipschitz constant of \( G \) is equal to \( 1 \), since the divergence of \( V \) is everywhere non-positive, and \( G(U)\subset U \) since \( U \) is convex. Since \( G \) is a diffeomorphism, it follows from Lemma \ref{lem:corespan} that \( \tilde{X_i}:=(G(X_i))^*\in \Span^*(M,U) \). Let \( \{\tilde{S_i}\}\subset \F(M, Y, U) \) be the corresponding sequence of film chains given by Theorem \ref{thm:setchaincorrespondance}. By Lemma \ref{lem:2}, we have:
		\begin{align*}
			\A^{n-1}(\tilde{S_i})=\sp^{n-1}(G(X_i))&=\sp^{n-1}(G(X_i\cap A))+\sp^{n-1}(G(X_i\cap B))+\sp^{n-1}(G(X_i\cap C))\\
			&\le\sp^{n-1}(X_i\cap A)+\sp^{n-1}(X_i\cap B)+s^{n-1}\sp^{n-1}(X_i\cap C)\\
			&=\sp^{n-1}(X_i)-(1-s^{n-1})\sp^{n-1}(X_i\cap C)\\
			&\le\A^{n-1}(S_i)-(1-s^{n-1})\sp^{n-1}(X_i\cap \O(x,\e))\le \A^{n-1}(S_i)-(1-s^{n-1})\d,
		\end{align*}
		thus contradicting the minimality of the sequence \( \{S_i\} \).
	\end{proof}

	So, in particular, \( \k_Y S_0 \in \C(U) \), so we may define its corresponding measure \( \mu_{\k_Y S_0} \). Unfortunately, this is not the end of the story, since \( S_0 \) is not a priori an element of \( \F(M, Y, U) \). There are two problems to take care of:

	\begin{enumerate}
		\item We don't know that \( \supp(S_0) \) spans \( M \). This will be resolved by Corollary \ref{cor:key} upon fixing the next problem:
		\item The Hausdorff spherical \( (n-1) \)-measure of \( \supp(\k_Y S_0) \) may not be equal to \( \frak{m} \) (indeed it might not be finite at all,) and the measure \( \mu_{\k_Y S_0} \) may not be \( \sp^{n-1}\lfloor_X \) for some \( X\in \Span^*(M,U) \). We prove in the next section that this cannot happen.
	\end{enumerate}

\section{Minimizing sequences}
	\label{sec:minimizing_sequences}
	\setcounter{thm}{0}
	Our main goal of this section is to prove:
	
	\begin{thm}
		\label{thm:measureequalsm}
		If \( S_0 \in \T(M, Y, U) \) is \( \A^{n-1} \)-minimizing, then \( \sp^{n-1}(\supp(S_0)) = \frak{m} \). Moreover, \( S_0\in \F(M, Y, U) \).
	\end{thm}

	\textbf{Notation}:
		All convergence of sets in this section will be in the Hausdorff distance. If \( M\subset U\subset \R^n \) and \( X\subset \R^n \), let \( \G(X,M,U) \) be the set of closed balls which are disjoint from \( M \), are contained in \( U \), and whose centers lie in \( X \). Let \( \G_c(X,M,U)\subset \G(X,M,U) \) be the collection of those balls \( B \in \G(X,M,U) \) for which \( \H^{n-2}\lfloor_X(\fr\, B)<\infty \). Note that Lemma \ref{lem:4} implies that if \( \H^{n-1}(X)<\infty \), then almost all\footnote{more precisely, for any center point and almost all radii} balls in \( \G(X,M,U) \) are in \( \G_c(X,M,U) \).

	\begin{defn}
		\label{def:beta}
		For \( \{X_k\}_{k \ge 0} \) a sequence of subsets of \( \R^n \), let  \[ \b = \b(\{X_k\},M,U) := \inf\left\{\liminf_{k \to \i} \frac{\sp^{n-1}(X_k(p,r))}{\a_{n-1} r^{n-1}} : \bar{\O}(p,r) \in \G(X_0,M,U) \right\}. \]
	\end{defn}
	
	We will show that if \( \b > 0 \), then \( \sp^{n-1}(X_0\setminus M) < \i \) (Theorem \ref{thm:finite},) and if \( \b \ge 1 \), then \( \sp^{n-1}(X_0 \cap W) \le \liminf \sp^{n-1}(X_k \cap W) \) for all open \( W \subset U \) (Theorem \ref{thm:lsc2}.) 
 	\begin{defn}
 		\label{def:rmin}

		Let \( \e_k \to 0 \) be a strictly decreasing sequence of non-zero reals, and \( M\subset U\subset \R^n \). A sequence \( \{X_k\} \), \( k\geq 1 \), is called \emph{\textbf{(uniformly) Reifenberg \( \{\e_k\}\)-regular}} if there exists \( \frak{a} > 0 \) such that for each \( k \) and \( \bar{\O}(p,r) \in \G(X_k,M,U) \) with \( r > \e_k \), \[ \sp^{n-1}(X_k(p,r)) \ge \frak{a} r^{n-1}. \] 
		
		If \( \e_k = 2^{-k} \), we say \( \{X_k\} \) is \emph{\textbf{Reifenberg regular}}.

 	\end{defn}

	Observe that if \( \{X_k\} \) is Reifenberg \( \{\e_k\} \)-regular, then so too is any subsequence. We shall show in Proposition \ref{lem:2*} that if \( \{X_k\} \subset \Span^*(M,U) \) is Reifenberg regular, and \( X_k \to X_0 \), then there exists a subsequence such that \( \b > 0 \), and in Proposition \ref{prop:beta} that \( \b \ge 1 \) for a further subsequence.

	\subsection*{Overview of the proof of Theorem \ref{thm:measureequalsm}}
		\label{sub:lower_regular_minimizing_sequences}

		\begin{enumerate}[1.]
			\item In \S\ref{sub:haircuts} we use our Compactness Theorem \ref{cor:com} to find an \( \A^{n-1} \)-minimizing sequence \( S_k \to S_0 \) in \( \T(M, Y, U) \) where \( \{S_k\}\subset \F(M, Y, U) \). Let \( X_k := \supp(\k_Y S_k) \), \( k \ge 0 \). In \S \ref{sub:haircuts} Theorem \ref{thm:haircut} we ``cut hairs'' and create a modified minimizing sequence \( \{\hat{X}_k\} \subset \Span^*(M,U) \) which converges to \( X_0 \) in the Hausdorff metric. We then apply Theorem \ref{thm:setchaincorrespondance} to create a corresponding sequence \( \{\hat{S}_k\}\subset \F(M, Y, U) \) of film chains. We use the compactness result of Corollary \ref{cor:com} to find a convergent subsequence \( \hat{S}_{k_i} \to \hat{S}_0 \in \T(M, Y, U) \), and use Proposition \ref{prop:whitney} to show that \( \hat{S}_0 = S_0 \). We may assume without loss of generality, then, that our original sequence  \( S_k \to S_0 \) satisfies  \( X_k \to X_0 \) as in Theorem \ref{thm:haircut}.

			\item In \S\ref{sub:lrm} we show that there exists a subsequence \( \{S_{k_i} \} \) of \( \{S_k\} \) whose supports \( \{X_{k_i} \} \) form a Reifenberg regular sequence. We may again assume that our original sequence \( \{S_k\} \) has this property, too. From Proposition \ref{prop:lrmin} we deduce that \( \b(\{X_k\},M,U) > 0 \) which implies that \( \sp^{n-1}(X_0) < \i \).
 			\item In \S\ref{sub:lower_semicontinuity_of_hausdorff_measure} we find a subsequence \( \{S_{k_i}\} \) whose supports are \( \e \)-uniformly concentrated in a disk \( \bar{\O}(p,r) \) where  \( \frac{\sp^{n-1}(X_{k_i}(p,r))}{\a_{n-1} r^{n-1}} \) is close to the infimum of \( \b(\{X_{k_i}\},M,U) \). This is sufficient to prove Proposition \ref{prop:beta}: If \( \{X_k\} \) is a Reifenberg regular minimizing sequence, there exists a subsequence \( \{X_{k_i}\} \) with \( \b(\{X_{k_i}\},M,U) \ge 1 \).  Lower semicontinuity of Hausdorff spherical measure for Reifenberg regular minimizing sequences Theorem \ref{thm:lsc2} follows easily. It is also easy to deduce from this that \( \sp^{n-1}(X_0) = \frak{m} \). That \( S_0\in \F(M, Y, U) \) follows shortly thereafter.
		\end{enumerate}

	 \subsection{Convergence in the Hausdorff metric}
		\label{sub:haircuts}
		Lemmas \ref{lem:8}-\ref{lem:haircut} are used to modify a spanning set \( X \in \Span^*(M,U) \) inside a closed ball \( \bar{\O}(p,r) \) disjoint from \( M \). These are followed by Lemma \ref{lem:gridsmall} which modifies \( X \) inside a closed cube. These lemmas will assist us in our proofs of Theorem \ref{thm:haircut} and Lemma \ref{lem:epsilonbound}. Lemma \ref{lem:8} is an adaptation of Lemma 8 in \cite{reifenberg}, p. 15.    

		\begin{lem}
			\label{lem:8}   
			There exists a constant \( 1\leq K_n<\i \) such that if \( X \in \Span^*(M,U) \) and \( \bar{\O}(p,r) \in \G_c(X,M,U) \), then there exists a compact set \( Z\subset \R^n \) spanning \( M \) satisfying
			\begin{enumerate}
				\item\label{lem:8:item:1} \( Z \cap \O(p,r)^c = X \cap \O(p,r)^c \);
				\item\label{lem:8:item:2} \( Z(p,r) \subset h(x(p,r)) \);
				\item\label{lem:8:item:3} \( Z(p,r) \subset \O(x(p,r), K_n \sp^{n-2}(x(p,r))^{1/(n-2)}) \); and
				\item\label{lem:8:item:4} \( \sp^{n-1}(Z(p,r)) \le K_n \sp^{n-2}(x(p,r))^{(n-1)/(n-2)}. \)
			\end{enumerate}
		\end{lem} 
 			
		\begin{proof}
			Let \( K_n \) be the constant \( K_{n-1}^n \) in Lemma 8 of \cite{reifenberg}, and let \( Z \) be the set defined as follows: let \( Z \setminus \bar{\O}(p,r) := X \setminus \bar{\O}(p,r) \) and let \( Z(p,r) \) be the ``surface'' defined in Lemma 8 of \cite{reifenberg}, setting \( A=x(p,r) \). Since \( Z(p,r) \) is contained in the convex hull of \( x(p,r) \), it follows that \( z(p,r)=x(p,r) \), and so \( Z \) satisfies \ref{lem:8:item:1}-\ref{lem:8:item:4}. It remains to show that \( Z \) spans \( M \). In dimension \( n=3 \) this is easiest to see: in this case, there is a grid of cubes and the set \( x(p,r) \) is a disjoint union of subsets \( x_i \), each contained in the interior of one such cube. The ``surface'' obtained by Reifenberg is then the disjoint union of the cones \( c_{p_i}(x_i) \), for some \( p_i\in x_i \). To see that \( Z \) spans in this case, suppose it does not. Then there exists a simple link \( N=\eta(S^1) \) of \( M \) disjoint from \( Z \). We modify the proof of Lemma \ref{lem:cutandcone} by projecting \( \eta \) radially to the frontier of each cube, then radially onto edges, and finally orthogonally in a coordinate onto the frontier of \( \O(p,r) \). The new curve is still a simple link of \( M \), but is now disjoint from \( X \) as well, giving a contradiction. 

			For \( n>3 \), the construction in Lemma 8 of \cite{reifenberg} is considerably more difficult, but the basic idea to showing \( Z \) spans \( M \) is the same: Using an inductive argument, Reifenberg ``chops'' \( \bar{\O}(p,r) \) into a number of \( n \)-cubes \( \frak{Q}=\{Q_i\}_{i=1}^N \), forming a cover of \( \bar{\O}(p,r) \) by \( n \)-cubes with mutually disjoint interiors. Pairs of opposite faces of each \( n \)-cube \( Q\in\frak{Q} \) are ordered by the coordinate vector to which they are perpendicular, and a face in the \( k \)-th pair lies in the same hyperplane as the corresponding face of any other \( n \)-cube neighboring \( Q \) in any of the remaining \( n-k \) directions. 

			For a given face \( F \) in the \( k \)-th pair of faces of an \( n \)-cube \( Q \), let \( \frak{N}(F) \) denote the maximal union, containing \( F \), of faces neighboring in the above manner\footnote{I.e., \( F\in \frak{N}(F) \), and if \( F'\in\frak{N}(F) \) is in the \( k \)-th pair of faces of some \( n \)-cube \( Q' \), then \( \frak{N}(F) \) also contains the corresponding face of any cube \( Q'' \) neighboring \( Q' \) in any of the remaining \( n-k \) directions. So, \( \frak{N}(F)=F \) when \( F \) is a face in the \( n \)-th pair of faces of \( Q \), and \( \frak{N}(F) \) is the entire hyperplane containing \( F \) when \( F \) is a face in the \( 1 \)-st pair.}. Let \[ \frak{N}_k := \{\frak{N}(F) : F \textrm{ is in the } k\textrm{-th pair of faces for some cube } Q \}. \] Likewise, let \( \cal{N}_k :=\cup F \), where the union is taken over all \( k \)-th faces \( F \) of the cubes \( Q\in \frak{Q} \).

			Now, let \( \eta \) be as above. We first project \( \eta \) radially to the frontier of each cube, so now \( \eta(S^1)\cap \O(p,r)\subset \cup_{s=1}^n \cal{N}_s \). Now suppose \( \eta(S^1)\cap \O(p,r) \) is contained in \( \cup_{s=1}^k \cal{N}_s \). Then by projecting \( \eta(S^1)\cap\frak{N}(F) \cap \O(p,r) \) onto the frontier\footnote{considered as a subset of the hyperplane containing \( \frak{N}(F) \)} of \( \frak{N}(F) \cap \bar{\O}(p,r) \), for each \( \frak{N}(F)\in \frak{N}_k \) (which we may do by induction on \( n \), starting at \( n=4 \), by the \( n=3 \) case of the proof,) it will follow that \( \eta(S^1)\cap \O(p,r) \) is contained in \( \cup_{s=1}^{k-1} \cal{N}_s \). Thus, by downward induction on \( k \), this will result in a new simple link of \( M \) whose intersection with \( \bar{\O}(p,r) \) lies entirely within \( \fr\,\O(p,r) \), and is thus disjoint from \( X \) by \ref{lem:8} \ref{lem:8:item:1}, giving a contradiction.
		\end{proof}
		
		The following lemma is an adaptation of Lemma 9 in \cite{reifenberg}, p. 18:
		\begin{lem}
			\label{lem:9}
			Suppose \( X \in \Span^*(M,U) \) and \( \bar{\O}(p,r) \in \G_c(X,M,U) \) with \[ \sp^{n-2}(x(p,r)) < r^{n-2}/(2K_n)^{n-2}. \] There exists a compact set \( \widetilde{X} \subset \R^n \) spanning \( M \) such that
			\begin{enumerate}
				\item\label{lem:9:item:1} \( \widetilde{X}\cap \bar{\O}(p,r)^c = X\cap \bar{\O}(p,r)^c \);
				\item\label{lem:9:item:2} \( \widetilde{X}(p,r) \subset \fr\,\O(p,r) \); and
				\item\label{lem:9:item:3} \( \sp^{n-1}(\widetilde{X}(p,r)) \le 2^{n-1}K_n \sp^{n-2}(x(p,r))^{(n-1)/(n-2)} \).
			\end{enumerate}
		\end{lem}
 
		\begin{proof}
			Let \( Z \) be as in Lemma \ref{lem:8}. Lemma \ref{lem:8} \ref{lem:8:item:3} implies \( Z(p,r) \subset \overline{\O(p,r)} \cap \O(p,r/2)^c \). So, we can radially project \( Z(p,r) \) onto \( \fr\,\O(p,r) \), and there exists a Lipschitz extension \( \pi: \R^n\to \R^n \) of this projection with Lipschitz constant \( 2 \), \( \pi\lfloor_{\O(p,r)^c}\equiv Id \), and \( \pi(\bar{\O}(p,r))\subset \bar{\O}(p,r) \). Then \( \widetilde{X} := \pi Z \) is compact and spans \( M \) by Theorem \ref{thm:lipspan}. Lemmas \ref{lem:2} and \ref{lem:8}\ref{lem:8:item:4} yield
			\begin{align*}
				\sp^{n-1}(\widetilde{X}(p,r)) = \sp^{n-1}(\pi(Z(p,r))) \leq 2^{n-1} \sp^{n-1}(Z(p,r)) \le 2^{n-1} K_n \sp^{n-2}(x(p,r))^{(n-1)/(n-2)}.
			\end{align*} 
		\end{proof}
The statement of the next lemma is essentially (2) and (3) on p. 24 of \cite{reifenberg}.
 
		\begin{lem}
			\label{lem:halfradius}
			Suppose \( X \in \Span^*(M,U) \) and \( \bar{\O}(p,r) \in \G(X,M,U) \) with \[ \sp^{n-1}(X(p,r)) \le \frac{r^{n-1}}{\left(2(n-1)\right)^{n-1}\left(2^n K_n\right)^{n-2}}. \] Suppose \( W \subset (r/2,r) \) has full Lebesgue measure. There exists \( r' \in W \) such that \[ \sp^{n-2}(x(p,r'))^{(n-1)/(n-2)} \le \frac{1}{2^n K_n} \int_0^{r'} \sp^{n-2} (x(p,t))dt. \]
		\end{lem}

		\begin{proof}
			Suppose there is no such \( r' \). Let \( F(p,s):=\int_0^s \sp^{n-2}(x(p,t))dt \). By the Lebesgue Differentiation Theorem, \[ \frac{\frac{d}{ds}F(p,s)}{F(p,s)^{(n-2)/(n-1)}} > \frac{1}{(2^n K_n)^{(n-2)/(n-1)}} \] for almost every \( r/2 < s < r \). Integrating, this implies
			\begin{align} 
				\label{lem:halfradius:eq:1}
				\int_{r/2}^r \frac{d}{ds}\left(F(p,s)^{1/(n-1)}\right) ds > \frac{r}{2(n-1)\left(2^n K_n\right)^{(n-2)/(n-1)}}.
			\end{align}
			Since \( F(p,s) \) is increasing and absolutely continuous, the function \( F(p,s)^{1/(n-1)} \) is also absolutely continuous, and so the left hand side of \eqref{lem:halfradius:eq:1} is equal to \( F(p,r)^{1/(n-1)} - F(p,r/2)^{1/(n-1)} \) by the Fundamental Theorem for Lebesgue Integrals. Lemma \ref{lem:4} gives \[ \sp^{n-1}(X(p,r)) \ge F(p, r) > \frac{r^{n-1}}{\left(2(n-1)\right)^{n-1}\left(2^n K_n\right)^{n-2}}, \] contradicting our initial assumption.
		\end{proof}
The statement of the next lemma is essentially lines (2) and (6) on p. 24 of \cite{reifenberg}.
		\begin{lem}
			\label{lem:haircut}
			Suppose \( X \in \Span^*(M,U) \) and \( \bar{\O}(p,r) \in \G(X,M,U) \) with \[ \sp^{n-1}(X(p,r)) \le \frac{r^{n-1}}{\left(2(n-1)\right)^{n-1}\left(2^n K_n\right)^{n-2}}. \] If \( W \subset (r/2,r) \) has full Lebesgue measure, then there exist \( r/2 < r' < r \) such that \( r'\in W \), and a compact set \( \hat{X}\subset \R^n \) spanning \( M \) such that
			\begin{enumerate}
				\item \( \hat{X} \cap \bar{\O}(p,r')^c = X \cap \bar{\O}(p,r')^c \);
				\item \( \hat{X}(p,r') \subset \fr\,\O(p,r') \); and
				\item\label{lem:haircut:item:3} \( \sp^{n-1}(\hat{X}(p,r')) \le \sp^{n-1}(X(p,r'))/2. \)
			\end{enumerate}
		\end{lem}

		\begin{proof}
			By Lemmas \ref{lem:halfradius} and \ref{lem:4}, there exists \( r' \in W \) such that
			\begin{align}
				\label{lem:haircut:eq:1}
				\sp^{n-2}(x(p,r'))^{(n-1)/(n-2)} \le \frac{1}{2^n K_n}\int_0^{r'}\sp^{n-2} (x(p,t))dt \le \frac{1}{2^n K_n}\sp^{n-1}(X(p,r')).
			\end{align}
			Thus, \( \sp^{n-2}(x(p,r'))^{1/(n-2)} \le \frac{r}{2^{n+1}(n-1)K_n}<r'/(2K_n) \), and so we may apply Lemma \ref{lem:9} to get the required set \( \hat{X} \), since \eqref{lem:haircut:eq:1} and Lemma \ref{lem:9}\ref{lem:9:item:3} give
			\begin{align*}
				\sp^{n-1}(\hat{X}(p,r')) \le 2^{n-1}K_n\sp^{n-2}(x(p,r'))^{(n-1)/(n-2)} \le \frac{1}{2}\sp^{n-1}(X(p,r')).
			\end{align*}
		\end{proof}

The statement of the next lemma is essentially lines (14) and (17) on p. 25 of \cite{reifenberg}.

		\begin{lem}
			\label{lem:gridsmall}
			Suppose \( X \in \Span^*(M,U) \) and \( Q \) is a closed \( n \)-cube of side length \( \ell \) disjoint from \( M \), with \[ \sp^{n-1}(X \cap \mathring{Q}) < \frac{\ell^{n-1}}{\left(4(n-1)\right)^{n-1}\left(2^n K_n\right)^{(n-2)}}. \] Then there exists a compact set \( X' \) spanning \( M \) such that
			\begin{enumerate}
				\item \( X'\cap Q^c = X\cap Q^c \);
				\item \( X'\cap Q\subset \fr\, Q \);
				\item \( X'\cap Q = (X\cap \fr\, Q)\cup Y \),
			\end{enumerate}
			where \( Y\subset \fr Q \) satisfies
			\begin{align}
				\label{lem:gridsmall:estimate}
				\sp^{n-1}(Y)\leq (2\sqrt{n})^{n-1} \sp^{n-1}(X\cap \mathring{Q}).
			\end{align}
		\end{lem}

		\begin{proof}
			Let \( p \) be the center point of \( Q \). Then \( \bar{\O}(p,\ell/2) \subset Q \) satisfies the conditions of Lemma \ref{lem:haircut}, and so there exists \( \ell/4 < r < \ell/2 \) and a compact set \( \hat{X} \) spanning \( M \) such that
			\begin{itemize}
				\item \( \hat{X} \cap \bar{\O}(p,r)^c = X \cap \bar{\O}(p,r)^c \);
				\item \( \hat{X}(p,r) \subset \fr\, \O(p,r) \); and
				\item \( \sp^{n-1}(\hat{X}(p,r)) \leq \sp^{n-1}(X(p,r))/2. \)
			\end{itemize}
			The remaining part of \( \hat{X}\cap \mathring{Q} \) is contained in the region \( \mathring{Q} \cap \O(p,r)^c \), so we may radially project \( \hat{X}\cap \mathring{Q} \) from \( p \) to a set \( Y\subset \fr\, Q \), and the image \( X' \) of \( \hat{X} \) under this map spans \( M \) by Theorem \ref{thm:lipspan}. Moreover, since the Lipschitz constant of the projection is bounded above by \( 2\sqrt{n} \), by Lemma \ref{lem:2} and Lemma \ref{lem:haircut} \ref{lem:haircut:item:3}, we have
			\begin{align*}
				\sp^{n-1}(Y) &\le (2\sqrt{n})^{n-1}\sp^{n-1}(\hat{X} \cap \mathring{Q})\\
				&= (2\sqrt{n})^{n-1}\left[\sp^{n-1}(\hat{X}\cap \mathring{Q}\cap (\bar{\O}(p,r)^c))+\sp^{n-1}(\hat{X}(p,r))\right]\\
				&\leq (2\sqrt{n})^{n-1} \sp^{n-1}(X\cap \mathring{Q}).
			\end{align*}
		\end{proof}

		\begin{defn}
			\label{def:binarysubdivision}
			A collection \( \frak{S} \) of closed \( n \)-cubes is a \emph{\textbf{dyadic subdivision of \( \R^n \)}} if \( \frak{S}=\sqcup_{k\in \Z}\frak{S}_k \), where each \( \frak{S}_k \) is a cover of \( \R^n \) by \( n \)-cubes of side length \( 2^{-k} \) that intersect only on faces, and such that \( \frak{S}_{k+1} \) is a refinement of \( \frak{S}_k \).
		\end{defn}

		Unlike the above sequence of lemmas, the following theorem is an entirely new result:

		\begin{thm}
			\label{thm:haircut}
			If \( S_0 \in \T(M, Y, U) \) is \( \A^{n-1} \)-minimizing, then there exists a sequence \( \{S_k\}\subset \F(M, Y, U) \) such that \( S_k \to S_0 \) in \( \hB_{n-1}^2(U) \), \( \k_Y S_k \to \k_Y S_0 \) in \( \hB_n^1(U) \), \( \supp(\k_Y S_k) \to \supp(\k_Y S_0) \) and \( \sp^{n-1}(\supp(\k_Y S_k)) \to  \frak{m} \)  
 		\end{thm}  

		\begin{proof}
			Let \( T_k \to S_0 \) in \( \hB_{n-1}^2(U) \) with \( \{ T_k\} \subset \F(M, Y, U) \). Since \( \{ \k_Y T_k\}\subset \C(U) \) (Definition \ref{def:filmchain} \ref{def:filmchain:item:3}) and \( \k_Y T_k \to \k_Y S_0 \) in \( \hB_n^2(U) \) (Definition \ref{def:cone},) it follows from Corollary \ref{cor:positivesequence} that there exists a subsequence \( \k_Y T_{k_i} \to \k_Y S_0 \) in \( \hB_n^1(U) \). So, let us assume without loss of generality that this is the case, that \( T_k\to S_0 \) in \( \hB_{n-1}^2(U) \) with \( \{ T_k\}\subset \F(M, Y, U) \) and \( \k_Y T_k\to \k_Y S_0 \) in \( \hB_n^1(U) \). Let \( X_k = \supp(\k_Y T_k) \) for \( k \ge 0 \).

			Let \( \frak{S}=\sqcup_{j\in \Z}\frak{S}_j \) be a dyadic subdivision of \( \R^n \) and let \( \frak{D}_j=\{Q\in \frak{S}_j : Q\cap X_0=\emptyset, Q\cap U \neq \emptyset \} \),  let \( \frak{d}_j< \i \) be the cardinality of \( \frak{D}_j \) and let \( \frak{N}_j=\{Q\in \frak{S}_j: Q\cap X_0 \neq \emptyset \} \). Note that \( \frak{D}_j\sqcup \frak{N}_j \) covers \( \bar{U} \). By Theorem \ref{thm:convexhull}, there exists \( N \in \N \) such that if \( j\geq N \) and \( Q\in \frak{N}_j \), then \( Q\subset U \). So, \( U_j:=\cup_{Q \in \frak{N}_j} Q \) is a neighborhood of \( X_0 \), and \( U_j\subset \O(X_0,\sqrt{n}2^{-j+1}) \). If \( j\geq N \), then \( U_j\subset U \). For the rest of this proof, let us assume \( j\geq N \).

			Since \( \supp(\mu_{\k_Y S_0})=X_0 \) by Proposition \ref{prop:whitney} \ref{samesupports}, it follows that \( \mu_{\k_Y S_0}(Q) = 0 \) for all \( Q\in\frak{D}_j \). In particular, \( \mu_{\k_Y S_0}(\fr\,Q) = 0 \), so we may apply Proposition \ref{prop:maggi} to deduce \( \lim_{k \to \i}\mu_{\k_Y T_k}(Q) = \mu_{\k_Y S_0}(Q) = 0 \). Since \( T_k \in \F(M, Y, U) \), we have \( \mu_{\k_Y T_k}(Q) = \sp^{n-1}\lfloor_{X_k}(Q) \), hence \( \lim_{k \to \i}\sp^{n-1}\lfloor_{X_k}(Q) = 0 \).

			Let \( \rho_j := \min\{2^{-(n-1)j}/((2^n K_n)^{(n-2)}(4(n-1))^{n-1}),\,\, 2^{-j}/\frak{d}_j\} \). Since \( 0<\frak{d}_j<\infty \), there exists \( N_j\in \N \) such that if \( j \ge N_j \) then \( \sp^{n-1}(X_j \cap Q) < \rho_j \) for all \( Q\in \frak{D}_j \). In particular,
			\begin{equation}
				\label{thm:haircut:eq:1}
				\sum_{Q \in \frak{D}_j} \sp^{n-1}(X_{N_j} \cap Q) < 2^{-j}.
			\end{equation}
			for each \( j \). By Lemma \ref{lem:gridsmall} there exists a compact set \( Y_j\subset \R^n \) spanning \( M \), equal to \( X_{N_j} \) outside \( \cup_{Q\in \frak{D}_j}Q \), such that \( Y_j\cap \mathring{Q}=\emptyset \) for each \( Q\in \frak{D}_j \) and
			\begin{equation}
				\label{eq:face}
				\sp^{n-1}(Y_j \cap F) \le (2(2\sqrt{n})^{n-1}+1) \rho_j
			\end{equation}
			for each face \( F \) of \( Q\in \frak{D}_j \). (This is because each face \( F \) is shared by at most two adjacent cubes in \( \frak{D}_j \).)

			Let \( Q\in \frak{D}_j \) and suppose \( F \) is a face of \( Q \) not contained in \( U_j \). Since \( \sp^{n-1}(Y_j \cap F) < 2^{-j(n-1)-1}=1/2\sp^{n-1}(F) \) and \( Y_j\cap F \) is closed, there exists a finite sequence of deformations (see Definition \ref{def:competitor}) \( \{\phi_i\} \) of \( Y_j \) which sends \( Y_j\cap F \) to the \( (n-2) \)-skeleton of \( F \). Repeating this process for each face \( F \) of each \( Q\in \frak{D}_j \) such that \( F \) is not contained in \( U_j \), we get a compact set \( Z_j \) spanning \( M \) by Theorem \ref{thm:lipspan}.

			Furthermore, the core \( Z_j^* \) is contained in \( U_j\subset U \), since \( Z_j\cap U_j^c \) is a subset of the \( (n-2) \)-skeletons of cubes in \( \frak{D}_j \). By \eqref{thm:haircut:eq:1}, Lemma \ref{lem:gridsmall}\eqref{lem:gridsmall:estimate}, and \eqref{eq:face}, we have
			\begin{align*}
				\sp^{n-1}(Z_j^*)\leq \sp^{n-1}(Z_j)\leq \sp^{n-1}(Y_j) &\leq \sp^{n-1}(X_{N_j}) + (2\sqrt{n})^{n-1} \sum_{Q \in \frak{D}_j} \sp^{n-1}(X_{N_j}\cap \mathring{Q})\\
				&\leq \sp^{n-1}(X_{N_j}) + (2\sqrt{n})^{n-1} 2^{-j} < \i. 
			\end{align*}
			In particular, \( Z_j^*\in \Span^*(M,U) \) by Lemma \ref{lem:corespan} and \( \sp^{n-1}(Z_j^*)\to \frak{m} \).

			Now, for each \( j\geq 1 \) let \( S_j' \) be the film chain corresponding to \( Z_j^* \) in Theorem \ref{thm:setchaincorrespondance}. Since \( \{S_j'\}\subset \T(M, Y, U) \), it follows from Corollary \ref{cor:com} that there exists a subsequence \( \{S_{j_i}'\} \) converging to some \( S_0' \) in \( \T(M, Y, U) \). As in the beginning of the proof, by taking a further subsequence, we can also ensure that \( \k_Y S_{j_i}'\to \k_Y S_0' \) in \( \C(U) \).

			We show \( \k_Y S_0' = \k_Y S_0 \), and hence \( S_0' = S_0 \) by Proposition \ref{prop:t2complete} \ref{item:bdry}. Let \( X_0'=\supp(\k_Y S_0') \). Since \( Z_{j_i}^* \subset U_j  \) and \( U_j\subset \O(X_0,\sqrt{n}2^{-j+1}) \), it follows from Lemma \ref{lem:converge} that \( X_0'\subset X_0\subset U \). Therefore, \( \k_Y S_0'\in \C(U) \), and we may apply Proposition \ref{prop:whitney} to get a corresponding measure \( \mu_{\k_Y S_0'} \). It suffices to show \( \mu_{\k_Y S_0'}=\mu_{\k_Y S_0} \).

			 Let \( \frak{S}' \) be a new dyadic subdivision of \( \R^n \) such that each cube \( Q\in\frak{S}' \) is \( (\k_Y S_0) \)-compatible\footnote{See Definition \ref{def:compatible}} and \( (\k_Y S_0') \)-compatible. Fix \( Q\in\frak{S}' \). Then by Proposition \ref{prop:maggi},
			\begin{align*}  
				\mu_{\k_Y S_0'}(Q)=\lim_{i\to\infty} \mu_{\k_Y S_{j_i}'}(Q)=\lim_{i\to\infty} \sp^{n-1}(Z_{j_i}^*\cap Q)&\leq \lim_{i\to\infty} \sp^{n-1}(X_{j_i}\cap Q)+ (2\sqrt{n})^{n-1}  2^{-j_i}\\
				&=\lim_{i\to\infty} \mu_{\k_Y T_{j_i}}(Q)=\mu_{\k_Y S_0}(Q).
			\end{align*}
			Likewise, by \eqref{thm:haircut:eq:1}, \[ \mu_{\k_Y S_0}(Q)=\lim_{i\to\infty} \sp^{n-1}(X_{j_i}\cap Q)\leq \lim_{i\to\infty} \sp^{n-1}(Z_{j_i}^*\cap Q)+2^{-j_i} = \mu_{\k_Y S_0'}(Q). \] Now if \( W\subset \R^n \) is open, by taking a Whitney decomposition of \( W \) using cubes from \( \frak{S}' \) we conclude that \( \mu_{\k_Y S_0'}(W)=\mu_{\k_Y S_0}(W) \). Since both measures are finite Borel measures on \( \R^n \) and hence Radon measures, outer regularity proves the two measures are equal. 

			Finally, since \( \supp(\k_Y S_{j_i}')=Z_{j_i}^*\subset U_k\subset \O(\supp(\k_Y S_0),\sqrt{n}2^{-j+1}) \), it follows from Lemma \ref{lem:converge} and compactness of \( \supp(\k_Y S_0) \) that \( \supp(\k_Y S_{j_i}') \) converges to \( \supp(\k_Y S_0) \).
		\end{proof}

	\subsection{Lower density bounds}
		\label{sub:lrm}
		The first three results of this section strengthen Theorem \ref{thm:haircut} so that it produces \( S_k \to S_0 \) where the sequence \( \{\supp(\k_Y S_k)\} \) is Reifenberg regular, along with its other properties (see Proposition \ref{prop:lrmin} for a formal statement.)  

		Lemma \ref{lem:epsilonbound} is essentially (19) on p. 26 of \cite{reifenberg}.
		\begin{lem}
			\label{lem:epsilonbound} 
 			Suppose that \( \{X_k\} \subset \Span^*(M,U) \) is minimizing and \( X_k \to X_0 \). Then there exists a subsequence \( k_i\to\i \) such that
			\begin{equation}
				\label{lem:epsilonbound:2}
				\sp^{n-1}(X_{k_i}(p,r)) > \frac{2^{(-i-1)(n-1)}}{\left(2(n-1)\right)^{n-1}\left(2^n K_n\right)^{n-2}} \,\,\, \mbox{ for all }\,\,\, \bar{\O}(p,r) \in \G(X_{k_i},M,U) \mbox{ and } r>2^{-i-1}.
			\end{equation}
		\end{lem}

		\begin{proof}
			If not, there exist \( N_1 \) and \( N_2 >0 \) such that for all \( k\geq N_1\), there exists \( \bar{\O}(p_k,r_k)\in \G(X_k,M,U) \) with \( r_k>2^{-N_2-1} \) such that \[ \sp^{n-1}(X_k(p_k,r_k)) \leq \frac{2^{(-N_2-1)(n-1)}}{\left(2(n-1)\right)^{n-1}\left(2^n K_n\right)^{n-2}}. \] Since \( U \) is bounded, there exists a subsequence \( p_{k_j}\to p\in X_0 \). By Lemma \ref{lem:haircut}, there exist \( 2^{-N_2-2} < r_{k_j}' < r_{k_j} \) and a compact spanning set \( \hat{X}_{k_j} \) satisfying the conclusions of Lemma \ref{lem:haircut}. In particular,
			\begin{equation}
				\label{eq:epbound1}
				\sp^{n-1}(\hat{X}_{k_j}(p_{k_j},r_{k_j}')) \le \frac{1}{2}\sp^{n-1}(X_{k_j}(p_{k_j},r_{k_j}')).
			\end{equation}

			Let \( \bar{\O}(p,r)\in \G(X_0,M,U) \) be \( (\k_Y S_0) \)-compatible, with \( r<2^{-N_2-2} \). Then for \( j \) large enough, \( \bar{\O}(p,r)\subset \bar{\O}(p_{k_j},r_{k_j}') \). Since \( p\in X_0 \), Propositions \ref{prop:maggi} and \ref{prop:whitney} \ref{samesupports} imply \[ 0 < \mu_{\k_Y S_0}(\bar{\O}(p,r)) = \lim_{j\to\i} \sp^{n-1}(X_{k_j}(p,r)). \] Thus, for \( j \) large enough,
			\begin{equation}
				\label{eq:epbound2}
				0 < \frac{1}{2} \mu_{\k_Y S_0}(\bar{\O}(p,r)) < \sp^{n-1}(X_{k_j}(p,r)) \leq \sp^{n-1}(X_{k_j}(p_{k_j}, r_{k_j}')).  			
			\end{equation}
			But since \( X_{k_j}\cap\bar{\O}(p_{k_j},r_{k_j}')^c =\hat{X}_{k_j}\cap\bar{\O}(p_{k_j},r_{k_j}')^c \), we may use \eqref{eq:epbound1} and \eqref{eq:epbound2} to deduce
			\begin{align*}
				\sp^{n-1}(\hat{X}_{k_j})&=\sp^{n-1}(\hat{X}_{k_j}(p_{k_j},r_{k_j}'))+\sp^{n-1}(X_{k_j}\cap\bar{\O}(p_{k_j},r_{k_j}')^c)\\
				&\le \frac{1}{2} \sp^{n-1}(X_{k_j}(p_{k_j},r_{k_j}')) + \sp^{n-1}(X_{k_j}\cap\bar{\O}(p_{k_j},r_{k_j}')^c)\\
				&=\sp^{n-1}(X_{k_j})-\frac{1}{2} \sp^{n-1}(X_{k_j}(p_{k_j},r_{k_j}'))\\
				&<\sp^{n-1}(X_{k_j})-\frac{1}{4} \mu_{\k_Y S_0}(\bar{\O}(p,r)).
			\end{align*}
			Since \( \sp^{n-1}(X_{k_j})\to \frak{m} \), we have \( \sp^{n-1}(\hat{X}_{k_j})<\frak{m} \) for \( j \) large enough, a contradiction.
		\end{proof}
The next Lemma is essentially (20) on p. 26 of \cite{reifenberg}.
		\begin{lem}
			\label{prop:prelrmin}
			Suppose that \( \{X_k\} \subset \Span^*(M,U) \) is minimizing and \( X_k \to X_0 \). Then there exists a subsequence \( \{X_{k_i}\} \) and a constant \( \frak{a} > 0 \) such that \[ F_{k_i}(p,r) \geq \frak{a} r^{n-1} \] for all \( i>0 \), \( \bar{\O}(p,r) \in \G(X_{k_i},M,U) \) and \( r > 2^{-i} \).
		\end{lem}

		\begin{proof} 
			Since \( \sp^{n-1}(X_k)\to \frak{m} \), let us assume without loss of generality that \[ \sp^{n-1}(X_k)\leq \frak{m} + \frac{2^{(-k-1)(n-1)-1}}{\left(2(n-1)\right)^{n-1}\left(2^n K_n\right)^{n-2}}. \] Let \( \{X_{k_i}\} \) be the subsequence determined by Lemma \ref{lem:epsilonbound}. Fix \( i \), and suppose \( \bar{\O}(p,r)\in \G(X_{k_i},M,U) \) and \( r>2^{-i} \). Let \( \bar{\O}(p,s) \in \G_c(X_{k_i},M,U) \) where \( r>s>2^{-i-1} \). Let \( Z_{k_i} \) be the set determined by Lemma \ref{lem:8} using \( X_{k_i} \) and \( \bar{\O}(p,s) \). Since \( \sp^{n-1}(Z_{k_i})\geq \frak{m} \) and \( Z_{k_i}\cap\bar{\O}(p,s)^c=X_{k_i}\cap\bar{\O}(p,s)^c \), we have
			\begin{equation}
				\label{eq:prel}
				\sp^{n-1}(X_{k_i}(p,s)) - \frac{2^{(-k_i-1)(n-1)-1}}{\left(2(n-1)\right)^{n-1}\left(2^n K_n\right)^{n-2}}\leq \sp^{n-1}(Z_{k_i}(p,s)). 		
			\end{equation}
			Thus, it follows from Lemmas \ref{lem:4}, \ref{lem:epsilonbound} and \ref{lem:8} that
			\begin{align*}
				F_{k_i}(p,s) \leq \sp^{n-1}(X_{k_i}(p,s)) &\leq 2\sp^{n-1}(X_{k_i}(p,s)) - \frac{2^{(-k_i-1)(n-1)}}{\left(2(n-1)\right)^{n-1}\left(2^n K_n\right)^{n-2}}\\
				&\leq 2\sp^{n-1}(Z_{k_i}(p,s)) \le 2 K_n \sp^{n-2}(x_{k_i}(p,s))^{(n-1)/(n-2)}.
			\end{align*}
			In other words, for almost every \( x\in (2^{-i-1},r) \), \[ \frac{\frac{d}{dx} F_{k_i}(p,x)}{\left(F_{k_i}(p,x)\right)^{(n-2)/(n-1)}}\geq (2K_n)^{-(n-2)/(n-1)}. \] Integrating, this implies \( F_{k_i}(p,r)^{1/(n-1)}-F_{k_i}(p,2^{-i-1})^{1/(n-1)} \ge \frac{1}{n-1}(2K_n)^{-(n-2)/(n-1)}(r-2^{-i-1}) \), and thus \[ F_{k_i}(p,r) \geq \frac{r^{n-1}}{(2(n-1))^{n-1}(2K_n)^{n-2}}. \]
			Setting \( \frak{a}=(2(n-1))^{-(n-1)}(2K_n)^{-(n-2)} \) completes the proof.
		\end{proof}

		Summarizing what we have so far, Theorem \ref{thm:haircut}, Proposition \ref{prop:prelrmin} and Lemma \ref{lem:4} imply:

		\begin{prop}
			\label{prop:lrmin}
			Suppose \( S_0\in \T(M, Y, U) \) is \( \A^{n-1} \)-minimizing. Then there exists a sequence \( S_k\to S_0 \) in \( \hB_{n-1}^2(U) \) where \( \{S_k\}\subset \F(M, Y, U) \) such that:
			\begin{enumerate}
				\item \( \k_Y S_k \to \k_Y S_0 \) in \( \C(U) \);
				\item \( X_k \to X_0 \) where \( X_k:=\supp(\k_Y S_k) \);
				\item \( \{X_k\} \subset \Span^*(M,U) \) is Reifenberg regular.
			\end{enumerate}
		\end{prop}

		The next result is essentially Lemma 2* on p. 27 of \cite{reifenberg}. Recall \( \b \) from Definition \ref{def:beta}.
		\begin{prop}
			\label{lem:2*}
			If \( \{X_k\} \) is a Reifenberg regular sequence of compact subsets of \( U \) with finite \( \sp^{n-1} \)-measure and \( X_k\to X_0 \), then \[ \b \geq \frak{a}/\a_{n-1} > 0. \]
		\end{prop}
 
		\begin{proof}
			Let \( \bar{\O}(p,r) \in \G(X_0,M,U) \) and \( 0<\d<1 \). For sufficiently large \( k \), there exists \( \bar{\O}(p_k,r_k)\in \G(X_k,M,U) \) such that \[ 2^{-k}<\d \cdot r < r_k < r \] and \( \bar{\O}(p_k,r_k)\subset \bar{\O}(p,r) \). By the definition of Reifenberg regular \( \sp^{n-1}(X_k(p,r))\geq \sp^{n-1}(X_k(p_k,r_k))\geq \frak{a}\, r_k^{n-1} >\frak{a}\, \d^{n-1}\, r^{n-1} \) and thus \( \b\geq \frak{a} \,\d^{n-1}/\a_{n-1} \). Now let \( \d \to 1 \).
		\end{proof}
		
		\begin{thm}
			\label{thm:finite}
			Suppose \( M\subset U\subset \R^n \), and \( \{X_k\} \) is a sequence of compact subsets of \( U \) with \( \b>0 \). If \( X_k\to X_0 \subset\subset U \) and \( \liminf_{k \to \i}\sp^{n-1}(X_k) =: \frak{c} < \i \), then \[ \sp^{n-1}(X_0\setminus M) < \i. \]
		\end{thm}

		\begin{proof}
			Suppose \( \{\bar{\O}(p_i, r_i)\}_{i\in I} \subset \G(X_0,M,U) \) is a collection of disjoint balls. If \( J\subset I \) is finite, then by the definition of \( \b \),
			\begin{align*}
				\b\sum_{j\in J} \a_{n-1} r_j^{n-1} \leq \sum_{j\in J} \liminf_{k \to \i} \sp^{n-1}(X_k(p_j,r_j)) &\leq \liminf_{k \to \i} \sum_{j\in J} \sp^{n-1}(X_k(p_j,r_j))\\
				&\leq \liminf_{k \to \i}\sp^{n-1} (X_k) = \frak{c}.
			\end{align*}
			Since \( I \) is necessarily countable, it follows from Proposition \ref{lem:2*} that
			\begin{equation}
				\label{44}
				\a_{n-1} \sum_{i\in I}r_i^{n-1} \leq \frac{\frak{c}}{\b}.
			\end{equation}
			Now fix \( \d \) and \( \d' \) such that \( \d' > \d > 0 \) and \( \O(X_0,\d)\subset U \). Then the subcollection of \( \G(X_0,M,U) \) consisting of balls of radius \( r<\d \) covers \( X_0 \cap \O(M,\d')^c \). So, by the Vitali Covering Lemma and \eqref{44}, it follows that \[ \sp_{10\d}^{n-1}(X_0 \cap \O(M,\d')^c) \le 5^{n-1} \frac{\frak{c}}{\b}. \] Letting \( \d \to 0 \) and then \( \d' \to 0 \), we deduce that \[ \sp^{n-1}(X_0 \setminus M) < \i. \]
		\end{proof}

	\subsection{Lower semicontinuity of Hausdorff spherical measure}
		\label{sub:lower_semicontinuity_of_hausdorff_measure}
 		
		In this section we establish Theorem \ref{thm:lsc2} which yields lower semicontinuity of Hausdorff spherical measure for Reifenberg regular minimizing sequences of compact spanning sets.  The proof follows easily once we establish that  \( \b \ge 1 \).

		We first reproduce four technical Lemmas \ref{lem:7}-\ref{lem:4*} adapted from \cite{reifenberg}. The next result, which is modeled after Lemma 7 in \cite{reifenberg}, p. 12, makes use of a modified cone construction to fit with our definition of a spanning set:
		\begin{lem}
			\label{lem:7}
			Let \( Z \subset U \) be reduced, compact, \( \sp^{n-1}(Z) < \i \), and \( \bar{\O}(p,r) \in \G_c(Z,M,U) \). If \( P_0\in \O(p,r) \) and \( z(p,r)=\sqcup_{j=1}^N \mathit{z_j} \), where \( \mathit{z_j}\subset \O(P_0,r_j) \) for some \( r_j \), \( j=1,\dots,N \), then for each \( \e > 0 \), there exists a reduced, compact set \( \widehat{Z} \subset U \) such that:
			\begin{enumerate}
				\item \( \widehat{Z} \setminus \O(p,r) = Z \setminus \O(p,r) \);
				\item \( \sp^{n-1}(\widehat{Z}(p,r)) \leq (1+ \e) \sum_{j=1}^N\frac{r_j}{n-1}\sp^{n-2}(\mathit{z_j}) \);
				\item \( \widehat{Z}(p,r) \subset h(z(p,r) \cup P_0) \);
				\item There exists an \( (n-1) \)-dimensional polyhedron \( P \subseteq \widehat{Z}(p,r) \) such that \[ \sp^{n-1}(\widehat{Z}(p,r) \setminus P) < \e. \]
			\end{enumerate}
			Furthermore, if \( Z \) spans \( M \), then \( \widehat{Z} \) spans \( M \).
		\end{lem}

		\begin{proof}
			Let \( \widehat{Z}(p,r) \) be the surface determined by \( A=x(p,r) \) in Lemma 7 of \cite{reifenberg}. Then \( \widehat{Z}:=\widehat{Z}(p,r)\cup (Z\setminus \bar{\O}(p,r)) \) satisfies (a)-(d). The proof uses Lemmas \ref{lem:1}, \ref{lem:3}, \ref{lem:4}, \ref{lem:5}, and \ref{lem:6}.  Assume \( Z \) spans \( M \). To see that \( \widehat{Z} \) spans \( M \), we use a modification of the argument in Lemma \ref{lem:cutandcone}. Suppose \( N=\eta(S^1) \) is a simple link of \( M \) and is disjoint from \( \widehat{Z} \).
			
			First project the portion of \( \eta \) inside \[ E_0:=\O(p,r) \setminus \cup_i C_{P_i} \left( \fr\,\O(p,r)\cap \bar{\O}(P_i, \frak{q}_i) \right) \] radially away from \( P_0\in E_0 \) onto \( \fr\, E_0 \), where \( P_i \) and \( \frak{q}_i \) are defined in Reifenberg's proof. Repeating this initial step as the inductive step for each \( X_i\subset \bar{\O}(P_i,\frak{q}_i) \) as defined in Reifenberg's proof, we obtain by induction a new simple link \( \eta' \) of \( M \) with \( \eta'(S^1)\subset \O(p,r)^c \) disjoint from \( \widehat{Z} \), hence also disjoint from \( Z \), yielding a contradiction.
	 	\end{proof}
		
		The following is an adaptation of Lemma 1* of \cite{reifenberg}, p. 27.	 

		\begin{lem}
			\label{lem:1*}
	Suppose  \( X \in \Span^*(M,U) \) and \( \sp^{n-1}(X) \le \frak{m} + \d \) for some \( \d > 0 \). If \( \bar{\O}(p,r)\subset U\setminus M \), then \[ \sp^{n-1}(X(p,r)) \leq \frac{r}{n-1} \sp^{n-2}(x(p,r)) + \d. \]
		\end{lem}

		\begin{proof}
			If \( \bar{\O}(p,r) \notin \G_c(X,M,U) \), we are done. Otherwise, let \( \e>0 \) and let \( \widehat{X} \in \Span^*(M,U) \) be defined as in Lemma \ref{lem:7}. Then \[ \sp^{n-1}(X(p,r)) \leq \sp^{n-1}(\widehat{X}(p,r)) + \d \leq (1+\e) \frac{r}{n-1} \sp^{n-2}(x(p,r)) + \d. \]
		\end{proof}

		The following is an adaptation of Lemma 3* of \cite{reifenberg}, p. 28.   
		\begin{lem}
			\label{lem:3*}
			Suppose \( \{X_k\} \subset \Span^*(M,U) \) is minimizing with \( X_k \to X_0 \). Then \[ \b = \inf \left\{\liminf_{k \to \i} \frac{F_k(p,r)}{\a_{n-1} r^{n-1}} : \bar{\O}(p,r) \in \G(X_0,M,U)\right\}. \]
		\end{lem} 

		\begin{proof}
			Fix \( \bar{\O}(p,r) \in \G(X_0,M,U) \). By Lemma \ref{lem:4}, for almost every \( t\in(0,r) \), the ball \( \bar{\O}(p,t) \in \G_c(X_k,M,U) \) for all \( k > 0 \). Let \( \e_k\to 0 \) such that \( \sp^{n-1}(X_k)\leq \frak{m}+\e_k \). Using the definition of \( \b \) and Lemma \ref{lem:1*} we have
			\begin{align*}
				\a_{n-1} r^{n-1} = \int_0^r \frac{n-1}{t}\a_{n-1} t^{n-1} dt &\leq \int_0^r \frac{n-1}{t\b} \liminf_{k\to\i} \sp^{n-1}(X_k(p,t))dt \\
				&= \int_0^r \frac{n-1}{t \b} \liminf_{k\to\i}(\sp^{n-1}(X_k(p,t))- \e_k) dt\\
				&\leq \frac{1}{\b} \int_0^r \liminf_{k\to\i} \sp^{n-2}(x_k(p,t))dt\\
				&\leq \frac{1}{\b} \liminf_{k\to\i} \int_0^r \sp^{n-2}(x_k(p,t))dt\\
				&=\frac{1}{\b} \liminf_{k\to\i} F_k(p,r).
			\end{align*}
			Thus, by Lemma \ref{lem:4}, \[ \b \leq \inf_{\bar{\O}(p,r) \in \G(X_0,M,U)}\left\{\liminf_{k\to\i} \frac{F_k(p,r)}{\a_{n-1} r^{n-1}} \right\} \leq \inf_{\bar{\O}(p,r) \in \G(X_0,M,U)} \left\{\liminf_{k\to\i} \frac{\sp^{n-1}(X_k(p,r))}{\a_{n-1} r^{n-1}} \right\} = \b. \]
		\end{proof}

		The following lemma is an adaptation of Lemma 4* of \cite{reifenberg}, p. 28.  
		\begin{lem}
			\label{lem:4*}
			Suppose \( \{X_k\} \subset \Span^*(M,U) \) is a Reifenberg regular minimizing sequence with  \( X_k\to X_0 \). 
			If \( 0<r_1 < r_2 \) and \( \bar{\O}(p, r_2) \in \G(X_0,M,U) \), then \[ \liminf_{k_i\to\i} \frac{F_{k_i}(p, r_1)}{\a_{n-1} r_1^{n-1}} \leq \liminf_{k_i\to\i} \frac{F_{k_i}(p, r_2)}{\a_{n-1} r_2^{n-1}} \] and \[ \limsup_{k_i \to \i} \frac{F_{k_i}(p, r_1)}{\a_{n-1} r_1^{n-1}} \leq \limsup_{k_i\to\i} \frac{F_{k_i}(p, r_2)}{\a_{n-1} r_2^{n-1}} \] for every subsequence of integers \( k_i\to \i \).
		\end{lem}

		\begin{proof}
			By Lemma \ref{lem:3*}, for sufficiently large \( k \), \( \b \a_{n-1} r_1^{n-1}/(n-1) \leq F_k(p,r) \mbox{ for all } r \geq r_1. \) Let \( \e_k\to 0 \) such that \( \sp^{n-1}(X_k)\leq \frak{m}+\e_k \). By Lemmas \ref{lem:1*}, \ref{lem:2*}, and \ref{lem:4}, for almost every \( r\in(r_1,r_2) \), \[ \frac{n-1}{r} \left(1 - \frac{\e_k}{\frac{1}{n-1}\b\a_{n-1} r_1^{n-1}} \right) \leq \frac{\frac{d}{dr} F_k(p,r)}{F_k(p,r)}, \] so integrating from \( r_1 \) to \( r_2 \), we get \[ \left(1 - \frac{\e_k}{\frac{1}{n-1}\b\a_{n-1} r_1^{n-1}}\right)\ln \frac{r_2^{n-1}}{r_1^{n-1}} \leq \ln \frac{F_k(p,r_2)}{F_k(p,r_1)}. \] In other words, \[ \exp \left(\frac{-\e_k}{\frac{1}{n-1}\b \a_{n-1} r_1^{n-1}} \ln \frac{r_2^{n-1}}{r_1^{n-1}} \right) \frac{F_k(p,r_1)}{\a_{n-1} r_1^{n-1}}\leq \frac{F_k(p,r_2)}{\a_{n-1} r_2^{n-1}}. \] The result follows since \( \e_k \to 0 \) as \( k \to \i \).
		\end{proof}

		\begin{defn}
			\label{def:uniform}  
			Let \( \e > 0 \). We say that a sequence of compact sets \( \{X_k\} \) is \emph{\textbf{ \( \e \)-uniform with respect to}} \( \bar{\O}(p',r') \in \G(X_0,M,U) \) if
			\begin{equation}
				\label{eq:6*a}
				\b \leq \liminf_{k\to\i} \frac{F_k(p,r)}{\a_{n-1} r^{n-1}} \leq \limsup_{k \to \i} \frac{F_k(p,r)}{\a_{n-1} r^{n-1}} \leq \b + \e
			\end{equation}
			and
			\begin{equation}
				\label{eq:6*b}
				\b \leq \liminf_{k\to\i} \frac{\sp^{n-1}(X_k(p,r))}{\a_{n-1} r^{n-1}} \leq \limsup_{k \to \i} \frac{\sp^{n-1}(X_k(p,r))}{\a_{n-1} r^{n-1}} \leq \b + \e
			\end{equation}
			for every \( \bar{\O}(p,r) \in \G(X_0,M,U) \) with \( \bar{\O}(p,r) \subset \bar{\O}(p',r') \).
		\end{defn}
 
 		\begin{lem}
			\label{lem:uniformexist}
			Suppose \( \{X_k\} \subset \Span^*(M,U) \) is a Reifenberg regular minimizing sequence with \( X_k\to X_0 \). For each \( \e > 0 \), there exists \( \bar{\O}(p',r') \in \G(X_0,M,U) \) such that \( \{X_k\} \) is \( \e \)-uniform with respect to \( \bar{\O}(p',r') \).
		\end{lem}

		\begin{proof} 
			The proof is straightforward and uses Lemma \ref{lem:4*}. Details can be found as (33) and (35) on pp. 35-36 of \cite{reifenberg} where \( p' = P_1 \) and \( r' = r_2 \).
		\end{proof}
		 
		The following lemma is an adaptation of Lemma 6* of \cite{reifenberg}, p. 32. It is a version of what is sometimes called the ``weak geometric lemma,'' but only holds for \( \bar{\O}(p',r') \in \G(X_0,M,U)\) for which \( \{X_k\} \) is \( \e \)-uniform.

		\begin{lem}[Weak geometric lemma]
			\label{lem:6*}
			Suppose \( \{X_k\} \subset \Span^*(M,U) \) is Reifenberg regular minimizing and \( X_k\to X_0 \). If \( \eta > 0 \), there exist \( \e_0 > 0 \) and \( v > 0 \) such that if  \( \{X_k\} \) is  \( \e_0 \)-uniform with respect to  \( \bar{\O}(p',r') \), then for each \( \bar{\O}(p,r) \subset \bar{\O}(p',r') \) with \( p \in X_0 \) there exists a point \( p^* \in X_0 \) and a hyperplane \( \Pi \) through \( p^* \) such that \( \bar{\O}(p^*, vr) \subset \bar{\O}(p, r) \) and \( X_0(p^*, vr) \subset \O(\Pi, \eta vr) \).
		\end{lem}

		\begin{proof}
			The proof is identical to that of Lemma 6* of \cite{reifenberg}, which follows from Lemma 5*, with the exception that Equation (6) on p.30 of \cite{reifenberg} follows from Lemmas \ref{lem:4}, \ref{lem:7} and \ref{lem:cutandcone}. The reader interested in details can find them in \cite{reifenberg}. There are several ambiguities in his proof of Lemma 5*, however, so we provide the reader with some helpful notes: First, the sets \( \ell_n^\theta \) and \( C_\theta \) consist of those points whose joins to \( P \) make an an angle not greater than \( \theta \) with the line passing through the points \( P \) and \( Q \), not just the line segment \( PQ \). Second, the point \( P_0 \) should be on the opposite side of \( P \) from \( Q \), not between the two as stated. Third, Equation (6) is valid only for \( n \) large enough. Fourth, the square roots in Equation (6) arise from the law of cosines applied three different times, and the inequality follows from applying Lemma \ref{lem:7} using center point \( P_0 \).
		\end{proof}

		The next result is based on Lemma 12 of \cite{reifenberg}, p. 22. We shall only apply this to  \(  \bar{\O}(p',r') \) for which \( \{X_k\} \) is  \( \e \)-uniform and minimizing.  

		\begin{lem}[Squashing]
			\label{lem:12}
			Let \( X \in \Span^*(M,U) \) and \( \bar{\O}(q,r)\in \G(X,M,U) \). If \( \Pi \) is a hyperplane containing \( q \) and \( x(q,r) \subset \O(\Pi, \e r) \) for some \( \e<1/2 \), then either \[ \sp^{n-1}(X(q,r)) \ge \a_{n-1} r^{n-1} - 2^{2(n-1)} \frac{\a_{n-1}}{\a_{n-2}}  \e r\,\sp^{n-2}(x(q,r)) \] or there exists a compact set \( X' \) spanning \( M \) equal to \( X \) outside \( \bar{\O}(q,r) \) such that \[ \sp^{n-1}(X'(q,r)) \leq 2^{2(n-1)} \frac{\a_{n-1}}{\a_{n-2}} \e r\,\sp^{n-2}(x(q,r)). \]
		\end{lem}

		\begin{proof}
			Let \( C \) be the radial projection from \( q \) onto \( \fr\,\O(q,r) \) of \( \cup_{x\in x(q,r)} I_x \), where \( I_x \) denotes the line segment joining \( x \) to its orthogonal projection on \( \Pi \). By Lemma \ref{lem:2} and Lemma \ref{lem:6}, \[ \sp^{n-1}(C)\leq 2^{2(n-1)} \frac{\a_{n-1}}{\a_{n-2}} \e r\,\sp^{n-2}(x(q,r)). \] There are two possibilities: Either \( Y:=(X\setminus X(q,r))\cup C \) spans \( M \), or it does not (see Figure \ref{fig:Lemma12}.)  
			\begin{figure}[htbp]              
				\centering
				\includegraphics[height=3in]{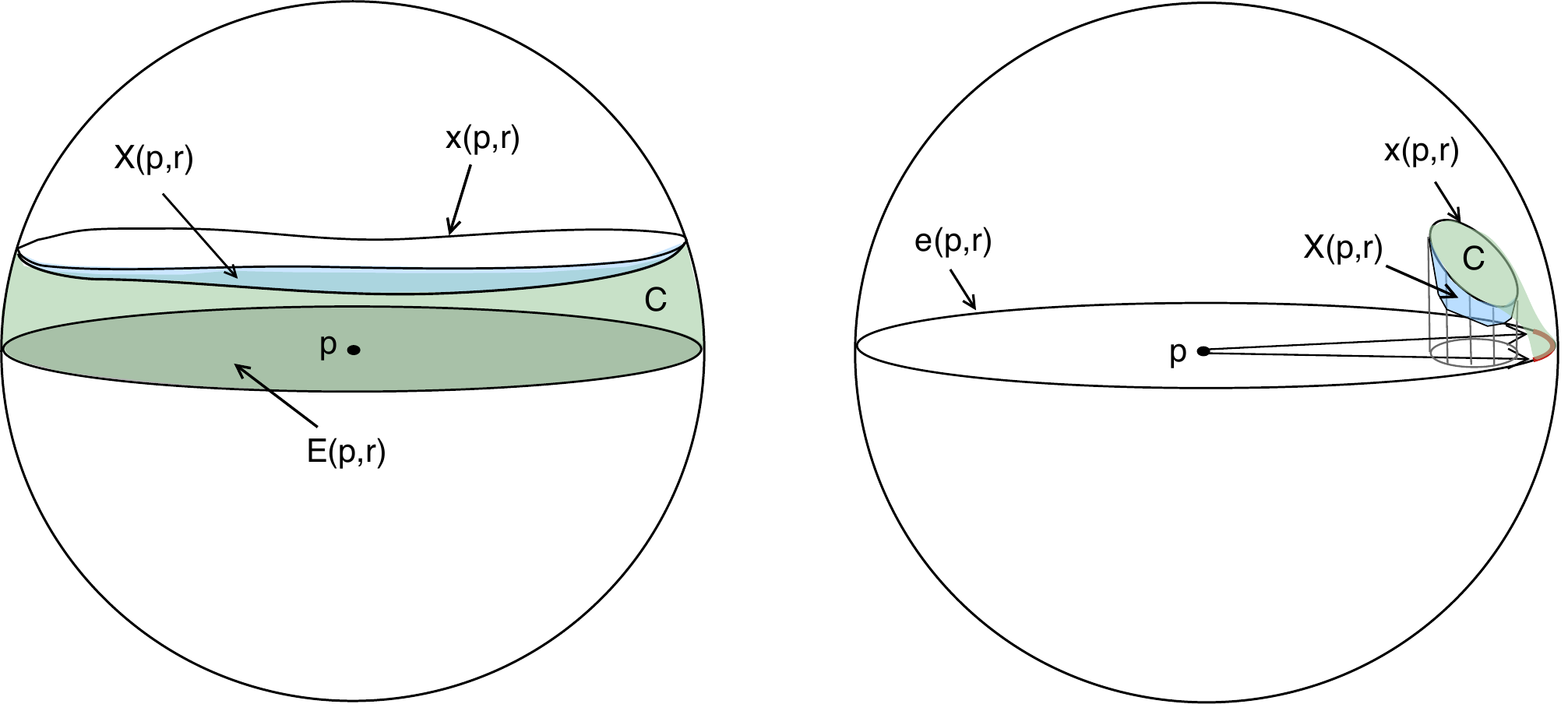}
				\caption{Two cases for Lemma \ref{lem:12}}
				\label{fig:Lemma12}
			\end{figure}
			Suppose \( Y \) does not span \( M \). We show the orthogonal projection of \( C\cup X(q,r) \) onto \( \Pi \) contains \( \Pi\cap \O(q,r) \), and thus by Lemma \ref{lem:1} the first conclusion of the lemma will be satisfied. Let \( N=\eta(S^1) \) be a simple link of \( M \) disjoint from \( Y \). Let us assume without loss of generality that the intersection \( N\cap \fr\,\O(q,r) \) is transverse. The intersection \( N\cap \bar{\O}(q,r) \) consists of a finite collection of arcs \( \{ \eta_i \} \) whose endpoints \( \{ p_i^1,p_i^2\} \) lie on \( \fr\,\O(q,r) \). Since \( C \subset \fr\,\O(q,r) \) we may replace each arc \( \eta_i\) with a pair of line segments, \( p_i^1q \) followed by \( qp_i^2 \) (so now, the curve \( \eta \) will only be piecewise smooth.) However, the intersection \( N\cap \bar{\O}(q,r) \) is now disjoint from \( \cup_{x\in x(q,r)} I_x \) as well.
			
			The hyperplane \( \Pi \) divides \( \bar{\O}(q,r) \) into two hemispheres; north with pole \( N \), and south with pole \( S \). Pick an endpoint \( p_i^j \), \( j=1,2 \), lying in some hemisphere with pole \( P \). Observe that for any \( x \) in the geodesic arc between \( p_i^j \) and \( P \), the line segment between \( x \) and \( q \) will be disjoint from \( \cup_{x\in x(q,r)} I_x \). This gives a regular homotopy of \( N \), and so we may assume without loss of generality that each \( p_i^j\in \{N,S\} \).
			
			Now, suppose the orthogonal projection of \( C\cup X(q,r) \) onto \( \Pi \) does not contain \( \Pi\cap \O(q,r) \). Let \( a\in \Pi\cap \O(q,r) \) be such a point and let \( a_N, a_S \) be the points in the northern and southern hemispheres of \( \fr\,\O(p,r) \), respectively, whose orthogonal projection onto \( \Pi \) is \( a \). By sliding the \( p_i^j \) along the geodesic arcs from \( N \) to \( a_N \) and \( S \) to \( a_S \), we can, as above, assume each \( p_i^j\in \{a_N,a_S\} \). Now replace the pair of segments \( p_i^1q \), \( qp_i^2 \) with the single segment \( p_i^1p_i^2 \) (it may be degenerate.) Smoothing the resulting curve, this gives a simple link \( N \) of \( M \) disjoint from \( X \), a contradiction.
			
		 	Finally, if \( Y \) does span \( M \), then the set \( X':= Y^* \) satisfies the second conclusion of the lemma.
		\end{proof}

		Then next result relies on the concept of uniform concentration which was introduced, although not formally defined, in \cite{reifenberg} (pp. 35-36 (39)-(41),) and later rediscovered in \cite{solimini}.

		\begin{prop}
			\label{prop:beta}
			Suppose \( \{X_k\} \subset \Span^*(M,U) \) is Reifenberg regular minimizing with \( X_k \to X_0 \). Then \[ \beta \geq 1. \]
		\end{prop}

		\begin{proof}
			Let  \( \eta > 0 \). Obtain \( \e_0 \) and \( v \) from Lemma \ref{lem:6*}. By Lemma \ref{lem:uniformexist} for each \( 0 < \e < \e_0 \), there exists \( \bar{\O}(p',r') \in \G(X_0,M,U) \) such that  \( \{X_k\} \) is \( \e \)-uniform with respect to \(  \bar{\O}(p',r') \). According to Lemma \ref{lem:6*} for each \( \bar{\O}(p,r) \subset \bar{\O}(p',r') \) with \( p \in X_0 \) there exists \( p^* \in X_0 \) and a hyperplane \( \Pi \) through \( p^* \) such that \( \bar{\O}(p^*, vr) \subset \bar{\O}(p, r) \) and \( X_0(p^*, vr) \subset \O(\Pi_{p^*}, \eta vr) \).

			Next apply Lemma \ref{lem:12} to each \( \bar{\O}(p^*,v r)\in \G(X,M,U) \) and its hyperplane  \( \Pi \). We can rule out part (b) since \( \{X_k\} \) is minimizing. Since \( \Pi \) contains \( p^* \) and \( x(p^*, vr) \subset \O(\Pi_{p^*}, \eta vr) \), then part (a) implies

			\begin{align}
				\label{eq:a}
				\sp^{n-1}(X(p^*, vr)) \ge \a_{n-1} (vr)^{n-1} - 2^{2(n-1)} \frac{\a_{n-1}}{\a_{n-2}}  \e vr\,\sp^{n-2}(x(p^*, v r)).
			\end{align}

			Since \( \{X_{n_i}\} \) is \( \e \)-uniform with respect to \( X_0(p',r') \), it follows that
			\begin{align}
				\label{eq:vrb}
				\sp^{n-1}X_{n_i}(p^*, vr/2)/ \a_{n-1} (vr/2)^{n-1} > \b/2
			\end{align}
			and
			\begin{align}
				\label{eq:vra}
				F_{n_i}(p^*, vr) < (\b + 2 \e) \a_{n-1}(vr)^{n-1}
			\end{align}
			for sufficiently large \( i \).

			Then we can find \( \rho_i \) with \(vr/2 < \rho_i < vr \) such that
			\begin{align}
				\label{eq:rho}
				\sp^{n-2}(x_{n_i}(p^*, \rho_i)) \le 2(\b + 2\e)\a_{n-1}(vr)^{n-2}.
			\end{align}

			Then \eqref{eq:a} and \eqref{eq:rho} imply

			\begin{align*}
				\sp^{n-1}(X_k(p^*,\rho_i)) &\ge \a_{n-1} \rho_i^{n-1} - 2^{2(n-1)}  \frac{\a_{n-1}}{\a_{n-2}} 2 \eta vr \sp^{n-2}(x(p^*, \rho_i))\\
				&\ge \a_{n-1} \rho_i^{n-1} - 2^{2(n-1)+2}\,\, \frac{\a_{n-1}^2}{\a_{n-2}}\eta(\b + 2\e)(vr)^{n-1}.
 			\end{align*}

			Thus \( \sp^{n-1}(X_k(p^*,v r))/\a_{n-1} (vr)^{n-1} \ge 1 - 2^{2(n-1)+2} \,\, \frac{\a_{n-1}}{\a_{n-2}}\eta(\b + 2\e) \). (This establishes what some call ``\( \e \)-uniform concentration'' for \( \{X_k\} \) with respect to \( \O(p',r') \).) It follows that \( \sp^{n-1}(X_k(p^*,v r))/(\a_{n-1} (vr)^{n-1}) \le \b + \e \) by Definition \ref{def:uniform}. Hence \( \b + \e \ge 1 - 2^{2(n-1)+2} \,\, \frac{\a_{n-1}}{\a_{n-2}}\eta(\b + 2\e) \).

			Since this holds for all \( \e \) and \( \eta \), the result follows.
 		\end{proof}

		\begin{thm}
			\label{thm:lsc2}
			Suppose \( \{X_k\} \subset \Span^*(M,U) \) is Reifenberg regular minimizing and \( X_k \to X_0 \subset \subset U \). If \( W \subset \R^n \) is open, then \[ \sp^{n-1}(X_0 \cap W) \leq \liminf \sp^{n-1}(X_k \cap W). \]
		\end{thm}

		\begin{proof}
			Let \( \d>0 \). Since \( \sp^{n-1}(M)=0 \) it follows from Theorem \ref{thm:finite} and Lemma \ref{lem:1} that we may cover \( \sp^{n-1} \) almost all of \( X_0 \cap W \) by a collection \( \{\bar{\O}(p_i, r_i)\}\subset \G(X_0,M,U) \) of disjoint balls of diameter \( 2r_i<\d \) and contained in \( W \). Thus by Proposition \ref{prop:beta},
			\begin{align*}
				\sp_\d^{n-1}(X_0\cap W)&= \sp_\d^{n-1}(X_0\cap W \cap \cup_i \bar{\O}(p_i, r_i))\\
				&\leq \sum_i \a_{n-1} r_i^{n-1}\\
				&\leq \sum_i \liminf_{k \to \i} \sp^{n-1}(X_k(p_i,r_i))\\
				&\leq \liminf_{k \to \i} \sum_i \sp^{n-1}(X_k(p_i,r_i))\\
				&\leq \liminf_{k \to \i} \sp^{n-1} (X_k \cap W).
			\end{align*}
		\end{proof}

		\begin{proof}[Proof of Theorem \ref{thm:measureequalsm}]
			Let \( S_0\in \T(M,Y,U) \) be \( \A^{n-1} \)-minimizing. By Proposition \ref{prop:lrmin} and Theorem \ref{thm:convexhull}, we may apply Theorem \ref{thm:lsc2}. We set \( W=\R^n \), and this yields, by \eqref{eq:frakm} p. \pageref{eq:frakm}, \[ \sp^{n-1}(\supp (\k_Y S_0)) \leq \frak{m}. \] In particular, \( \supp(\k_Y S_0) \) will have empty interior, and hence by Corollary \ref{cor:key}, \( \supp(\k_Y S_0)=\supp(S_0) \). On the other hand, by Theorem \ref{thm:convexhull} and Proposition \ref{prop:t2complete}\ref{item:support}, \( \supp(S_0)^*\in \Span^*(M,U) \) and hence \( \sp^{n-1}(\supp(S_0))= \frak{m} \). We next prove that \( S_0\in \F(M, Y, U) \).

			Pick a sequence \( \{S_k\} \) as in Proposition \ref{prop:lrmin} and a dyadic subdivision \( \frak{S} \) of \( \R^n \) such that each \( Q\in \frak{S} \) is \( \k_Y S_k \)-compatible\footnote{See Definition \ref{def:compatible}. This is possible by Proposition \ref{prop:t2complete}.} and \( X_k:=\supp(S_k) \)-compatible for all \( k\geq 0 \). Theorem \ref{thm:lsc2} implies \[ \sp^{n-1}(X_0 \cap Q) \leq \liminf \sp^{n-1}(X_k\cap Q) \] for all \( Q\in \frak{S} \). This is in fact an equality, for if \( \sp^{n-1}(X_0 \cap Q) < \liminf \sp^{n-1}(X_k\cap Q) \) for some \( Q\in \frak{S} \), then by Theorem \ref{thm:lsc2} applied to \( W=X_0\cap Q^c \),
			\begin{align*}
				\frak{m}=\sp^{n-1}(X_0) &= \sp^{n-1}(X_0\cap Q) + \sp^{n-1}(X_0\cap Q^c)\\
				&\leq \sp^{n-1}(X_0\cap Q) + \liminf \sp^{n-1}(X_k \cap Q^c)\\
				& < \liminf \sp^{n-1}(X_k \cap Q) + \liminf \sp^{n-1}(X_k \cap Q^c)\\
				&\leq \liminf \sp^{n-1}(X_k) = \frak{m}.
			\end{align*}

			So, by Proposition \ref{prop:maggi}, \( \sp^{n-1}\lfloor_{X_0}(Q)=\liminf \mu_{\k_Y S_k}(Q) = \mu_{\k_Y S_0}(Q) \). Using a Whitney decomposition, this equality extends to all open sets \( W \), and hence by outer regularity of the finite Borel measure \( \mu_{\k_Y S_0} \),
			\begin{equation}
				\label{eq:spmu}
				\sp^{n-1}\lfloor_{X_0} = \mu_{\k_Y S_0}.
			\end{equation}
			It follows from Corollary \ref{cor:key} and Proposition \ref{prop:whitney} \ref{samesupports} that \( X_0 \) is reduced. Thus, \( X_0\in \Span^*(M,U) \), and so by Proposition \ref{prop:t2complete} \ref{item:bdry}-\ref{item:positive} and Theorem \ref{thm:convexhull}, \( S_0\in \F(M, Y, U) \).
		\end{proof}

\section{Regularity}
	\label{sub:almgren_minimality}	
	Let \( \Span^*(M):=\cup_U\Span^*(M,U) \) where the union ranges over all convex open sets \( U \) containing \( M \). Likewise, let \( \F(M):=\cup_{U,Y}\F(M, Y, U) \) where the union ranges over all convex open sets \( U \) containing \( M \) and all vector fields \( Y \) as in \S\ref{sub:constructions_with_curves}. The following result follows from Corollary \ref{cor:existenceagain} and Theorem \ref{thm:measureequalsm}.
	
	\begin{thm}
		\label{thm:big}
		There exists an element \( S_0\in \F(M) \) such that \( \A^{n-1}(S_0)=\sp^{n-1}(\supp(S_0))=\frak{m} \).
	\end{thm}

	Thus, we have proved the bulk of our main theorem: there exists a size-minimizing element of \( \F(M) \). We now prove some regularity results for such size-minimizing elements.

	\begin{cor}
		\label{cor:almin}
		If \( S_0\in \F(M) \) is a size-minimizer and \( \bar{\O}(p,r) \) is disjoint from \( M \), then \( \supp(S_0)\cap \O(p,r) \) is \( (1, \delta) \)-restricted\footnote{See \cite{almgren} and \cite{almbulletin}, although we replace \( \H^{n-1} \) measure with \( \sp^{n-1} \) in Almgren's definition of \( (1,\delta) \)-restricted. Experts assure us that these regularity theorems hold for Hausdorff spherical measure as well as for Hausdorff measure. But just to be safe we know that \( X_0 \) is \( (n-1) \)-rectifiable since it is quasiminimal \cite{davidsemmes}, in which case the two measures are the same. If \( \phi \) is Lipschitz, the set \( \phi(X_0) \) is also \( (n-1) \)-rectifiable, and so \( \sp^{n-1}(\H(X_0))=\sp^{n-1}(\phi(X_0)) \).} with respect to \( \O(p,r)^c \) for all \( \d>0 \).
	\end{cor}

	\begin{proof}
		Let \( X_0=\supp(S_0) \). Suppose \( \phi(X_0) \) is a competitor\footnote{See Definition \ref{def:competitor}.} of \( X_0 \) with respect to \( M \). By Theorem \ref{thm:lipspan} and Lemma \ref{lem:corespan}, \( \phi(X_0)^*\in \Span(M) \), and so putting this together we have \[ \sp^{n-1}(X_0)\leq \sp^{n-1}(\phi(X_0)). \] This gives the result, since every permissible deformation in the definition of restricted sets is of this type.
	\end{proof}
		
	It follows from Corollary \ref{cor:almin} and \cite{almgrenannals} (1.7) that the support of any size minimizer \( S_0\in \F(M) \), which we know exists by Theorem \ref{thm:big}, is almost everywhere a real analytic \( (n-1) \)-dimensional minimal submanifold of \( \R^n \). Soap film regularity for \( n=3 \) follows from \cite{taylor}, and this completes the proof of our main theorem\footnote{As a corollary we can deduce that the solutions for Jordan curves in \( \R^3 \) found in \cite{plateau10} are the same as the solutions in this paper and thus have the structure of a soap film. Suppose \( A_0 \) is a solution found as a limit of dipole surfaces in \( \R^3 \). Since dipole surfaces which span \( M \) are also film chains, then \( A_0 \) is a solution in \( \F(M) \). Conversely, if \( S_0 \) is a solution in \( \F(M) \), then it can be written as a series of embedded dipole surfaces (from Taylor) \( \sum_{i=1}^\i P_{V_i} \widetilde{\s_i} \) where \( V_i \) is a vector field whose component orthogonal to \( \s \) is unit. Let \( B_\e \) be the differential \( 2 \)-chain whose support is the boundary of the \( \e \)-tubular neighborhood of \( M \). Let \( X_\e \) be the unit vector field transverse to \( B_\e \). For sufficiently large \( N \), \( \sum_{i=1}^\i P_{V_i} \widetilde{\s_i} + P_{X_\e} B_\e \) is a dipole surface which spans \( M \). Its limit is \( S_0 \). This proves that \( S_0 \) obtained from Theorem \ref{thm:big} is a solution in \cite{plateau10}.}.

\addcontentsline{toc}{section}{References} 
\bibliography{bibliography.bib}{}
\bibliographystyle{amsalpha}  

\end{document}